\documentclass[psamsfonts]{amsart}
\usepackage{a4}
\usepackage[utf8]{inputenc}
\usepackage[english]{babel}
\usepackage{graphicx}
\usepackage{subfigure}
\usepackage{amsfonts,amstext,amsmath,amsthm,amssymb}
\usepackage[labelfont=bf]{caption}
\usepackage{yfonts}
\usepackage{dsfont}
\usepackage{tikz-cd}
\usepackage{tipa}
\usepackage{endnotes}
\usepackage{xfrac}
\usepackage{tabularx}
\usepackage{abstract}
\usepackage{stmaryrd}
\usepackage{slashed}
\usepackage{mathrsfs}
\usepackage{mathtools}
\usepackage{leftidx,tensor}
\usepackage{lscape}
\usepackage{tikz}
\usetikzlibrary{calc}
\usepackage{blkarray}
\usepackage{adjustbox}
\usepackage{url}
\usepackage{multirow}
\usepackage{pgfplots}
\usepackage[hidelinks]{hyperref}
\usepackage{dirtytalk}
\usepackage{dynkin-diagrams}
\usepackage{empheq}
\usepackage{xcolor}
\usepackage[top=3cm, bottom=3cm, right=3cm, left=3cm]{geometry}

\theoremstyle{plain}

\theoremstyle{definition}{
\newtheorem{defi}{Definition}[section]

\newtheorem*{nota}{Notation}}
\theoremstyle{remark}{
\newtheorem{rem}{Remark}
\newtheorem{ex}{Example}}
\theoremstyle{plain}{
\newtheorem{thm}{Theorem}[section]
\newtheorem{lem}{Lemma}[section]
\newtheorem{cor}{Corollary}[section]
\newtheorem{prop}{Proposition}[section]
\newtheorem{con}{Conjecture}[section]}

\AtEndEnvironment{defi}{\hfill$\lozenge$}
\AtEndEnvironment{ex}{\hfill$\clubsuit$}

\newtheorem{manualtheoreminner}{Theorem}
\newenvironment{manualtheorem}[1]{%
  \renewcommand\themanualtheoreminner{#1}%
  \manualtheoreminner
}{\endmanualtheoreminner}

\newtheorem{manualassumptioninner}{Assumption}
\newenvironment{manualassumption}[1]{%
  \renewcommand\themanualassumptioninner{#1}%
  \manualassumptioninner
}{\endmanualassumptioninner}

\tikzset{
  curarrow/.style={
  rounded corners=8pt,
  execute at begin to={every node/.style={fill=red}},
    to path={-- ([xshift=50pt]\tikztostart.center)
    |- (#1) {}
    -| ([xshift=-50pt]\tikztotarget.center)
    -- (\tikztotarget)}
    }
}

\newsavebox{\pullbackright}
\sbox\pullbackright{%
\begin{tikzpicture}%
\draw (0,0) -- (1ex,0ex);%
\draw (1ex,0ex) -- (1ex,1ex);%
\end{tikzpicture}}

\newsavebox{\pullbackleft}
\sbox\pullbackleft{%
\begin{tikzpicture}%
\draw (0ex,0ex) -- (0ex,1ex);%
\draw (1ex,0ex) -- (0ex,0ex);%
\end{tikzpicture}}

\newcommand{\m}[1]{{\mathrm{#1}}}

\newcommand{\cat}[1]{\textbf{#1}}
\newcommand{\tot}[3]{\underset{#1}{\m{Tot}^{#2}}\left(#3\right)}
\newcommand{\und}[1]{\hspace{#1}\text{and}\hspace{#1}}
\newcommand{\oder}[1]{\hspace{#1}\text{or}\hspace{#1}}

\makeatletter
\DeclareRobustCommand{\chemical}[1]{%
	{\(\m@th
		\edef\resetfontdimens{\noexpand\)%
			\fontdimen16\textfont2=\the\fontdimen16\textfont2
			\fontdimen17\textfont2=\the\fontdimen17\textfont2\relax}%
		\fontdimen16\textfont2=2.7pt \fontdimen17\textfont2=2.7pt
		\mathrm{#1}%
		\resetfontdimens}}
\makeatother

\title{Spin(7)-Orbifold Resolutions}
\author{V. F. Majewski}

\begin{document}

\maketitle
\begin{abstract}
We develop an analytic and geometric framework for resolving compact Spin(7)-orbifolds by smooth torsion-free Spin(7)-manifolds. These orbifolds arise naturally as boundary points in the Gromov--Hausdorff compactification of the moduli space of exceptional
holonomy metrics, and smooth Gromov--Hausdorff resolutions can be viewed as paths from the boundary back into the smooth part of the moduli space.\\

Our construction replaces the singular strata by adiabatic torsion-free asymptotically conically fibred spaces. The local resolution data are encoded by McKay-type correspondences and Chen--Ruan local systems, while the global deformation problem is controlled by the uniform elliptic theory for Dirac-type operators on
orbifold resolutions developed in the author's previous work \cite{majewskiDirac}. In
particular, the obstruction map and the associated isentropicity condition from that theory provide the criterion for whether the local harmonic resolution data glue to global harmonic forms on the smooth
resolution. In this paper, we link the vanishing of the resulting obstruction map to the string cohomology of the orbifold.\\

When this obstruction vanishes, we deform the preglued
Spin(7)-structure to a genuine torsion-free Spin(7)-structure. This extends Joyce's resolution theorem \cite{Joyce1996b,Joyce1999} to the nonflat case and yields new families of compact Spin(7)-manifolds. By dimensional reduction, the same framework recovers and extends the Joyce--Karigiannis theory of $G_2$-orbifold resolutions
\cite{joyce2017new}.
\end{abstract}

\tableofcontents

\section{Introduction}
\label{Introduction}

Special holonomy geometry lies at the intersection of differential geometry, geometric analysis and mathematical physics. By Berger's classification, the possible irreducible non-symmetric holonomy groups of Riemannian manifolds include the classical groups $\m{SU}(n)$ (Calabi--Yau manifolds), $\m{Sp}(n)$ (hyperkähler manifolds) and the exceptional groups
\begin{align*}
G_2\und{1.0cm}\m{Spin}(7).
\end{align*}
The corresponding geometries are Ricci-flat and are described by parallel spinors or equivalently distinguished differential forms satisfying nonlinear first-order equations. In dimension eight, a $\m{Spin}(7)$-structure is given by a positive Cayley four-form
\begin{align*}
    \Phi \in \Gamma(X,\m{Cay}_+(X)) \subset \Omega^4(X),
\end{align*}
which determines a Riemannian metric $g_\Phi$, an orientation and a spin structure. The structure is torsion-free precisely when
\begin{align*}
    \m{d}\Phi = 0,
\end{align*}
or equivalently when $\Phi$ is parallel with respect to the Levi-Civita connection of $g_\Phi$. In that case the holonomy of $g_\Phi$ is contained in $\m{Spin}(7)$, and the metric is Ricci-flat.\\

The local deformation theory of torsion-free $\m{Spin}(7)$-structures is well understood. After quotienting by diffeomorphisms isotopic to the identity, the moduli space of torsion-free $\m{Spin}(7)$-structures on a closed eight-manifold is a smooth finite-dimensional manifold.\cite[Thm. D]{Joyce1996b} Infinitesimal deformations are represented by harmonic four-forms, and the deformation problem becomes elliptic after gauge fixing.\\

The global structure of this moduli space is much less understood. A family of torsion-free $\m{Spin}(7)$-structures may leave every compact subset of the smooth moduli space. Since the associated metrics are Ricci-flat, such degenerations belong to the general compactness theory of Einstein metrics \cite{anderson1990convergence,cheeger1986collapsing,cheeger1990collapsing,cheegernaber2015regularity,jiang2020l2curvatureboundsmanifolds}. Under suitable diameter and non-collapsing assumptions, a sequence of Ricci-flat manifolds admits a Gromov--Hausdorff convergent subsequence. The limit is a compact metric space which is smooth away from a singular set of codimension at least four. On the regular part, the limiting metric is again Ricci-flat, and in the special holonomy setting the limiting special holonomy structure persists away from the singular set.\\

In general, such Ricci-flat limit spaces can have complicated stratified singularities. Orbifolds form the most regular and analytically tractable class among these limits. If the limit is an orbifold, then the singular set is a union of smooth strata, and a neighbourhood of each stratum is modelled on a quotient of a normal bundle by a finite subgroup of the holonomy group. Thus a compact torsion-free $\m{Spin}(7)$-orbifold should be regarded as a natural boundary point of the moduli space of smooth torsion-free $\m{Spin}(7)$-manifolds.\\

This leads to the guiding question of the paper:
\vspace{0.3cm}
\begin{center}
    \textit{\textbf{Which torsion-free Spin(7)-orbifolds can be resolved to smooth torsion-free Spin(7)-manifolds?}}
\end{center}
\vspace{0.3cm}
Equivalently, given a compact $\m{Spin}(7)$-orbifold $(X,\Phi)$, one asks whether there exists a family of smooth $\m{Spin}(7)$-manifolds $(X_t,\Phi_t)$ such that
\begin{align*}
    (X_t,\Phi_t) \xrightarrow[C^\infty GH]{t\to 0} (X,\Phi)
\end{align*}
in the Gromov--Hausdorff topology as $t \to 0$, and such that the convergence is smooth away from the singular strata of $X$. Such a family should be thought of as a path from an orbifold boundary point back into the smooth part of the moduli space.\\

This paper develops a general framework for constructing such paths. The construction is based on the following geometric picture. Let $S \subset X$ be a singular stratum of a compact $\m{Spin}(7)$-orbifold $(X,\Phi)$. The restriction of the metric and of the Cayley form to the normal directions of $S$ defines a conically fibred normal cone
\begin{align*}
    \nu_0 \colon (N_0,\Phi_0) \longrightarrow S.
\end{align*}
The fibres are quotient cones of the form $V/\Gamma$, where $\Gamma$ is a finite subgroup of $\m{Spin}(7)$ preserving the relevant normal representation. A resolution of $(X,\Phi)$ should replace this singular normal cone by a smooth non-compact model
\begin{align*}
    \nu_\zeta \colon  (N_\zeta,\Phi_\zeta) \longrightarrow S,
\end{align*}
which is asymptotic to $(N_0,\Phi_0)$ at infinity. Such spaces are asymptotically conically fibred, or ACF, and they play the role of bubbles along the singular stratum.\\

Once the ACF models have been chosen, one glues large compact subsets of $N_\zeta$ into shrinking tubular neighbourhoods of $S$. This produces a smooth family of manifolds $X_{\zeta;t}$ and a family of preglued $\m{Spin}(7)$-structures
\begin{align*}
    \Phi^{\m{pre}}_{\zeta;t} \in \Gamma(X_{\zeta;t},\m{Cay}_+(X_{\zeta;t})).
\end{align*}
The family $(X_{\zeta;t},\Phi^{\m{pre}}_{\zeta;t})$ converges back to $(X,\Phi)$ in the Gromov--Hausdorff sense. However, the preglued form $\Phi^{\m{pre}}_{\zeta;t}$ is in general not torsion-free. Its torsion is concentrated in the gluing region and must be removed by solving a nonlinear perturbation problem.\\

The analytic core of the paper is the solution of this perturbation problem uniformly in the resolution parameter $t$. The nonlinear equation is
\begin{align*}
\m{d}\Theta(\Phi^{\m{pre}}_{\zeta;t}+\eta)=0,
\end{align*}
where $\Theta$ denotes the projection from a tubular neighbourhood of the bundle of positive Cayley forms. After imposing a gauge condition, the linearisation is governed by the Hodge--de Rham operator
\begin{align*}
D_{\zeta;t}=\m{d}+\m{d}^{*_{\Phi^{\m{pre}}_{\zeta;t}}}.
\end{align*}
Thus the problem reduces to constructing a uniformly bounded right-inverse for $D_{\zeta;t}$ on a family of degenerating manifolds.\\

This is a genuinely singular elliptic problem. As $t \to 0$, the geometry of $X_{\zeta;t}$ decomposes into three different regimes: the orbifold region, the conically fibred normal cone region and the ACF resolution region. Ordinary Hölder or Sobolev spaces do not see these different scales uniformly. We therefore use weighted Hölder spaces adapted to the resolution geometry. These spaces are built so that the Hodge--de Rham operator on $X_{\zeta;t}$ can be compared with model operators on the three limiting pieces
\begin{align*}
D \quad \text{on the orbifold part}, \qquad
\widehat D_0 \quad \text{on the normal cone}, \qquad
\widehat D_\zeta \quad \text{on the ACF resolution}.
\end{align*}
The uniform elliptic theory for such operators is developed in \cite{majewskiDirac}. It gives a linear gluing theorem for Dirac-type operators on orbifold resolutions and provides uniform estimates once a linear obstruction map vanishes.\\

This obstruction has a natural geometric meaning. The harmonic forms on the degenerating family $X_{\zeta;t}$ should, in the adiabatic limit, be described by harmonic forms on the orbifold together with harmonic forms on the local ACF models. However, not every collection of limiting harmonic forms necessarily glues to a genuine harmonic form on $X_{\zeta;t}$. The failure is measured by an obstruction map in the exact linear gluing sequence. We call a resolution isentropic if this obstruction vanishes. In that case, the Hodge theory of the resolution is completely described by its adiabatic limit, and the Hodge--de Rham operator admits a uniformly bounded right-inverse on the complement of its kernel.\\

The term isentropic is meant to reflect the analogy with an adiabatic reversible process in statistical physics; no harmonic information is lost in passing from the smooth resolution to the singular limit. In the present setting, isentropicity is not merely a technical analytic assumption. One of the main points of the paper is that it can be interpreted topologically and geometrically in terms of the string cohomology of the orbifold and of the local resolution models.\\

In codimension four this interpretation becomes especially transparent. The local normal representation is quaternionic, and the isotropy group is a finite subgroup
\begin{align*}
\Gamma \subset \m{Sp}(1).
\end{align*}
The normal cone fibres are Kleinian singularities $\mathbb {H}/\Gamma$. Their smooth ALE hyperkähler resolutions are classified by Kronheimer's construction \cite{kronheimer1989construction,kronheimer} and are governed by the McKay correspondence. The cohomology of the exceptional set is identified with the non-trivial irreducible representations of $\Gamma$, or equivalently with the non-trivial conjugacy classes of $\Gamma$. On the orbifold side, the same data appear as the degree-two Chen--Ruan local system associated with the singular stratum.\\

Thus, for a codimension-four stratum $S$, the resolution parameters form a bundle
\begin{align*}
\mathfrak{P}_\Gamma\coloneqq \wedge^2_+T^*S \otimes \mathfrak{H}^2_\Gamma,
\end{align*}
where $\mathfrak{H}^2_\Gamma$ denotes the local system produced by the McKay correspondence. A section
\begin{align*}
\zeta \in \Gamma(S,\mathfrak{ P}_\Gamma)\cong \Omega^2_+(S,\mathfrak{H}_\Gamma)
\end{align*}
chooses, at each point of the stratum, an ALE hyperkähler resolution of the normal cone fibre. The ACF $\m{Spin}(7)$-resolution of the normal cone is adiabatic torsion-free precisely when $\zeta$ is harmonic as an $\mathfrak{H}^2_\Gamma$-valued self-dual two-form. In this way, the local resolution data are naturally encoded by the Chen--Ruan cohomology group
\begin{align*}
    \m{H}^2(S,\mathfrak{H}^2_\Gamma) \subset \m{H}^4_{\m{CR}}(X).
\end{align*}

This gives a cohomological interpretation of the gluing data. A torsion-free $\m{Spin}(7)$-form on the orbifold determines a class in $\m{H}^4(X)$, while the resolution parameter contributes a Chen--Ruan class \say{supported} on the singular stratum. The resolved torsion-free structure constructed in this paper has cohomology class determined by the pair
\begin{align*}
    [\Phi] \oplus [\zeta] \in H^4_{\m{CR}}(X).
\end{align*}
This is one of the reasons for viewing the Chen--Ruan cohomology of the orbifold as the natural receptacle for the limiting cohomological data of a $\m{Spin}(7)$-orbifold resolution.\\

We now state the two main results of the paper. The first one explains the geometric input needed for the construction: the local resolution data of a codimension-four singular stratum are naturally parametrised by the McKay, or equivalently Chen--Ruan, local system. The second one is the nonlinear existence theorem, which deforms the preglued $\m{Spin}(7)$-structure to a genuine torsion-free one.

\begin{manualtheorem}{\ref{codimensionfourMcKayDuality}/\ref{alladiabaticresolutions}/\ref{isentropicityequalcrepantresolution}}[Orbifold McKay Correspondence and Adiabatic Normal-Cone Resolutions]
Let $(X,\Phi)$ be a compact torsion-free $\m{Spin}(7)$-orbifold and let $S\subset X$ be a singular stratum of codimension four with isotropy group $\Gamma\subset \m{Sp}(1)$. Let
\begin{align*}
\nu_0\colon (N_0,\Phi_0)\longrightarrow (S,g_S)
\end{align*}
denote the corresponding $\m{Spin}(7)$-normal cone. Then the McKay correspondence identifies the degree-two Chen--Ruan local system of the stratum $S$ with the fibrewise second cohomology of the ALE hyperkähler resolutions of $\mathbb{H}/\Gamma$:
\begin{align*}
\mathfrak{H}^2_\Gamma \cong \m{H}^2(N_\zeta/S).
\end{align*}
Consequently, the local resolution parameters form the bundle
\begin{align*}
\mathfrak{P}_\Gamma\coloneqq \wedge^2_+T^*S\otimes \mathfrak{H}^2_\Gamma .
\end{align*}
Let
\begin{align*}
\zeta\in\Gamma(S,\mathfrak{P}^{\m{reg}}_\Gamma)
\subset\Omega^2_+(S,\mathfrak{H}^2_\Gamma)
\end{align*}
be a regular section. Then $\zeta$ determines an ACF $\m{Spin}(7)$-resolution of the normal cone
\begin{align*}
\rho_\zeta\colon (N_\zeta,\Phi_\zeta)\dashrightarrow (N_0,\Phi_0).
\end{align*}
This resolution is of rate $-4$, i.e.
\begin{align*}
(\nabla^{\Phi_0})^k\big((\rho_\zeta)_*\Phi_\zeta-\Phi_0\big)=\mathcal{O}(r^{-4})\und{1.0cm}
(\nabla^{\Phi_0,V})^k\big((\rho_\zeta)_*\Phi_\zeta-\Phi_0\big)=\mathcal{O}(r^{-4-k}).
\end{align*}
Moreover, pulling back the universal dilation morphism yields a diffeomorphism
\begin{align*}
\delta_t\colon (N_{t^2\cdot \zeta},\Phi_{t^2\cdot \zeta})\cong (N_\zeta,\Phi^t_\zeta).
\end{align*}
The ACF $\m{Spin}(7)$-resolution
\begin{align*}
\rho_\zeta\colon (N_\zeta,\Phi_\zeta)\dashrightarrow (N_0,\Phi_0)
\end{align*}
is adiabatic torsion-free if and only if $\zeta$ is harmonic with respect to $g_S$ and the flat McKay local system $\mathfrak{H}^2_\Gamma$. Equivalently, $\zeta$ represents a class
\begin{align*}
[\zeta]\in \m{H}^2(S,\mathfrak{H}^2_\Gamma)\subset \m{H}^4_{\m{CR}}(X,\mathbb{R}).
\end{align*}
In this case the pair $(\Phi,\zeta)$ determines a natural Chen--Ruan cohomology class
\begin{align*}
[\Phi]\oplus[\zeta]\in \m{H}^4_{\m{CR}}(X,\mathbb{R})
\cong
\m{H}^4(X,\mathbb{R})\oplus \m{H}^2(S,\mathfrak{H}^2_\Gamma).
\end{align*}
\end{manualtheorem}

Thus, in codimension four, the resolution parameters are not auxiliary choices but form the stringy extension of the cohomology class of the orbifold $\m{Spin}(7)$-structure. The class $[\Phi]\in \m{H}^4(X,\mathbb{R})$ records the limiting torsion-free $\m{Spin}(7)$-structure on the orbifold, while the class
\begin{align*}
    [\zeta]\in \m{H}^2(S,\mathfrak{H}^2_\Gamma)\subset \m{H}^4_{\m{CR}}(X,\mathbb{R})
\end{align*}
records the choice of harmonic ALE resolution data along the singular stratum. The analytic condition of isentropicity will later identify precisely when this Chen--Ruan cohomological picture agrees with the ordinary cohomology of the smooth resolution. More precisely, for the codimension-four resolutions considered in this paper, isentropicity is equivalent to the equality
\begin{align*}
    \m{H}^\bullet(X_{\zeta;t},\mathbb{R})\cong \m{H}^\bullet_{\m{CR}}(X,\mathbb{R})
\end{align*}
as graded vector spaces. In this case the cohomology class of the resolved torsion-free $\m{Spin}(7)$-structure is uniquely determined by the pair
\begin{align*}
    [\Phi]\oplus[\zeta]\in \m{H}^4_{\m{CR}}(X,\mathbb{R}).
\end{align*}

We can now state the main existence theorem. The construction in Section \ref{Adiabatic Spin(7)-Orbifold Resolutions} produces, from the harmonic local resolution data $\zeta$, a preglued smooth Gromov--Hausdorff resolution
\begin{align*}
    \rho_{\zeta;t}\colon (X_{\zeta;t},\Phi^{\m{pre}}_{\zeta;t})\dashrightarrow (X,\Phi).
\end{align*}
The form $\Phi^{\m{pre}}_{\zeta;t}$ is a genuine $\m{Spin}(7)$-structure, but it is not necessarily torsion-free. The adiabatic torsion-free condition on the local ACF models implies that its torsion tends to zero in the adapted Hölder norms. If, in addition, the resolution is isentropic, the uniform elliptic theory of Section \ref{Hodge Theory on Tame Gromov--Hausdorff-Resolutions of Riemannian Orbifolds} provides a uniformly bounded right-inverse for the Hodge--de Rham operator. The nonlinear deformation problem can then be solved by a fixed-point argument.

\begin{manualtheorem}{\ref{mainthm}}[Existence of Torsion-Free $\m{Spin}(7)$-Orbifold Resolutions]
Let
\begin{align*}
    (X_{\zeta;t},\Phi^{\m{pre}}_{\zeta;t})\dashrightarrow (X,\Phi)
\end{align*}
denote the adiabatic torsion-free $\m{Spin}(7)$-orbifold resolution constructed in Section \ref{Adiabatic Spin(7)-Orbifold Resolutions}. Then the torsion of the preglued $\m{Spin}(7)$-structure satisfies
\begin{align*}
\left|\left|\m{d}\Phi^{\m{pre}}_{\zeta;t}\right|\right|_{\mathfrak{C}^{0,\alpha}_{\beta-1;t}}
\lesssim t^{\vartheta},
\end{align*}
where $\vartheta>0$ satisfies \eqref{vartheta}. If the resolution is isentropic, then for all sufficiently small $t>0$ there exists a torsion-free $\m{Spin}(7)$-structure
\begin{align*}
\Phi_{\zeta;t}\in \Gamma(X_{\zeta;t},\m{Cay}_+(X_{\zeta;t}))
\end{align*}
such that
\begin{align*}
    \rho_{\zeta;t}\colon (X_{\zeta;t},\Phi_{\zeta;t})\dashrightarrow (X,\Phi)
\end{align*}
defines a torsion-free $\m{Spin}(7)$-orbifold resolution. Equivalently,
\begin{align*}
    [\rho_{\zeta;t}\colon (X_{\zeta;t},\Phi_{\zeta;t})\dashrightarrow (X,\Phi)]\in\cat{GHRes}(X,\Phi),
\end{align*}
where $\cat{GHRes}(X,\Phi)$ denotes equivalence classes of smooth Gromov--Hausdorff resolutions.\footnote{Two smooth Gromov--Hausdorff resolutions are identified when they agree for $t\in(0,T)$ for some $0<T$.} Moreover,
\begin{align*}
    \left|\left|\Phi_{\zeta;t}-\Phi^{\m{pre}}_{\zeta;t}\right|\right|_{\mathfrak{D}^{1,\alpha}_{\beta;t}}
\lesssim t^{\vartheta}.
\end{align*}
Here
\begin{align*}
\Phi_{\zeta;t}=\Theta(\Phi^{\m{pre}}_{\zeta;t}+\eta_{\zeta;t}),
\end{align*}
where $\eta_{\zeta;t}$ is the unique solution to the fixpoint problem \eqref{fixpointproblem}.
\end{manualtheorem}

The theorem gives a precise meaning to the statement that the orbifold $(X,\Phi)$ lies on the boundary of the smooth $\m{Spin}(7)$-moduli space. The family $(X_{\zeta;t},\Phi_{\zeta;t})$ consists of smooth torsion-free $\m{Spin}(7)$-structures for $t>0$, converges to $(X,\Phi)$ in the smooth Gromov--Hausdorff sense as $t\to 0$, and remains uniformly controlled by the preglued model. The estimate
\begin{align*}
    \left|\left|\Phi_{\zeta;t}-\Phi^{\m{pre}}_{\zeta;t}\right|\right|_{\mathfrak{D}^{1,\alpha}_{\beta;t}}
\lesssim t^{\vartheta}
\end{align*}
shows that the genuine torsion-free structure is asymptotic to the explicitly constructed adiabatic resolution. In particular, the geometric information carried by the local McKay parameter $\zeta$ survives the nonlinear deformation and determines the exceptional cohomology classes of the resolved $\m{Spin}(7)$-manifold.\\

This theorem extends Joyce's original resolution method \cite{Joyce1996b,Joyce1999} for compact $\m{Spin}(7)$-orbifolds to a broader setting. Joyce's construction resolved flat orbifold singularities by inserting ALE spaces and then applying a nonlinear perturbation theorem for $\m{Spin}(7)$-structures with small torsion. The present paper replaces the flat local picture by a fibred one: the singularities may occur along positive-dimensional strata, and the local models are ACF spaces fibred over the strata. The analytic input is correspondingly stronger, requiring uniform elliptic theory on degenerating orbifold resolutions rather than elliptic estimates on a fixed smooth manifold.\\

The framework also gives a $\m{Spin}(7)$-analogue of the Joyce--Karigiannis resolution theory \cite{joyce2017new} for $G_2$-orbifolds with codimension-four singularities. In the $G_2$-case, the resolution data are governed by ALE hyperkähler geometry over associative strata. In the $\m{Spin}(7)$-case, codimension-four singular strata are four-dimensional and the resolution parameters are self-dual two-forms with values in the McKay local system. The resulting picture is closely related, but the deformation theory is governed by the $\m{Spin}(7)$ Hodge--de Rham operator and the geometry of Cayley four-forms.\\

The paper is organised as follows.\\

In Section \ref{Spin7-Structures on Orbifolds} we recall the basic linear algebra and geometry of $\m{Spin}(7)$-structures on orbifolds. We review Cayley forms, the induced metric and the decomposition of differential forms into irreducible $\m{Spin}(7)$-representations. We also discuss the deformation problem for $\m{Spin}(7)$-structures with small torsion and recall how, after gauge fixing, the Hodge--de Rham operator appears as the relevant linearised operator.\\

Section \ref{The Moduli Space of Torsion-Free Spin(7)-Structures and Smooth Gromov--Hausdorff Resolutions} introduces smooth Gromov--Hausdorff resolutions of Riemannian orbifolds. We explain how orbifold singularities are modelled by conically fibred normal cone bundles and how ACF spaces serve as local resolution models. We define interpolating resolutions and the notion of tameness, which expresses that a given resolution is asymptotic to an explicit gluing model in weighted $C^k$-norms.\\

Section \ref{Hodge Theory on Tame Gromov--Hausdorff-Resolutions of Riemannian Orbifolds} develops the Hodge theory of tame smooth Gromov--Hausdorff resolutions. Using the analytic results of \cite{majewskiDirac}, we compare the Hodge--de Rham operator on the smooth resolution with the model operators on the orbifold, normal cone and ACF pieces. This leads to a linear gluing sequence for kernels and cokernels, a criterion for isentropicity and the construction of uniformly bounded right-inverses.\\

Section \ref{Adiabatic Spin(7)-Orbifold Resolutions} constructs adiabatic $\m{Spin}(7)$-resolutions of normal cone bundles. We analyse the $\m{Spin}(7)$-normal cone of a singular stratum and formulate the adiabatic torsion-free condition. In codimension four, we use Kronheimer's ALE spaces and the McKay correspondence to construct universal local resolution families. We then identify the resulting resolution parameters with Chen--Ruan local systems and prove that harmonic sections produce adiabatic torsion-free ACF $\m{Spin}(7)$-resolutions.\\

Section \ref{Existence of torsion-free Spin(7) Structures on Orbifold Resolutions} proves the nonlinear existence theorem. Starting from an adiabatic torsion-free preglued structure, the uniform right-inverse from Section 4 is used to solve the torsion-free equation by a contraction mapping argument. This produces genuine torsion-free $\m{Spin}(7)$-structures on the smooth resolutions.\\

Finally, Section \ref{Examples of Compact Spin(7)-Manifolds} gives examples. These include generalised Kummer-type constructions and resolutions of orbifolds obtained from products of lower-dimensional special holonomy spaces. These examples illustrate how the abstract resolution theory produces new compact $\m{Spin}(7)$-manifolds and how the resolution parameters are reflected in the topology of the resulting spaces.

\section*{Acknowledgement}
I am particularly grateful to my PhD supervisor Thomas Walpuski who suggested this problem to me. Additionally, I extend my gratitude to Gorapada Bera, Dominik Gutwein,  Jacek Rzemieniecki, Fabian Lehmann, Luc\'ia Mat\'in Merch\'an, Thorsten Hertl, Daniel Platt and Johannes Nordström for their valuable discussions, constructive tips and helpful insights. This work was funded by the Deutsche Forschungsgemeinschaft (DFG, German Research Foundation) under Germany's Excellence Strategy – The Berlin Mathematics Research Center MATH+ (EXC-2046/1, project ID: 390685689).

\section{Spin(7)-Structures on Orbifolds}
\label{Spin7-Structures on Orbifolds}

We will begin with a brief introduction into octonionic linear algebra and geometry of spaces (manifolds, orbifolds) with holonomy $\m{Spin}(7)$. As the construction of manifolds with holonomy $\m{Spin}(7)$ by resolving orbifolds is the main goal of this paper, these definitions will be essential. For those readers with a background in the geometry of special holonomy manifolds it is still recommended to read the next two sections as notations and conventions may differ from other authors. More detailed versions of the following discussion can be found in the work of Joyce \cite[Sec. 10]{joyce2000compact}, Salamon and Walpuski in \cite[Sec. 5-7, Sec. 9]{salamon2010notes}.

\subsection{Octonionic Linear Algebra and Spin(7)}
\label{Octonionic Linear Algebra and Spin(7)}

Let $(W,g_0,\m{vol}_0)$ denote an eight-dimensional oriented Euclidean vector space and let $\m{SO}(W,g_0)\cong\m{SO}(8)$ denote the subgroup of automorphisms preserving both $g_0$ and the orientation $\m{vol}_0$ and let $\m{Spin}(W,g_0)$ denote its universal cover. The spin group $\m{Spin}(W,g_0)$ has three real eight dimensional representations: the vector $W$ representation and two chiral spinor representations 
$\mathbb{S}^+_{g_0}$ and $\mathbb{S}^-_{g_0}$. The space $\mathbb{S}^+_{g_0}$ and $\mathbb{S}^-_{g_0}$ can be endowed with a $\m{Spin}(W,g_0)$-invariant inner products $h_0$ such that the Clifford multiplication 
\begin{align*}
    \m{cl}_{g_0}\colon W\rightarrow\m{End}(\mathbb{S}^+_{g_0},\mathbb{S}^-_{g_0})\oplus \m{End}(\mathbb{S}^-_{g_0},\mathbb{S}^+_{g_0})
\end{align*}
is skew-adjoint. By fixing a unit length spinor $\psi_{0}\in \mathbb{S}^+_{g_0}$ we can identify $ W^*\cong \mathbb{S}^-_{g_0}$ as $\m{Stab}(\psi_{0})\subset \m{Spin}(W,g_0)$-representations. Moreover, the action of $\m{Stab}(\psi_0)\subset \m{Spin}(W,g_0)$ on $\psi^\perp_{g_0}$ factors through $\m{SO}(\psi^\perp_{0},h_0)\cong \m{SO}(7)$ and we can identify it with the universal cover of the latter and hence 
\begin{align*}
    \m{Stab}(\psi_0)\cong \m{Spin}(7)\subset \m{Spin}(8).
\end{align*}

\begin{figure}[h!]
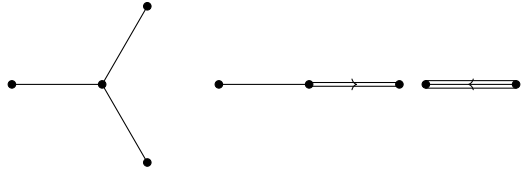

\centering
\dynkin[scale=1, edge length=1.2cm]{D}{4}
\dynkin[scale=1, edge length=1.2cm]{B}{3}
\dynkin[scale=1, edge length=1.2cm]{G}{2}
\caption{The Dynkin diagram of $\mathfrak{spin}(8)$, folded to the one of $\mathfrak{spin}(7)$ by $\mathbb{S}^-\cong W$ or folded to $\mathfrak{g}_2$ by $\mathbb{S}^+\cong \mathbb{S}^-\cong W$.}
\end{figure}

Using the Clifford multiplication by elements in $\wedge^\bullet W$ we can express 
\begin{align}
\label{Cayleyformspinor}
    h_0(\psi_0,\m{cl}_{g_0}(\bullet)\psi)=1+\Phi_0+\m{vol}_{0}
\end{align}
where $\Phi_0\in \wedge^4W^*$ is the so called \textbf{Cayley form} of $\psi_0$. The stabilizer of the right-hand side \eqref{Cayleyformspinor} under the natural action of $\m{GL}(W)$ on the exterior algebra $\wedge^\bullet W^*$ is isomorphic to $\m{Spin}(7)$ and $\Phi_0$ uniquely determines the inner product $g_0$, the orientation $\m{vol}_0$ and the spinor $\psi_0$ up to a sign.\\

The \textbf{space of (positive) Cayley forms} on $W$ can be naturally identified with the $\m{GL}_+(W)$-orbit 
\begin{align*}
    Cay_+\coloneqq \m{GL}_+(W).\Phi_0\cong \m{GL}_+(8)/\m{Spin}(7)\subset \wedge^4 W^*
\end{align*}
inside the space of four forms and naturally projects onto the space of inner products 
\begin{align*}
    g_\bullet \colon Cay_+\twoheadrightarrow \m{Met}_+\coloneqq \m{GL}_+(W).g_0\cong \m{GL}_+(8)/\m{SO}(8).
\end{align*}
Given $\Phi\in Cay_+$ the exterior algebra decomposes into irreducible $\m{Stab}(\Phi)$-representations such that 
\begin{align*}
    \wedge^2W^*\cong& \wedge^2_{\Phi,21}\oplus \wedge^2_{\Phi,7}\cong \mathfrak{stab}(\Phi)\oplus \psi^{\perp}\\
    \wedge^3W^*\cong& \wedge^3_{\Phi,8}\oplus \wedge^3_{\Phi,48}\cong W^*\oplus \ker(\omega\mapsto \Phi\wedge \omega)\\
    \wedge^4 W^*\cong &\wedge^4_{\Phi,1}\oplus \wedge^{4}_{\Phi,7}\oplus \wedge^4_{\Phi,27}\oplus \wedge^4_{\Phi,35}\cong \mathbb{R}\Phi\oplus \psi^\perp\oplus N_\Phi Cay_+\oplus \wedge^4_-
\end{align*}
and 
\begin{align*}
    T_\Phi Cay_+\cong \wedge^4_{\Phi,1}\oplus \wedge^4_{\Phi,7}\oplus  \wedge^4_{\Phi,35}.
\end{align*}

In contrast to the orbit of positive three-forms on $\mathbb{R}^7$, i.e. the space of associative calibrations, the space of positive Cayley calibrations is not an open $\mathrm{GL}_+(8,\mathbb{R})$-orbit in the space of four-forms, but rather a submanifold of codimension 27. Consequently, generic small perturbations of a Cayley calibration in the space of four-forms do not remain within the space $Cay_+$. To compensate for this lack of openness, we construct a $\mathrm{GL}_+(W)$-equivariant tubular neighbourhood, which allows us to project a generically perturbed Cayley calibration back onto the space of positive Cayley forms.\\

Let $j\colon Cay_+\hookrightarrow \wedge^4W^*$ be the inclusion and 
\begin{equation*}
    \begin{tikzcd}
            TCay_+\arrow[r,hook]&j^*T\wedge^4W^*\arrow[r,two heads]&NCay_+
    \end{tikzcd}
\end{equation*}
be the induced exact sequence. Since both $TCay_+$ and $j^*T\wedge^4W^*$ can be equipped with a metric, by 
\begin{align*}
    G(v,w)=g_\Phi(v,w)\hspace{1.0cm}\text{for all }v,w\in\wedge^4W^*
\end{align*}
we obtain an orthogonal splitting $j^*T\wedge^4W^*\cong TCay_+\oplus NCay_+$. Let us define the set $Tub_\rho(Cay_+)\subset j^*T\wedge^4W^*=\{v\in NCay_+ \colon\left|\left|v\right|\right|<\rho\}$. Using parallel-transport by translation we can define 
\begin{align*}
    \m{exp}: Tub_\rho(Cay_+)\hookrightarrow\wedge^4W^*.
\end{align*}
For small enough $\rho$ this map is injective. We denote the projection map by 
\begin{align*}
    \Theta:Tub_\rho(Cay_+)\rightarrow Cay_+.
\end{align*}

\begin{lem}\cite[Prop. 10.5.9]{joyce2000compact}
\label{C0Festimateslem}
Both $Tub_\rho(Cay_+)$ and $Cay_+$ are $\m{GL}_+(W)$-spaces and the map $\Theta$ is $\m{GL}_+(W)$-equivariant. Given a $\Phi\in Cay_+$ and $\Phi+\eta\in\m{Tub}_\rho(Cay_+)$, we can write 
\begin{align}
\label{C0Festimates}
    \Theta(\Phi+\eta)=&\Phi+\pi_{\Phi,\top}(\eta)+Q_{\Phi}(\eta).
\end{align}
Here linear part is given by 
\begin{align*}
    \pi_{\Phi,\top}(\eta)=\pi_{\Phi,1}(\eta)+\pi_{\Phi,7}(\eta)+\pi_{\Phi,35}(\eta)
\end{align*}
with respect to $\Phi$ and $Q_\Phi(\eta)$ satisfies 
\begin{align*}
    |Q_\Phi(\eta)-Q_\Phi(\eta')|_{g_\Phi}\lesssim&|\eta-\eta'|_{g_\Phi}\left(|\eta|_{g_\Phi}+|\eta'|_{g_\Phi}\right).
\end{align*}
\end{lem}

\subsection{Finite Symmetries of Cayley Forms}
\label{Finite Symmetries of Cayley Forms}

Throughout this paper we will encounter various finite subgroups preserving a fixed Cayley structure, all of which fall into the following framework.\\

Let $\Gamma \subset \mathrm{Spin}(7)$ be a finite subgroup acting naturally on $W$, and let $H \coloneqq W^\Gamma$ denote the fixed subspace under this action and $V\coloneqq H^\perp$ it orthogonal complement. Then, by construction,
\begin{align*}
    \Gamma \subset \m{ker} \left( \mathrm{Spin}(7)\cap \left(\m{SO}(V)\times \m{SO}(H) \right) \twoheadrightarrow \mathrm{SO}(H) \right).
\end{align*}
This condition is very restrictive, and hence such subgroups are necessarily of one of the following types:

\begin{itemize}
\item[(1)] $\m{dim} (H) = 8$, $\Gamma = {1}$ and $W/\Gamma \cong W$.
\item[(2)] $\m{dim} (H) = 4$, $\Gamma$ is conjugate to a finite subgroup of $\mathrm{Sp}(1)$, and $W/\Gamma \cong \mathbb{R}^4 \times (\mathbb{H}/\Gamma)$.
\item[(3)] $\m{dim} (H) = 2$, $\Gamma$ is conjugate to a finite subgroup of $\mathrm{SU}(3)$, and $W/\Gamma \cong \mathbb{R}^2 \times (\mathbb{C}^3/\Gamma)$.
\item[(4)] $\m{dim} (H) = 1$, $\Gamma$ is conjugate to a finite subgroup of $G_2$, and $W/\Gamma \cong \mathbb{R} \times (\mathrm{Im},\mathbb{O}/\Gamma)$.
\item[(5)] $\m{dim} (H) = 0$ and $\Gamma$ is a finite subgroup of one of the nested groups $\m{Sp}(2)\subset\m{SU}(4)\subset\m{Spin}(7)$
\end{itemize}

The $\Gamma$-invariance of $\Phi$ forces a type decomposition with respect to the splitting $W\cong H\oplus V$.\\ 

\textbf{Case (2): }Both $H$ and $V$ are four dimensional real vector spaces. A $\Gamma$-invariant Cayley form $\Phi$ decomposes into 
\begin{align*}
    \Phi=\Phi^{4,0}+\Phi^{2,2}+\Phi^{0,4}\in\wedge^4 H^*\oplus\wedge^2 H^*\otimes\wedge^2V^*\oplus\wedge^4V^*
\end{align*}
The tensor $\Phi^{2,2}\in\wedge^2H^*\otimes\wedge^2 V^*$ is bi-self-dual with respect to volume forms $\Phi^{4,0}\in\wedge^4 H^*$ and $\Phi^{0,4}\in\wedge^4 V^*$, and an isomorphism $\phi \colon\wedge^4 H^*\rightarrow\wedge^4 H$. The map 
\begin{align*}
    \Phi^{2,2} \colon\wedge^2H\rightarrow\wedge^2 V^*
\end{align*}
defines an isomorphism of maximal positive subspaces. Three out of four tensors determine the fourth and we will write
\begin{align*}
    \Phi=\m{vol}_H-\m{tr}_+(\underline{\omega}_H\wedge\underline{\omega}_V)+\m{vol}_V.
\end{align*}
Here, $\m{tr}_+(\underline{\omega}_H\wedge\underline{\omega}_V)$ denotes the trace of the identification $\Phi\colon \wedge^2_+ H\rightarrow \wedge^2_+ V^*$. We view $\underline{\omega}_H\in \wedge^2_+ H^*\otimes\m{Im}(\mathbb{H})$ and $\underline{\omega}_V\in \wedge^2_+ V^*\otimes\m{Im}(\mathbb{H})$ as $\m{Sp}(1)$-structures on $H$ and $V$ respectively.\\

\textbf{Case (3): } The space $H$ is two dimensional and $V$ is a six dimensional real vector spaces. In this case a $\Gamma$-invariant Cayley form $\Phi$ decomposes into 
\begin{align*}
    \Phi=\Phi^{2,2}+\Phi^{1,3}+\Phi^{0,4}\in\wedge^2 H^*\otimes\wedge^2V^*\oplus\wedge^1 H^*\otimes\wedge^3V^*\oplus\wedge^4V^*
\end{align*}
Such a $\Gamma$-invariant Cayley form $\Phi$ is determined by an $z\in\wedge^{1,0}_{\mathbb{C}}H^* $ and a $\m{SU}(3)$-structure $(\omega,\theta)\in \wedge^2 V^*\oplus \wedge^{3,0}_{\mathbb{C}}V^*$ on $V$ and we write 
\begin{align*}
    \Phi=&\tfrac{1}{2}(z\wedge\theta-\overline{z}\wedge\overline{\theta})+z\wedge\overline{z}\wedge\omega+\tfrac{1}{2}\omega\wedge\omega\\
    =&e^0\wedge\m{Re}(\theta)-e^1\wedge\m{Im}(\theta)+e^0\wedge e^1\wedge\omega+\tfrac{1}{2}\omega\wedge\omega.
\end{align*}

\textbf{Case (4): } Let $H$ be a one dimensional and $V$ a seven dimensional real vector space. A $\Gamma$-invariant Cayley form $\Phi$ decomposes into 
\begin{align*}
    \Phi=\Phi^{1,3}+\Phi^{0,4}\in\wedge^1 H^*\otimes\wedge^3V^*\oplus\wedge^4 V^*
\end{align*}
Such a $\Gamma$-invariant Cayley form $\Phi$ is determined by a linear map
\begin{align*}
    \Phi^{1,3}:H\rightarrow \wedge^3 V^*
\end{align*}
such that $\Phi^{1,3}(e)=\varphi$ defines $G_2$ structures for $e\neq0$. Further we will write 
\begin{align*}
    \Phi=e\wedge\varphi+*_{\varphi}\varphi.
\end{align*}

\subsection{Spin(7)-Structures}
\label{Spin(7)-Structures}

The following section collects the necessary background on $\m{Spin}(7)$-structure on manifolds and more generally orbifolds.\\  

Let $X$ be an eight dimensional, real, connected, oriented and spinnable orbifold and $Fr_+(X)$ its $\m{GL}_+(8)$-frame bundle. The \textbf{bundle of Spin(7)-structures} or positive Cayley forms is given by 
\begin{align*}
    Cay_+(X)\coloneqq Fr_+(X)/\m{Spin}(7)\subset \wedge^4 T^* X
\end{align*}
A section $\Phi\in \Gamma(X,Cay_+(X))\subset \Omega^4(X)$ naturally defines a $\m{Spin}(7)\hookrightarrow \m{GL}_+(8)$-reduction 
\begin{align*}
    \m{Fr}(X,\Phi)\coloneqq \Phi^*Fr_+(X)\hookrightarrow Fr_+(X).
\end{align*}
The projection map 
\begin{align*}
    g_\bullet\colon Cay_+(X)\twoheadrightarrow \m{Met}(X)
\end{align*}
assigns to a Cayley form $\Phi\in\Gamma(X,Cay_+(X))$ a corresponding Riemannian metric $g_\Phi$ and hence $\m{Spin}(7)\hookrightarrow \m{SO}(8)$-reduction 
\begin{align*}
    \m{Fr}(X,\Phi)\hookrightarrow Fr_+(X,g_\Phi).
\end{align*}
The \textbf{torsion} of a $\m{Spin}(7)$-structure $\Phi$ is given by 
\begin{align*}
    T_\Phi\coloneqq \nabla^{g_\Phi}\Phi
\end{align*} 
and naturally forms an obstruction to $\m{Hol}_{g_\Phi}\subset \m{Spin}(7)$. We will call a $\m{Spin}(7)$-structure $\Phi$ \textbf{torsion-free} if $T_\Phi=0$.

\begin{rem}
    By the holonomy principle, the torsion-freeness $\Phi$ corresponds to the torsion-freeness of the reduction $\m{Fr}(X,\Phi)\hookrightarrow Fr_+(X,g_+)$, i.e. the Levi-Civita connection restricts to $\m{Fr}(X,\Phi)$.
\end{rem}

The following result by \cite{Fernandez1986} proves that the torsion-freeness of a $\m{Spin}(7)$-structure is equivalent to the Cayley form being closed. 

\begin{thm}\cite[Thm. 5.3]{Fernandez1986}\cite[Thm. 2]{bryant1987}
\label{torsionfreeequalclosedness}
    A $\m{Spin}(7)$-structure is torsion-free (1-flat), i.e. $T_\Phi=0$ if and only if the Cayley form is closed, i.e. $\m{d}\Phi=0$.
\end{thm}

Recall that $\Phi\in \Gamma(X,Cay_+(X))$ implies that $\Phi\in \Omega^4_{\Phi,1}(X)$ is a self-dual four form with respect to the metric $g_\Phi$. In the spirit of this paper it is of more geometric meaning to think of a torsion-free $\m{Spin}(7)$-structure as a harmonic four form
\begin{align}
\label{DPhiPhi}
    D_\Phi\Phi=0
\end{align}
with respect to its Hodge--de Rham operator $D_\Phi=\m{d}+\m{d}^{*_\Phi}$. 

\begin{nota}
    In this paper we will refer to an orbifold $(X,\Phi)$ with a torsion-free $\m{Spin}(7)$ as a $\m{Spin}(7)$-orbifold.
\end{nota}




Torsion-free $\m{Spin}(7)$-structures are one of the exceptional geometries appearing in Berger's classification of irreducible, non-symmetric Riemannian
holonomy groups \cite{berger}. In particular, as for Calabi--Yau, hyperkähler and $G_2$ metrics, it is Ricci-flat. This Ricci-flatness can be seen directly from the spinorial description of $\m{Spin}(7)$-structures.

\begin{prop}\cite[Prop. 10.5.6]{joyce2000compact}
\label{parallelspinor}
Let $\pm\psi\in \Gamma(X,S^+_{g_\Phi}X)$ denote the unit spinor (up to a sign) corresponding to a Cayley form $\Phi$. The $\m{Spin}(7)$-structure $\Phi$ is torsion-free if and only if $\pm\psi$ is parallel. The existence of a parallel non-vanishing spinor implies that $(X,g_\Phi)$ is Ricci-flat. 
\end{prop}

In general, the holonomy of a torsion-free $\m{Spin}(7)$-structure is only contained in $\m{Spin}(7)$, it may be a proper subgroup. On closed manifolds, the possible holonomy reductions are strongly constrained by the
$\hat{A}$-genus.

\begin{prop}\cite[Thm. 10.6.1]{joyce2000compact}
Let $(X,\Phi)$ be a $\m{Spin}(7)$-manifold. The $\hat{A}$-genus of $X$ is given by 
\begin{align*}
    24\hat{A}(X)=-1+b_1(X)-b_2(X)+b_3(X)+b_4^+(X)-2b_4^-(X).
\end{align*}
If $\hat{A}(X)=1$ the holonomy of $(X,\Phi)$ is equal to $\m{Spin}(7)$, if $\hat{A}(X)=2$ the holonomy of $(X,\Phi)$ is equal to $\m{SU}(4)$, if $\hat{A}(X)=3$ the holonomy of $(X,\Phi)$ is equal to $\m{Sp}(2)$ and if $\hat{A}(X)=4$ the holonomy of $(X,\Phi)$ is equal to $\m{Spin}(4)$.
\end{prop}

\begin{cor}
Every compact, connected Riemannian eight manifold with holonomy equal to $\m{Spin}(7)$ is simply connected, i.e. $\pi_1(X)=0$.
\end{cor}

\subsection{The Spin(7)-Deformation Problem}
\label{The Spin(7)-Deformation Problem}

In the following we will discuss how a $\m{Spin}(7)$-structure $\Phi$ with a \say{small} torsion can be deformed to a torsion-free $\m{Spin}(7)$-structure.

The map 
\begin{align*}
    TX\rightarrow \m{End}(\wedge^\bullet T^*X),v\mapsto\left(\omega\mapsto v^{\flat_{g_\Phi}}\wedge\omega-\iota_{v}\omega\right) 
\end{align*}
satisfies the Clifford property and hence, induces a faithfully representation of
\begin{align*}
    \m{cl}_{g_{\Phi}} \colon\m{Cl}(T^*X,g_{\Phi})\rightarrow\mathfrak{so}(\wedge^\bullet T^*X,g_{\Phi})\subset\m{End}(T^*X)
\end{align*}
which is skew-Hermitian with respect to $g_{\Phi}$. There exists an involutive isometry
\begin{align*}
    \epsilon_{g_{\Phi}}=i^{(\bullet(\bullet-1))}*_{g_{\Phi}} \colon\wedge^\bullet T^*X\rightarrow T^*X
\end{align*}
such that $\{\m{cl}_{g_{\Phi}},\epsilon_{g_{\Phi}}\}=0$, and 
\begin{align*}
    \wedge^\bullet T^*X=\wedge^\bullet_{+_{g_{\Phi}}}T^*X\oplus \wedge^\bullet_{-_{g_{\Phi}}}T^*X.
\end{align*}
Hence, the tuple 
\begin{align*}
    (\wedge^\bullet T^*X,\m{cl}_{g_{\Phi}},g_{\Phi},\nabla^{g_{\Phi}},\epsilon_{g_{\Phi}})
\end{align*}
forms a graded Hermitian Dirac bundle. Its Dirac operator 
\begin{align*}
    D_{\Phi}\coloneqq\m{d}+\m{d}^{*_{\Phi}}=\left(\begin{array}{cc}
         &D^-_{\Phi}  \\
         D_{\Phi}^+& 
    \end{array}\right)
\end{align*}
and $D^+_{\Phi}$ is the signature operator of $(X,\Phi)$. Notice that the skew-adjoint \say{opposite} signature operator, an operator that will be important in the deformation of $\m{Spin}(7)$-structures decomposes with respect to $\epsilon_{g_{\Phi}}$ into
\begin{align*}
    D_{\Phi}^{op}\coloneqq\m{d}-\m{d}^{*_{\Phi}}=\left(\begin{array}{cc}
         D^{op;-}_{\Phi}&\\
         & D^{op;+}_{\Phi}
    \end{array}\right)
\end{align*}
and hence,
\begin{align*}
    D_{\Phi}\pi_{\pm,\Phi}=\pi_{\pm,\Phi}D_{\Phi}\und{1.0cm}D_{\Phi}^{op}\pi_{\pm;\Phi}=\pi_{\mp;\Phi}D_{\Phi}^{op}.
\end{align*}
This operator squares to $-\Delta_{g_{\Phi}}$ and anti-commutes with the operator $D_{\Phi}$. Thus, on a closed manifold we deduce that 
\begin{align*}
    \m{ker}_{L^2}(D_{\Phi})=\m{ker}_{L^2}(D_{\Phi}^{op})\und{1.0cm}\m{im}_{L^2}(D_{\Phi})=\m{im}_{L^2}(D_{\Phi}^{op}).
\end{align*}

The idea of Joyce was to use singular perturbation theory to obtain a torsion-free $\m{Spin}(7)$-structure. We will briefly summarize these ideas and state the Joyce-existence result for torsion-free $\m{Spin}(7)$-structures.\\

The $\m{GL}_+(8)$-equivariant tubular neighbourhood 
\begin{align*}
    \Theta:Tub_\rho(Cay_+)\rightarrow Cay_+
\end{align*}
induces an open subset 
\begin{align*}
    Fr_+(X)\times_{\m{GL}_+(8)}Tub_\rho(Cay_+)\subset \wedge^4 T^*X
\end{align*}
of four forms that are projectable onto the bundle of Cayley forms. Let $\Phi$ be a $\m{Spin}(7)$-structure on an eight manifold $X$ with small torsion. The condition of a $\m{Spin}(7)$-structure in a small neighbourhood of ${\Phi}$ to be torsion-free is given by 
\begin{align*}
    0=\m{d}\Theta({\Phi}+\eta)=&\m{d}\Phi+\m{d}\pi_{\Phi,\top}(\eta)+\m{d}Q_{\Phi}(\eta)
\end{align*}

Instead of constructing a $\m{Diff}_0(X)$-orbit full of solutions, it is more convenient to construct a distinguished representative in the $\m{Diff}_0(X)$-orbit. Thus, we need to impose the gauge fixing condition of the form 
\begin{align*}
    \pi_{\Phi,8}\m{d} \pi_{\Phi,\top}(\eta)=0
\end{align*}
which defines a gauge slice for the $\m{Diff}_0(X)$-action on $\Gamma(X,Cay_+(X))$ \cite[Thm. D]{Joyce1996b}. Joyce \cite[p.32]{Joyce1996b} further, proved that this condition is equivalent to 
\begin{align*}
    \eta\in \mathcal{H}^4_{\Phi,1}\oplus \mathcal{H}^4_{\Phi,7}\oplus \Omega^4_{\Phi,-}(X).
\end{align*}
Consequently, if $\pi_1(X)$ and $\hat{A}(X)=1$ we can write 
\begin{align*}
    0=\m{d}{\Phi}+\m{d}\pi_{\Phi,\top}(\eta)+\m{d}Q_{\Phi}(\eta)\und{1.0cm}
    0=\pi_{\Phi,1\oplus 7}\eta.
\end{align*}
Now, $\pi_{\Phi,1\oplus 7}\eta=0$ implies $\eta=-*_{\Phi}\eta $ and hence, by taking the $*_{g_{\Phi}}$ of the first equation we obtain
\begin{align}
\label{torsion-freespin7strcmodgauge1}
    -\m{d}\eta=\m{d}\Phi+\m{d}Q_{\Phi}(\eta)\und{1.0cm}
    -\m{d}^{*_{\Phi}}\eta=*_{\Phi}\m{d}\Phi+*_{\Phi}\m{d}Q_{\Phi}(\eta).
\end{align}
Using these operators, we are able to express \eqref{torsion-freespin7strcmodgauge1} as
\begin{empheq}[box=\colorbox{blue!10}]{align}
\label{Diracfixedpoint}
    -D_{\Phi}\eta=&D_{\Phi}\pi_{\Phi,-}\left\{Q_{\Phi}(\eta)\right\}+D^{op}_{\Phi}\left(\Phi+\pi_{\Phi,+}\left\{Q_{\Phi}(\eta)\right\}\right).
\end{empheq}

Let $R_{\Phi}$ denote a right-inverse of $D_{\Phi}$ and let $R_{\Phi}\gamma=\eta$. Then $\eta$ is a solution to the above if $\gamma$ is a solution to the fix-point problem
\begin{empheq}[box=\colorbox{blue!10}]{align}
\label{fixpointproblem}
    \gamma=&-D_{\Phi}\pi_{\Phi,-}\left\{Q_{\Phi}(R_{\Phi}\gamma)\right\}-D^{op}_{\Phi}\left(\Phi+\pi_{\Phi,+}\left\{Q_{\Phi}(R_{\Phi}\gamma)\right\}\right).
\end{empheq}
Alternatively, we can restrict ourselves to $\eta\in\m{ker}(D_{\Phi})^\perp=\m{im}(R_{\Phi})$ and solve for 
\begin{align*}
    \eta=&-(D_{\Phi}|_{\m{ker}(D_{\Phi})^\perp})^{-1}\left(D_{\Phi}\pi_{\Phi,-}\left\{Q_{\Phi}(\eta)\right\}+D^{op}_{\Phi}\left(\Phi+\pi_{\Phi,+}\left\{Q_{\Phi}(\eta)\right\}\right)\right)
\end{align*}
as $D_{\Phi}\pi_{\Phi,-}\left\{Q_{\Phi}(\eta)\right\}+D^{op}_{\Phi}\left(\Phi+\pi_{\Phi,+}\left\{Q_{\Phi}(\eta)\right\}\right)\in\m{ker}(D_{\Phi})^\perp$.

\begin{rem}
    Let $X$ be simply connected and $\hat{A}(X)=1$. The moduli space is of dimension $\m{dim}(\mathfrak{Spin}(7)[X])=\hat{A}(X)+b^1(X)+b^4_-(X)=1+b^4_-(X)$. Given a solution $\Phi'=\Theta(\Phi+\eta)$ of \eqref{Diracfixedpoint}, a local neighbourhood of $[\Phi']$ is modelled by a small neighbourhood of the zero in $\mathbb{R}\Phi'\oplus \m{ker}(D_{\Phi})$.
\end{rem}

In order to show that \eqref{Diracfixedpoint} has a solution, we need to show that the right-hand side of \eqref{fixpointproblem} defines a contractive mapping. The following quadratic estimate on $Q_\Phi$ was used by Joyce to prove the existence of a solution.

\begin{lem}\cite[Lem. 5.1.1]{Joyce1996b}
\label{C1Festimates}
Let $X$ be an oriented eight manifold, $\Phi$ a $\m{Spin}(7)$-structure and $g_{\Phi}$ the corresponding Riemannian metric. Let $\eta\in\Omega^4(X)$ and $|\eta|_{g_{\Phi}}<\rho$ such that $\Phi+\eta\in \Gamma(X,Tub_\rho(Cay_+(X)))$ and 
\begin{align}
\label{taylortheta}
    \Theta(\Phi+\eta)=\Phi+\pi_{\Phi,\top}(\eta)+Q_{\Phi}(\eta)
\end{align}
where $Q_{\Phi}\colon B_{\rho,g_{\Phi}}(\wedge^4T^* X)\rightarrow \wedge^4T^* X$ is a smooth map satisfying 
\begin{align*}
    Q_{\Phi}(0)=0.
\end{align*}
For all $\eta,\eta'\in\Omega^4(X)$ with  $|\eta|_{g_{\Phi}},|\eta'|_{g_{\Phi}}<\rho$ 
\begin{align}
\label{thetaestimate}
\begin{split}
    |Q_{\Phi}(\eta_1)-Q_{\Phi}(\eta_2)|_{g_{\Phi}}\lesssim&\hspace{0.5cm}|\eta_1-\eta_2|_{g_{\Phi}}(|\eta_1|_{g_{\Phi}}+|\eta_2|_{g_{\Phi}})
\end{split}
\end{align}
and
\begin{align}
\label{nablathetaestimate}
    |\nabla^{g_\Phi} Q_{\Phi}(\eta_1)-\nabla^{g_\Phi} Q_{\Phi}(\eta_2)|_{g_{\Phi}}\lesssim c_0(\Phi,\eta_1,\eta_2)\cdot|\eta_1-\eta_2|_{g_{\Phi}}+c_1(\Phi,\eta_1,\eta_2)\cdot|\nabla^{g_\Phi} (\eta_1-\eta_2)|_{g_{\Phi}}
\end{align}
where
\begin{align*}
    c_0(\Phi,\eta_1,\eta_2)=&|\m{d}\Phi|_{g_{\Phi}}(|\eta_1|_{g_{\Phi}}+|\eta_2|_{g_{\Phi}})+|\nabla \eta_1|_{g_{\Phi}}+|\nabla^{g_\Phi} \eta_2|_{g_{\Phi}}+|\nabla^{g_\Phi} \eta_1|_{g_{\Phi}}|\eta_2|_{g_{\Phi}}+|\eta_1|_{g_{\Phi}}|\nabla^{g_\Phi}\eta_2|_{g_{\Phi}}\\
    c_1(\Phi,\eta_1,\eta_2)=&|\eta_1|_{g_{\Phi}}+|\eta_2|_{g_{\Phi}}+|\eta_2|_{g_{\Phi}}|\eta_1-\eta_2|_{g_{\Phi}}.
\end{align*}
\end{lem}

In his seminal work \cite{Joyce1996b}, Joyce used this estimate to prove that under certain conditions \eqref{Diracfixedpoint} has a smooth solution, i.e. that in some cases $\m{Spin}(7)$-structures with small torsion can be deformed into torsion-free ones.

\begin{thm}\cite[Thm. A]{Joyce1996b}
Let $X$ be a compact eight-dimensional manifold and $\Phi\in\mathcal{C}ay(X)$. Assume that there exists a smooth four form $\eta$ such that 
\begin{align*}
    \m{d}(\Phi+\eta)=0
\end{align*}
and that the following four conditions hold
\begin{itemize}
    \item $\left|\left|\eta\right|\right|_{L^2}\lesssim t^{9/2}$
    \item $\left|\left|\eta\right|\right|_{L^{10}}\lesssim t$
    \item If ${\chi}\in W^{1,10}\Omega^4_-(X)$ then
    \begin{align*}
        \left|\left|\nabla{\chi}\right|\right|_{L^{10}}\lesssim\left|\left|\m{d}{\chi}\right|\right|_{L^{10}}+t^{-21/5}\cdot\left|\left|{\chi}\right|\right|_{L^2}
    \end{align*}
    \item If ${\chi}\in W^{1,10}\Omega^4_-(X)$ then ${\chi}\in\Omega^{0;4}_-(X)$ and
    \begin{align*}
        \left|\left|{\chi}\right|\right|_{C^0}\lesssim t^{1/5}\cdot\left|\left|\nabla{\chi}\right|\right|_{L^{10}}+t^{-4}\cdot\left|\left|{\chi}\right|\right|_{L^2}
    \end{align*}
\end{itemize}
Then there exists a smooth torsion-free $\m{Spin}(7)$-structure $\Phi$ satisfying
\begin{align*}
\left|\left|\Phi-\Phi\right|\right|_{C^0}\lesssim t^{1/2}. 
\end{align*}
\end{thm}

In the above Theorem, $W^{1,10}\Omega^4_-(X)$ denotes the Sobolev space of anti-self dual four forms that are weakly differentiable and the Sobolev norm is taken with respect to $L^{10}$, while $\Omega^{0;4}_-(X)$ denotes the space of $C^0$ anti-self-dual forms.

\section{The Moduli Space of Torsion-Free Spin(7)-Structures and Smooth Gromov--Hausdorff Resolutions}
\label{The Moduli Space of Torsion-Free Spin(7)-Structures and Smooth Gromov--Hausdorff Resolutions}

\subsection{Compactifications of Moduli Spaces of Spin(7)-Manifolds and Orbifold-Degenerations}
\label{Compactifications of Moduli Spaces of Spin(7)-Manifolds and Orbifold-Degenerations}

At the heart of exceptional holonomy geometry lies the study of moduli spaces of torsion-free $\m{Spin}(7)$-structures on smooth manifolds. These moduli spaces encode the local deformation theory of special holonomy structures, but they also raise global compactness questions. While the local theory describes small deformations of a fixed smooth $\m{Spin}(7)$-manifold, the global theory asks what happens to families of torsion-free $\m{Spin}(7)$-structures which leave every compact subset of the smooth moduli space. Since torsion-free $\m{Spin}(7)$-metrics are Ricci-flat, such degenerations naturally fall within the compactness theory of Einstein metrics.

\begin{defi}
    The \textbf{moduli space of Spin(7)-manifolds} is defined to be 
    \begin{align*}
        \mathfrak{Spin}(7)\coloneqq&\left\{(X,\Phi)\mid \Phi\in\Gamma(X,Cay_+(X)),D_\Phi\Phi=0\right\}/\cat{Diff}_0\\
        =&\left\{(X,\Phi)\mid \Phi\in\Gamma(X,Cay_+(X)),\m{d}\Phi=0\right\}/\cat{Diff}_0.
    \end{align*}
Here $\cat{Diff}_0$ denotes the group of diffeomorphisms isotopic to the identity. 
\end{defi}

A fundamental feature of this moduli space is that the deformation theory of torsion-free $\m{Spin}(7)$-structures is governed by an elliptic and unobstructed deformation complex. Infinitesimal deformations are described by harmonic four-forms, and the linearisation of the torsion-free condition is elliptic after gauge fixing. Moreover, this deformation complex can be naturally extended to an elliptic complex \cite[Thm. 7.2]{goto2002deformations}, which folds into a twisted Dirac operator. This ellipticity implies that the moduli space is finite-dimensional.

\begin{thm}\cite[Thm. D]{Joyce1996b}
    Let $X$ be a closed spinnable eight-manifold. The connected component $\mathfrak{Spin}(7)[X]$ is a smooth manifold of dimension 
    \begin{align*}
        \m{dim}(\mathfrak{Spin}(7)[X])=\hat{A}(X)+b_1(X)+b^4_-(X).
    \end{align*}
    It is endowed with a natural Riemannian metric, called the Weil--Peterson-metric, given by the $L^2$-metric on
    \begin{align*}
        T_{[\Phi]}\mathfrak{Spin}(7)\cong
        \mathcal{H}^4_{1;\Phi}\oplus\mathcal{H}^4_{7;\Phi}\oplus \mathcal{H}^4_{35;\Phi}.
    \end{align*}
\end{thm}

The theorem describes the smooth local structure of $\mathfrak{Spin}(7)$, but it does not address its compactification. A degenerating family of torsion-free $\m{Spin}(7)$-structures should be understood as a compactness problem in $\mathfrak{Spin}(7)$. Assume, for instance, that we are given a family $(X_i,\Phi_i)$ of torsion-free $\m{Spin}(7)$-manifolds, leaving every compact subset of $\mathfrak{Spin}(7)$. Moreover, assume that the associated Ricci-flat metrics satisfy uniform diameter and non-collapsing bounds
\begin{align*}
    \m{diam}(X_i,g_{\Phi_i})\leq D
    \und{1.0cm}
    \m{Vol}(X_i,g_{\Phi_i})\geq v>0.
\end{align*}
Such conditions are linked to finite Weil-Peterson-distance noncompactness issues see \cite{Langlais_2025}. By Gromov-compactness, after passing to a subsequence, the spaces $(X_i,d_{g_{\Phi_i}})$ converge in the Gromov--Hausdorff topology to a compact metric space
\begin{align*}
    (X_i,d_{g_{\Phi_i}})\xrightarrow[GH]{}(X_\infty,d_\infty).
\end{align*}
The regularity theory for non-collapsed Einstein limits, developed by Anderson \cite{anderson1990convergence}, Cheeger--Colding \cite{cheeger1986collapsing}, Cheeger--Naber \cite{cheeger1990collapsing}, and Jiang--Naber\cite{jiang2020l2curvatureboundsmanifolds}, describes the structure of such limits. The regular set $X^{reg}_\infty\subset X_\infty$ is open and carries a smooth torsion-free $\m{Spin}(7)$-structure $\Phi_\infty$, and the convergence is smooth on compact subsets of $X^{reg}$. The singular set
\begin{align*}
    X^{sing}_\infty\coloneqq X_\infty\backslash X^{reg}_\infty
\end{align*}
has codimension at least four. More precisely, the codimension-four stratum is rectifiable and has locally finite Hausdorff measure. At generic points of this top stratum the tangent cone is unique and has the form
\begin{align*}
    \mathbb{R}^{n-4}\times C(\mathbb{S}^3/\Gamma),
\end{align*}
where $\Gamma\subset \m{Sp}(1)$ is a finite subgroup acting freely on $\mathbb{S}^3$. In general, an $\m{Spin}(7)$-limit-space need not be an orbifold; the singular set may have a more complicated stratified structure. Orbifolds should therefore be regarded as the most regular and analytically most tractable singular spaces arising in this compactness theory. In the orbifold case the singular set is a union of smooth strata, and a neighbourhood of a stratum $S$ of codimension $m$ is modelled on a normal cone bundles. The convergence is smooth away from the singular strata, while the failure of compactness is localised in shrinking tubular neighbourhoods of these strata.\\

This localisation is the geometric reason for studying orbifold resolutions. If curvature concentrates near a singular stratum, then after rescaling in the normal directions one expects complete non-compact Ricci-flat bubbles asymptotic to the corresponding quotient cone. In the special holonomy setting these bubbles should carry compatible special holonomy structures. Thus, for a $\m{Spin}(7)$-orbifold limit, the relevant local models are ALE, QALE, or more generally asymptotically conically fibred spaces resolving the $\m{Spin}(7)$ normal cone along the singular stratum.\\

The compactness picture suggests a reverse construction. Instead of starting with a smooth family and passing to its orbifold limit, one may start with a compact $\m{Spin}(7)$-orbifold $(X,\Phi)$ and try to reconstruct a smooth family by inserting suitable non-compact local models into tubular neighbourhoods of the singular strata. The exceptional geometry is then collapsed in the adiabatic limit, and the resulting family converges back to $(X,\Phi)$ in the Gromov--Hausdorff sense. Smooth Gromov--Hausdorff resolutions should be understood in precisely this way: they are controlled paths from an orbifold boundary point back into the smooth part of the moduli space.\\

Joyce's foundational constructions \cite{Joyce1996b,Joyce1999} provide the basic evidence for this perspective. Starting from compact $\m{Spin}(7)$-orbifolds with suitable singularities, Joyce constructed nearby smooth $\m{Spin}(7)$-manifolds by gluing in ALE local models. Thus orbifold points are not merely possible singular limits of smooth special holonomy metrics; in favourable cases they can also be resolved back into the smooth moduli space.\\

This raises the broader question of how the moduli space of $\m{Spin}(7)$-orbifolds, denoted by $\mathfrak{Spin}(7)^{\m{orb}}$, interacts with the smooth moduli space $\mathfrak{Spin}(7)$. There is a tautological inclusion
\begin{align*}
    \mathfrak{Spin}(7)\subset \mathfrak{Spin}(7)^{\m{orb}},
\end{align*}
but the size of the difference between these two spaces is not well understood. In particular, one would like to know which orbifold points in $\mathfrak{Spin}(7)^{\m{orb}}$ arise as Gromov--Hausdorff limits of smooth $\m{Spin}(7)$-manifolds, and which of them can be resolved back to smooth points of $\mathfrak{Spin}(7)$. The present paper addresses the second question by constructing smooth Gromov--Hausdorff resolutions under explicit geometric and analytic hypotheses.\\

From the metric viewpoint, a smooth $\m{Spin}(7)$ resolution of an orbifold singularity produces a path in the Gromov--Hausdorff topology leading from an orbifold boundary point back into the interior of $\mathfrak{Spin}(7)$. In Section \ref{Smooth Gromov--Hausdorff-Resolutions} we introduce smooth Gromov--Hausdorff resolutions, which formalise this idea. The notion of tameness introduced there should be understood as an analytic expression of the compactness heuristic above. After choosing the correct local bubbles, the family has no hidden degeneration away from the prescribed singular strata and is asymptotic to an explicit interpolating model.\\

Thus, one of the central problems in understanding the global structure of $\mathfrak{Spin}(7)$ is to determine to what extent the boundary of the moduli space is described by orbifolds, and when such orbifolds can be resolved. This perspective places $\m{Spin}(7)$-geometry within the broader paradigm of special holonomy moduli problems, where singular degenerations and their resolutions play a decisive role in shaping the global geometry of the moduli space.

\subsection{Smooth Gromov--Hausdorff-Resolutions}
\label{Smooth Gromov--Hausdorff-Resolutions}

In the preceding subsection we explained why orbifold limits form a natural and tractable class of compactness phenomena for degenerating families of torsion-free $\m{Spin}(7)$-structures. We now formalise the reverse operation. Starting with a Riemannian orbifold $(X,g)$, a smooth Gromov--Hausdorff resolution is a family of spaces which resolves the singular locus of $X$ while
converging back to $(X,g)$ as the resolution parameter tends to zero. Thus the exceptional geometry is visible at positive scale, but collapses in the limit.\\

The definition is designed to isolate the geometric features suggested by Einstein compactness. First, away from the singular set the family should not degenerate; after identifying the complements of the exceptional sets with $X^{\mathrm{reg}}$, the metrics converge smoothly. Second, all degeneration should be confined to a neighbourhood of the singular set; the exceptional
region has vanishing volume and collapses to $X^{\mathrm{sing}}$ in the Gromov--Hausdorff limit. This captures precisely the situation in which the limiting orbifold contains all macroscopic geometry, while the missing microscopic geometry is encoded by the local bubbles inserted along the singular strata.
\begin{defi}
Let $(X,g)$ be a Riemannian orbifold. A \textbf{smooth Gromov--Hausdorff resolution} of $(X,g)$ is given by a smooth family of (singular) Riemannian orbifolds $(X_t,g_t)$ and morphisms
\begin{align*}
    \rho_t\colon X_t\dashrightarrow X
\end{align*}
such that $\rho_t$ restricts to an orbifold diffeomorphism $\rho_t^{-1}(X^{reg})\cong X^{reg}$, the exceptional set $\Upsilon_t=\rho_t^{-1}(X^{sing})$ is of codimension $0$ and 
\begin{align*}
    (X_t,g_t)\xrightarrow[GH]{t\to 0}(X,g)\und{1.0cm} (X_t\backslash \Upsilon_t,g_t)\xrightarrow[C^\infty]{t\to 0}(X^{reg},g).
\end{align*}
In particular, $\m{vol}_{g_t}(\Upsilon_t)\xrightarrow[]{t\to 0} 0$. 
\end{defi}
A complete smooth Gromov--Hausdorff-resolution of $(X,g)$ is given by a smooth family of Riemannian manifolds $(X_t,g_t)$ resolving $(X,g)$ at $X^{sing}$. Smooth Gromov--Hausdorff resolutions of Riemannian orbifolds form a category $\widetilde{\cat{GHRes}}(X,g)$. A smooth Gromov--Hausdorff resolution $(X_t,g_t,\rho_t)$ is a $(0,\epsilon)$-family of Riemannian orbifold resolutions. A morphism $(\rho_t,X_t,g_t)\rightsquigarrow(\rho'_t,X'_t,g'_t)$ of Riemannian orbifold resolutions is given by a $(0,\min\{\epsilon,\epsilon'\})$-family of commuting diagrams
\begin{equation*}
    \begin{tikzcd}
	{(X_t,g_t)}\arrow[rd,"{\rho_t}"', dashed] \arrow[rr,"{\phi_t}", dashed]&& {(X'_t,g'_t)}\arrow[dl,"{\rho'_t}", dashed] \\
	 & {(X,g)}
\end{tikzcd}
\end{equation*}
Two resolutions $(\rho_t,X_t,g_t)$ and $(\rho'_t,X_t,g'_t)$ are called equivalent, if the morphisms $\phi_t$ are isometries of Riemannian orbifolds. The category of smooth Gromov--Hausdorff-resolutions of Riemannian is the quotient category 
\begin{align*}
    \cat{GHRes}(X,g)=\widetilde{\cat{GHRes}}(X,g)/\sim.
\end{align*}
With this definition, there is no need to repeatedly shrink the parameter range $(0, T)$ at each step. Some arguments may only apply to a smaller subfamily $(0, T')$, but this is irrelevant for our purposes. When discussing resolutions of Riemannian orbifolds, we are typically not concerned with the maximal interval of existence, but rather with the existence of a resolution at all. The above “germ-like” definition of the resolution category allows us to avoid unnecessary technicalities. Starting from a resolution $(X_t, g_t) \in \widetilde{\cat{GHRes}}(X, g)$, certain analytic steps may require restricting the family to $(0, \epsilon') \subset (0, \epsilon)$; however, in the category $\cat{GHRes}(X, g)$, these are identified as the same object, and no distinction is needed.

\subsubsection{Riemannian Orbifolds as Conically Fibred Singular Spaces}
\label{Riemannian Orbifolds as Conically Fibred Singular Spaces}

In this section, we recall the notion of conically fibred (CF) spaces and conically fibred singular (CFS) spaces as a systematic way to describe the local geometry of Riemannian orbifolds and more general stratified spaces. This viewpoint refines the classical description of quasi-conical neighbourhoods near singular strata by emphasising the fibration structure of each stratum, which plays a central role in resolution constructions.\\

Let in the following $\vartheta\colon (Y,g_{Y})\rightarrow (S,g_S)$ be a Riemannian fibration\footnote{Here we allow the Riemannian structure on the fibres to be singular.} of dimension $m-1+s=n-1$ (here $S$ is an $s$-dimensional Riemannian manifold) with compact fibres. The Riemannian structure on $(Y,g_Y)$ induces an Ehresmann connection on the fibration, i.e. $TY\cong (VY)^\perp \oplus VY\cong H_Y\oplus VY$.

\begin{defi}
A \textbf{conical fibration (CF)} over a link-fibration $(\vartheta\colon (Y,g_{Y})\rightarrow (S,g_S))$ is given by the (singular) warped product
\begin{align*}
    \left(CF(\vartheta)\colon (CF_\vartheta(Y),g_{\m{CF}_\vartheta(Y)})\rightarrow (S,g_S)\right)\coloneqq\left((\vartheta\circ\m{pr}_2\colon \mathbb{R}_{\geq 0}\times Y,\m{d}r^2+g^{2,0}_{Y}+r^2g^{0,2}_{Y})\rightarrow (S,g_S)\right)
\end{align*}
Usually, we will write $g_{CF_\vartheta(Y)}=\vartheta^*g_{S}+g_{\m{CF}_\vartheta(Y);V}$ where $g_{\m{CF}_\vartheta(Y);V}$ denotes the fibrewise conical structure. Given a conical fibration, we will denote by 
\begin{align*}
    r\coloneqq \m{pr}_1\colon CF_\vartheta(Y)\rightarrow \mathbb{R}_{\geq 0}
\end{align*}
the radius function.
\end{defi}

\begin{rem}
    Notice that the geometry of the conical fibration is the one of a singular warped product 
    \begin{align*}
        (\mathbb{R}_{\geq 0}\times Y,\m{d}r^2+g_{Y;r})
    \end{align*}
    of the real half-line and a collapsing fibration 
    \begin{align*}
        \left(Y,g_{Y;r}\right)\rightarrow (S,g_S),
    \end{align*}
    where $g_{Y;r}\coloneqq g^{2,0}_{Y}+r^2g^{0,2}_{Y}$.
\end{rem}

\begin{nota}
    Given $I\subset \mathbb{R}_{\geq 0}$, we will write  
    \begin{align*}
        CF^I_\vartheta(Y)\coloneqq r^{-1}(I)
    \end{align*}
\end{nota}

The group $(\mathbb{R}_{>0},\cdot)$ acts on a conical fibration via the dilatation 
    \begin{align*}
        \delta_t\colon CF_\vartheta(Y)\rightarrow CF_\vartheta(Y),\,(r,y)\mapsto (t\cdot r,y).
    \end{align*}

\begin{nota}
    Given a tensor $\eta$ on $CF_\vartheta(Y)$, we will denote by 
    \begin{align*}
        \eta^t\coloneqq \delta^*_t\eta 
    \end{align*}
    the $\mathbb{R}_{>0}$-family induced by the action.
\end{nota} 

In particular, the CF-metric satisfies 
\begin{align*}
    g^t_{CF_\vartheta(Y)}=g_{CF_\vartheta(Y)H}+t^2\cdot g_{CF_\vartheta(Y);V}.
\end{align*}

\begin{defi}
A singular Riemannian manifold $(X,g)$ is called \textbf{conically fibred singular (CFS)} of rate $\sigma$ with respect to a conical fibration over a link $\vartheta\colon (Y,g_Y)\rightarrow (S,g_S)$ if there exists a compact subset $K_R\subset X$ and a diffeomorphism $\tau\colon CF_\vartheta^{[0,R)}(Y)\cong  K_R $ such that 
\begin{align*}
    (\nabla^{g_{CF_\vartheta(Y)}})^k\left(\tau^*g-g_{CF_\vartheta(Y)}\right)=&\mathcal{O}(r^{\sigma})\\
    (\nabla^{g_{CF_\vartheta(Y);V}})^k\left(\tau^*g-g_{CF_\vartheta(Y)}\right)=&\mathcal{O}(r^{\sigma-k}).
\end{align*}
Usually, we will write 
\begin{align*}
    \tau^*g=g_{CF_\vartheta(Y)}+g_{\tau,\m{hot}}\und{1.0cm}|g_{\tau,\m{hot}}|_{g_{CF_\vartheta(Y)}}=\mathcal{O}(r^\sigma).
\end{align*}
\end{defi}

Let $i\colon (S,g_S)\hookrightarrow (X,g)$ be a connected singular stratum of codimension $m$ whose isotropy group is $\Gamma\subset\m{O}(m)$. We define the normal bundle of $S$ via
\begin{equation*}
    \begin{tikzcd}
        TS\arrow[r,hook]&i^*TX\arrow[r, two heads]&NS.
    \end{tikzcd}
\end{equation*}
We will think of $NS$ as a vector bundle $\nu\colon NS\rightarrow S$ with a fibrewise $\Gamma$-action. We define the normal cone bundle to be the conically fibred space 
\begin{align*}
    \nu_0\colon N_0\coloneqq NS/\Gamma\rightarrow S.
\end{align*}
Let further $\m{Fr}(X/S,g)$ be the natural $\m{N}_{\m{O}(n)}(\Gamma)$-reduction of the  $\m{O}(n)$ frame bundle $Fr_{\m{O}(n)}(X)$ of $(X,g)$ restricted to $S$. Define the associated (orbifold) bundle 
\begin{align*}
    Y\coloneqq \m{Fr}(X/S,g)\times_{\m{N}_{\m{O}(n)}(\Gamma)}\mathbb{S}^{m-1}/\Gamma\xrightarrow[]{\vartheta} S.
\end{align*}
Whenever $(S,g_S)$ is of type A, the bundle $Y$ is smooth. The Riemannian orbifold structure $g$ induces a splitting 
\begin{align*}
    TN_0\cong H_0\oplus VN_0\cong\nu^*_0TS\oplus\nu^*_0NS
\end{align*}
and a Riemannian orbifold structure on the normal cone bundle $N_0$ such that
\begin{align*}
    g_0=\nu^*i^*g=\nu^*_0g_S+g_{0;V}\in  \Gamma\left(N_0,\m{Sym}^{2}H_0^*\oplus\m{Sym}^2V^*N_0\right).
\end{align*}
Equivalently, we can think of a Riemannian orbifold structure on $(N_0,g_0)$ as a real Hermitian Lie algebroid structure on $(\prescript{ie}{}{}T[N_0\colon S],g_0)$.\\

Let $\epsilon>0$ and let  
\begin{align*}
    \m{Tub}_{5\epsilon}(S)\coloneqq (Y\times (0,5\epsilon))\subset N_0
\end{align*}
and
\begin{align*}
    \m{exp}_{g}\colon \m{Tub}_{5\epsilon}(S)\rightarrow U_{5\epsilon}\subset X
\end{align*}
be a tubular neighbourhood of width $5\epsilon$, induced by the Riemannian structure.\footnote{Later we will assume that $\epsilon\sim t^\lambda$ for $0\leq\lambda<1$.} Let $r\colon N_0\rightarrow[0,\infty)$ denote the radius function on the normal cone bundle. Then there exists an asymptotic expansion
\begin{align*}
    \m{exp}_{g}^*g=g_0+g_{hot}
\end{align*}
and the higher order terms satisfy
\begin{align*}
    |g_{hot}|_{g_0}=\mathcal{O}(r^2).
\end{align*}
The first order term vanishes, as $S$ is locally the fixed point sets of isometries. This structure is referred to as in the literature (see \cite[Chap. 2]{gray2003tubes}) as a Fermi atlas with respect to $S$.

\begin{lem}
    Let $S\subset X$ be a singular stratum of a Riemannian orbifold $(X,g)$. Its normal cone bundle $\nu_0\colon (N_0,g_0)\rightarrow (S,g_S)$ is a Riemannian CF-space. A tubular neighbourhood of the singular stratum endows $(X,g)$ with the structure of a Riemannian CFS-space of rate $2$.
\end{lem}

\subsubsection{Local Models of Resolutions}
\label{Local Models of Resolutions}

In the following we will axiomatises the local model of resolutions of CFS spaces at the working example of Riemannian orbifolds. We will introduce the concept of asymptotically conical fibrations (ACF).

\begin{defi}
A Riemannian fibration $\nu\colon (N,g)\rightarrow (S,g_S)$ is an \textbf{asymptotic conical fibration (ACF)} over a fibration $\vartheta\colon (Y,g_{Y})\rightarrow (S,g_S)$ of rate $\gamma$, if there exists compact subset $K_R$ and a fibre bundle diffeomorphism $\rho\colon N\backslash K_R\rightarrow CF^{[R,\infty)}_\vartheta(Y)$ such that
\begin{align*}
    (\nabla^{g_{CF_\vartheta(Y)}})^k\left(\rho_*g-g_{CF_\vartheta(Y)}\right)=&\mathcal{O}(r^{\gamma})\\
    (\nabla^{g_{CF_\vartheta(Y);V}})^k\left(\rho_*g-g_{CF_\vartheta(Y);V}\right)=&\mathcal{O}(r^{\gamma-k}).
\end{align*}
\end{defi}

Generalising the notion of these noncompact geometries leads to the notion of QACF spaces. These geometries include all local models of orbifold resolutions.

\begin{defi}
A Riemannian fibre bundle $\nu\colon (N,g)\rightarrow S$ is \textbf{quasi asymptotic to a conical fibration (QACF)} of rate $\gamma$ over a stratified fibration $\vartheta\colon (Y,g_{Y})\rightarrow (S,g_S)$, if there exists subset $S_R$ and a fibre bundle diffeomorphism $\rho\colon N\backslash S_R\rightarrow CF^{[0,R)}_\vartheta(Y^{reg})$ such that 
\begin{align*}
    (\nabla^{g_{CF_\vartheta(Y)}})^k\left(\rho_{*}g-g_{CF_\vartheta(Y)}\right)=\mathcal{O}(r^{\gamma}).
\end{align*}
\end{defi}

A special class of spaces that are (Q)ACF spaces are the so called (Quasi) Asymptotically Locally Euclidean ((Q)ALE) spaces. They are formed by the subclass of the former that fibre over the point $S=\{pt.\}$ and were introduced by Joyce in \cite[Chap. 8 and 9]{joyce2000compact}. Moreover, given a $G$-structure on these noncompact spaces, we will specify their asymptotic behaviour compared to a flat model structure.\\

Let in the following $(V,\Phi_V)$ be a $m$-dimensional Euclidean space and $\Gamma\subset G\subset \m{SO}(V,g_V)$ a finite subgroup, with a flat, translational invariant $G$-structure $\Phi_V\in \Gamma(V,\m{Str}_G(V))^V$.

\begin{defi}
Assume that $\Gamma$ acts freely on $V\backslash \{0\}$. An orbifold $(M,\Phi)$ with a torsion-free $G$-structure $\Phi\in\Gamma(M,\m{Str}_G(M))$ is called an \textbf{asymptotically locally Euclidean (ALE)} $G$-structure of rate $\gamma$, if there exists a compact subset $B_R\subset M$ and a diffeomorphisms 
\begin{align*}
    \rho\colon M\backslash B_R\rightarrow (V\backslash B_R(0))/\Gamma
\end{align*}
and the $G$-structure satisfies
\begin{align*}
    (\nabla^{g_{V}})^k\left(\rho_*\Phi-\Phi_{V}\right)=\mathcal{O}(r^{\gamma-k}),
\end{align*}
where $r\colon V/\Gamma\rightarrow [0,\infty)$ denotes the radius function.
\end{defi}

\begin{rem}
By a classical result of Bando, Kasue and Nakajima \cite{bando1989construction} Euclidean volume growth and $L^2$-integrable curvature implies that a non-compact, Ricci-flat manifold $(M,g)$ is ALE of rate $\gamma=1-m$. If further $m=4$ or $g$ Kähler, then $\gamma=-m$.
\end{rem}

Let $\Gamma\subset G\subset \m{SO}(V,g_V)$ be a finite subgroup acting on $V$ in a possible non-free way. Previously we assumed that the stratum of points with nontrivial stabiliser was only the origin; now we allow nontrivial linear subspaces to be stabilised by nontrivial subgroups of $\Gamma$.\\

We define 
\begin{align*}
    Fix(\Gamma_i)=&\{v\in V|\m{Stab}(v)=\Gamma_i\subset\Gamma\}\\
    C(T)=&\{g\in\Gamma|g.t=t,\forall t\in T\}\\
    N(T)=&\{g\in\Gamma|g.T=T\}
\end{align*}
 and denote by $\mathcal{L}=\{T_i=Fix(\Gamma_i)|\Gamma_i\subset\Gamma\}$ the set of all fixpoint vector spaces and let $I$ be an indexing set for $\mathcal{L}$. Set $T_0=Fix(1)=V$ and $T_\infty=Fix(\Gamma)=0$ and define a partial ordering $i\succeq j$ if $T_i\subset T_j$. If $V=T_i\oplus T_i$, then $T_i=Fix(C(T_i))$ and $C(T_i)$ act on $V/C(T_i)\cong T_i\times S_i/C(V_i)$. In particular, $W/N(T_i)=(T_i\times S_i/C(T_i))/(N(T_i)/C(T_i))$.\\

Let $T_i,T_j\in\mathcal{L}$. Then 
\begin{align*}
    C(T_i)C(T_j)=C(T_i\cap T_j)
\end{align*}
and consequently, the set $\mathcal{L}$ is closed under intersections. Moreover, the group $\Gamma$ acts on the indexing set $I$, by $T_{g.i}=g.T_i$.\\

We can identify the singular set of $V/\Gamma$ with the set 
\begin{align*}
    Sing(V/\Gamma)=\bigcup_{i\in I\backslash\{0\}}T_i/\Gamma.
\end{align*}
For a generic point $[v]\in Sing(V/\Gamma)$ the singularity of $V/\Gamma$ at $[v]$ is modelled on $T_i\times S_i/C(T_i)$.\\

Let $Sing(V/\Gamma)$ be the singular set of $V/\Gamma$. The singular set is given by the union 
\begin{align*}
    Sing(V/\Gamma)=\bigcup_{i\in I}T_i/\Gamma.
\end{align*}
Define the distance functions for $0\neq i\in I$
\begin{align*}
    d_i\colon V/\Gamma\rightarrow[0,\infty),\hspace{1.0cm}v\mapsto \m{dist}_{g_V}(v,T_i).
\end{align*}
Notice that $d_\infty=r$ is the standard radius function on $V/\Gamma$ and that $d_j\leq d_i$ for $i\succeq i$.

\begin{defi}
A noncompact, connected $m$-dimensional, orbifold $(M,\Phi)$ with a torsion-free $G$-structure $\Phi\in\Gamma(M,\m{Str}_G(M))$ is called \textbf{quasi asymptotically locally Euclidean (QALE) $G$-structure} of rate $\gamma$, if there exists a asymptotically cylindrical suborbifold\footnote{Here asymptotic cylindrical is meant in the topological sense.} $S_R$ and a diffeomorphism 
\begin{align*}
    \rho\colon M\backslash S_R\rightarrow (V/\Gamma)\backslash \m{Tub}_R(Sing(V/\Gamma))
\end{align*}
such that 
\begin{align*}
  (\nabla^{g_{V}})^k(\rho_*\Phi-\Phi_{V})=\sum_{0\neq i\in I}\mathcal{O}(d_i^{\gamma-k})
\end{align*}
for all $k\geq0$.
\end{defi}

\begin{defi}
Let $(M,\Phi)$ be a QALE $G$-structure. A set of QALE distance functions $\{\delta_i\}_{0\neq i\in I}$ on $M$ are functions $\delta_i\colon M\rightarrow [1,\infty)$ satisfying 
\begin{align*}
        (\nabla^{g_{V}})^k(\rho_*\delta_i-d_i)=\mathcal{O}(d_i^{1+\gamma-k})
\end{align*}
for all $k\geq 0$.
\end{defi}

\subsubsection{Interpolating and $k$-Tame Smooth Gromov--Hausdorff Resolutions}
\label{Interpolating and k-Tame smooth Gromov--Hausdorff Resolutions}

In this section we develop a systematic framework for analysing smooth Gromov--Hausdorff resolutions of Riemannian orbifolds. The guiding principle is that such resolutions should be modelled, at small scales near the singular locus, by asymptotically conical fibrations (ACF) glued into the normal cone of the singular stratum. This motivates the notion of interpolating resolutions and serve as reference models for more general resolutions. To quantify how closely an arbitrary smooth Gromov--Hausdorff resolution approximates a interpolating one, we introduce the notion of $k$-tameness, measured by $C^k_{loc}$-convergence of the associated families of metrics under suitable identifications. This formalism provides a convenient language for comparing analytic data across families, and will be the main tool for studying geometric operators, such as Dirac operators, in later sections.\\

The group $(\mathbb{R}_{>0},\cdot)$ acts on the normal cone bundle $\nu_0\colon N_0\rightarrow S$ via the dilatation of the fibres 
    \begin{align*}
        \delta_t\colon N_0\rightarrow N_0,\,[n_s]\mapsto [t\cdot n_s].
    \end{align*}

In particular, the CF-normal cone structure satisfies 
\begin{align*}
    g^t_0=g_{0;H}+t^2\cdot g_{0;V}.
\end{align*}

Let us assume that there exists a family of ACF spaces of rate $\gamma$, forming a smooth Gromov--Hausdorff resolution 
\begin{align*}
    (N_{t^2\cdot\zeta},g_{t^2\cdot \zeta})\in\cat{GHRes}(N_0,g_0),
\end{align*}
such that
\begin{equation*}
\begin{tikzcd}
        (N_\zeta,g^t_\zeta)\arrow[d,dashed,"\rho_\zeta",swap]\arrow[r,"\delta_t"]&(N_{t^2\zeta},g_{t^2\zeta})\arrow[d,dashed,"\rho_{t^2\cdot\zeta}"]\\
        (N_0,g^t_0)\arrow[r,"\delta_t"]&(N_0,g_0)
\end{tikzcd}
\end{equation*}
and the Riemannian structure decomposes into 
\begin{align*}
    \delta_t^*g_{t^2\cdot \zeta}=g_{\zeta;H}+t^2\cdot g_{\zeta;V}.
\end{align*}

Using the splitting $H_\zeta\oplus VN_\zeta$ we define the connection $ \nabla^{\oplus_\zeta}$ and its torsion tensor $T_{\oplus_\zeta}$. This connection is $t$-independent and we identify
\begin{align*}
    \nabla^{g^t_\zeta}=\nabla^{\oplus_\zeta}+\frac{1}{2}\overline{\Lambda^t_\zeta}(T_{\oplus_\zeta}).
\end{align*}
Following \cite[(10.3)]{berline1992heat} the limit 
\begin{align*}
    \nabla^{g^0_\zeta}=\nabla^{\oplus_\zeta}+\frac{1}{2}\overline{\Lambda}^0(T_{\oplus_\zeta})
\end{align*}
is well-defined.

\begin{rem}
    By abuse of notation we will denote the lift of the dilatation action by $\delta_t$.
\end{rem}

\begin{nota}
    Let $r\colon N_\zeta\rightarrow[0,\infty)$ be the radius function on $N_\zeta$ defined by $r=r\circ \rho_\zeta$. We will denote by
    \begin{align*}
        B_{R}(N_\zeta)\coloneqq r^{-1}(0,R)\und{1.0cm}A_{(R_1,R_2)}(S)\coloneqq\m{Tub}_{R_2}(S)\backslash \m{Tub}_{R_1}(S).
    \end{align*}
\end{nota}
\begin{defi}
We define the \textbf{gluing morphism} $\Gamma_{t^2\cdot\zeta}\coloneqq \m{exp}_g\circ \rho_{t^2\cdot\zeta}$ and denote the family of spaces
\begin{align*}
    X_{\zeta;t}\coloneqq(B_{5\epsilon}(N_{t^2\cdot\zeta})\sqcup X\backslash S)/_{\sim_ {\Gamma_{t^2\cdot\zeta}}}
\end{align*}
Furthermore, these spaces are equipped with a resolution morphism 
\begin{align*}
    \rho_{\zeta;t}\colon X_{\zeta;t}\dashrightarrow X.
\end{align*}
\end{defi}

\begin{rem}
The dilation map
\begin{align*}
    \delta_t \colon (N_\zeta, g^t_\zeta) \rightarrow (N_{t^2 \cdot \zeta}, g_{t^2 \cdot \zeta})
\end{align*}
identifies the resolution data of the normal cone bundle $(N_0, g_0)$ with the fixed fibration $\nu_\zeta \colon N_\zeta \to S$, equipped with a family of metrics $g^t_\zeta$. This allows us to define a family of spaces by gluing
\begin{align*}
    \rho^{t}_\zeta \colon X^t_{\zeta} \coloneqq \left( B_{5t^{-1}\epsilon}(N_\zeta) \sqcup X \setminus S \right) / \sim_{\Gamma^t_\zeta} \dashrightarrow X,
\end{align*}
where the gluing map is $\Gamma^t_\zeta \coloneqq \delta_t^* \Gamma_{t^2 \cdot \zeta}$. The spaces $X^t_{\zeta}$ and $X_{\zeta;t}$ are canonically diffeomorphic via the rescaling
\begin{align*}
    \delta_t \colon X^t_{\zeta} \cong X_{\zeta;t} \colon \delta_{t^{-1}}.
\end{align*}

We emphasize that $X^t_{\zeta}$ and $X_{\zeta;t}$ describe the same space, but the notation reflects different perspectives. We use the superscript $X^t_{\zeta}$ to indicate the viewpoint relative to $N_\zeta$, and the subscript $X_{\zeta;t}$ when viewing the space relative to $X$. This distinction is particularly useful since the gluing map $\Gamma^t_\zeta$ incorporates the rescaling map $\delta_t$, and it is often convenient to switch between the two pictures depending on the geometric or analytic context.

\end{rem}

Let ${\chi}\colon [0,1]\rightarrow[0,1]$ be a smooth positive, function such that $\m{supp}({\chi})\subset(0,1]$ and $\m{supp}(1-{\chi})\subset[0,1)$ fixed throughout this work. Define the smooth functions
\begin{align*}
    {\chi_i}(x)\coloneqq&\left\{\begin{array}{cl}
     0&x\in B_{(i-1)\epsilon}(N_{t^2\cdot\zeta})\\
     {\chi}((i-1)\epsilon^{-1}r-i+1)&x\in B_{i\epsilon}(N_{t^2\cdot\zeta})\backslash B_{(i-1)\epsilon}(N_{t^2\cdot\zeta}) \\
     1& \text{else} 
    \end{array}\right.
\end{align*}
and set $\delta^*_t\chi_i=\chi^t_i$.

\begin{nota}
Given two sections $\Psi$ and $\Phi$ of mutually vector bundles on $N_\zeta$ and $X$ that \say{glue} together by a lift $\hat{\Gamma}_{t^2\cdot\zeta}$ of $ \Gamma_{t^2\cdot\zeta}$, we define a new section on $X_t$ by
\begin{align*}
    \Psi\cup^t\Phi\coloneqq&(1-\chi^t_3)\Psi+\chi^t_3(\hat{\Gamma}^t_\zeta)^*( \Gamma^t_\zeta)_*\Phi\\
    =&(1-\chi_3)(\hat{\Gamma}^{t^{-1}}_\zeta)^*( \Gamma^{t^{-1}}_\zeta)_*\Psi+\chi_3\Phi\eqqcolon (\hat{\Gamma}^{t^{-1}}_\zeta)^*( \Gamma^{t^{-1}}_\zeta)_*\Psi\cup_t\Phi
\end{align*}
Instead of $(\hat{\Gamma}^t_\zeta)^*( \Gamma^t_\zeta)_*\Phi$ we will abuse the notation and write $\delta_t^*\Phi$.
\end{nota}

\begin{defi}
    We interpolate between the Riemannian orbifold structure on $X$ and the ACF structure on $B_{5t^{-1}\epsilon}(N_\zeta)$ using the cut-off function ${\chi^t_3}$ by
\begin{align*}
    g^{pre;t}_\zeta\coloneqq g^t_{\zeta}\cup^t g =\delta_t^*(g_{t^2\cdot\zeta}\cup_t g)\eqqcolon \delta_t^*g^{pre}_{\zeta;t}.
\end{align*}
\end{defi}

\begin{defi}
We set $\Omega_\zeta=(1-\chi_5)T_{\oplus_\zeta}$
and define the adapted metric connection 
    \begin{align}
        \label{hatnablaoplustzeta}
    \widehat{\nabla}^{g^t_\zeta}\coloneqq&\nabla^{\otimes_\zeta}+\frac{1}{2}\overline{\Lambda}^t_\zeta(\Omega_\zeta).
    \end{align}
\end{defi}


Using this adapted connection can be seen as a leading order term in the expansion of the Levi-Civita connection $\nabla^{g_{\zeta;t}^{pre}}$. The following lemma makes this statement more precise.

\begin{lem}
    The Levi-Civita connection with respect to $g^{pre}_{\zeta;t}$ satisfies 
    \begin{align*}
        g^{pre}_{\zeta;t}(\nabla^{g^{pre}_{\zeta;t}}_XY,Z)=&g_{t^2\cdot\zeta}(\widehat{\nabla}^{g_{t^2\cdot\zeta}}_XY,Z)\cup_tg(\nabla^{g}_XY,Z)\\&+\frac{1}{2}\m{d}\chi_3\wedge((\rho_{t^2\cdot\zeta})_*g_{t^2\cdot\zeta}-g)(X,Y,Z)
    \end{align*}
\end{lem}

\begin{proof}
    Let $X,Y,Z\in\mathfrak{X}(X_{\zeta;t})$. Then using Koszul's formula
    \begin{align*}
        2g^{pre}_{\zeta;t}(\nabla^{g_{\zeta;t}}_XY,Z)=&Xg^{pre}_{\zeta;t}(Y,Z)+Yg^{pre}_{\zeta;t}(X,Z)-Zg^{pre}_{\zeta;t}(X,Y)\\
        &+g^{pre}_{\zeta;t}([X,Y],Z)-g^{pre}_{\zeta;t}([Y,Z],X)+g^{pre}_{\zeta;t}([Z,X],Y).
    \end{align*}
\end{proof}

\begin{cor}
    The family $\rho_t\colon(X_{\zeta;t},g^{pre}_{\zeta;t})\dashrightarrow (X,g)$ defines a smooth Gromov--Hausdorff resolution. We will refer to such Gromov--Hausdorff resolutions as being \textbf{interpolating}. 
\end{cor}

In the following paper, we will analyse smooth Gromov--Hausdorff resolutions that are constructed through the described construction. In many cases, smooth Gromov--Hausdorff resolutions of interest are \say{close} to such resolutions constructed from such limiting data. This naturally leads to the following definition

\begin{defi}
\label{ktame}
    A smooth Gromov--Hausdorff resolution 
    \begin{align*}
        [\rho_t\colon(X_t,g_t)\dashrightarrow (X,g)]\in\cat{GHRes}(X,g)
    \end{align*}
    is called \textbf{$k$-tame}, if there exists a interpolating smooth Gromov--Hausdorff resolution 
    \begin{align*}
        [\rho_{\zeta;t}\colon(X_{\zeta;t},g^{pre}_{\zeta;t})\dashrightarrow (X,g)]\in\cat{GHRes}(X,g)
    \end{align*}
    and a family of diffeomorphisms $f_t\colon X_t\cong X_{\zeta;t}$ such that $f_t^*g_t\xrightarrow[C^{k}_{-1;t}]{t\to 0} g_{\zeta;t}^{pre}$. Here $C^{k}_{-1;t}$ is a weighted $C^k$-norm defined in \cite[Def. 9.4]{majewskiDirac}.
\end{defi}

In order to study the analytic properties of families of geometric operators, such as Dirac operators, on $k$-tame smooth Gromov--Hausdorff resolution $[\rho_t\colon(X_t,g_{t})\dashrightarrow (X,g)]\in \cat{GHRes}(X,g)$ it suffices to study them on the interpolating smooth Gromov--Hausdorff resolutions.\\

To avoid cumbersome booking work and improve the readability, we will restrict ourselves throughout this paper to orbifolds with singular strata of depth one and type A (see Appendix \ref{Orbifolds}). However, the results presented in this paper extend $k$-tame orbifold resolutions of general strata.

\begin{manualassumption}{1}
\label{ass1}    
We will assume the following: 
\begin{itemize}
    \item The Riemannian orbifold $(X,g)$ contains one singular stratum of depth one and type A.
    \item 
The orbifold resolution 
\begin{align*}
[\rho_t\colon(X_t,g_{t})\dashrightarrow (X,g)]=[\rho_{\zeta;t}\colon(X_{\zeta;t},g^{pre}_{\zeta;t})\dashrightarrow (X,g)].
\end{align*}
is interpolating.
\end{itemize}
\end{manualassumption}

\section{Hodge Theory on Tame Gromov--Hausdorff-Resolutions of Riemannian Orbifolds}
\label{Hodge Theory on Tame Gromov--Hausdorff-Resolutions of Riemannian Orbifolds}

As recalled in \eqref{Diracfixedpoint}, the condition of deforming a nearly torsion-free $\m{Spin}(7)$-structure into a torsion-free one as a nonlinear fixed-point problem involving a right-inverse of the Hodge–de Rham operator. Solving this deformation problem requires uniform analytic control of the Hodge theory across families of resolutions.\\

Here we develop such a framework in the general setting of tame smooth Gromov–Hausdorff-resolutions of Riemannian orbifolds. Building on the analytic results of \cite{majewskiDirac}, we establish Hodge theory for these resolutions, showing how harmonic forms and cohomology behave under the resolution process. The key point is to construct and compare harmonic forms on the resolution with those on the singular model, in a way that is compatible with gluing techniques.\\

These results form the analytic backbone of the paper. In Section \ref{Adiabatic Spin(7)-Orbifold Resolutions} they are applied to determine the resolution data required for $\m{Spin}(7)$-orbifold singularities, and in Section \ref{Existence of torsion-free Spin(7) Structures on Orbifold Resolutions} they are used to solve the deformation problem and prove the existence of torsion-free $\m{Spin}(7)$-orbifold resolutions.\\

In the following section, we will apply the author's results in \cite{majewskiDirac} on the uniform elliptic theory of Dirac operators on orbifold resolutions to the Hodge--de Rham operator on Riemannian orbifold resolutions. In Section \ref{The Cohomology of Resolved Orbifolds} we computed the cohomology of the resolved orbifold using topological methods. By Hodge-theory the cohomology of the resolved orbifold consists of harmonic forms. The following section relates these harmonic forms to kernels of certain Dirac operators on the three gluing regions, i.e. the CFS- the CF- and the ACF-part of the orbifold resolution.

\subsection{The Hodge--de Rham Operator of an Orbifold Resolution}
\label{The Hodge--de Rham Operator of an Orbifold Resolution}

We begin discussing the set-up of \cite{majewskiDirac} in the present case, by stating general results about the signature operator.\\

The map 
\begin{align*}
    T^*X_{\zeta;t}\rightarrow \m{End}(\wedge^\bullet T^*X_{\zeta;t}),\xi\mapsto\left(\omega\mapsto \xi\wedge\omega-\iota_{\xi^{\sharp_{g^{pre}_{\zeta;t}}}}\omega\right) 
\end{align*}
satisfies the Clifford property and hence, induces a faithfully representation of
\begin{align*}
    \m{cl}_{g^{pre}_{\zeta;t}} \colon\m{Cl}(T^*X_{\zeta;t},g^{pre}_{\zeta;t})\rightarrow\mathfrak{so}(\wedge^\bullet T^*X_{\zeta;t},g^{pre}_{\zeta;t})\subset\m{End}(T^*X_{\zeta;t}).
\end{align*}
which is skew-Hermitian with respect to $g^{pre}_{\zeta;t}$. There exists an involution-isometry
\begin{align*}
    \epsilon_{g^{pre}_{\zeta;t}}=i^{(\bullet(\bullet-1))}*_{g^{pre}_{\zeta;t}} \colon\wedge^\bullet T^*X_{\zeta;t}\rightarrow T^*X_{\zeta;t}
\end{align*}
such that $\{\m{cl}_{g^{pre}_{\zeta;t}},\epsilon_{g^{pre}_{\zeta;t}}\}=0$, and 
\begin{align*}
    \wedge^\bullet T^*X_{\zeta;t}=\wedge^\bullet_{+_{g^{pre}_{\zeta;t}}}T^*X_{\zeta;t}\oplus \wedge^\bullet_{-_{g^{pre}_{\zeta;t}}}T^*X_{\zeta;t}.
\end{align*}
Hence, the tuple 
\begin{align*}
    (\wedge^\bullet T^*X_{\zeta;t},\m{cl}_{g^{pre}_{\zeta;t}},g^{pre}_{\zeta;t},\nabla^{g^{pre}_{\zeta;t}},\epsilon_{g^{pre}_{\zeta;t}})
\end{align*}
forms a graded Hermitian Dirac bundle. Its Dirac operator 
\begin{align*}
    D_{\zeta;t}^{pre}\coloneqq \m{d}+\m{d}^{*^{pre}_{\zeta;t}}=\left(\begin{array}{cc}
         &D_{\zeta;t}^{pre;-}  \\
         D_{\zeta;t}^{pre;+}& 
    \end{array}\right)
\end{align*}
and $D_{\zeta;t}^{pre;+}$ is the signature operator of $(X_{\zeta;t},\Phi^{pre}_{\zeta;t})$.\\

Notice that the skew-adjoint \say{opposite} signature operator decomposes with respect to $\epsilon_{g^{pre}_{\zeta;t}}$ into
\begin{align*}
    D_{\zeta;t}^{op;pre}\coloneqq \m{d}-\m{d}^{*_{\zeta;t}^{pre}}=\left(\begin{array}{cc}
         D^{op;pre;-}_{\zeta;t}&\\
         & D^{op;pre;+}_{\zeta;t}
    \end{array}\right).
\end{align*}
and hence,
\begin{align*}
    D_{\zeta;t}^{pre}\pi_{\pm;t}=\pi_{\pm;t}D_{\zeta;t}^{pre}\und{1.0cm}D_{\zeta;t}^{op;pre}\pi_{\pm;t}=\pi_{\mp;t}D_{\zeta;t}^{op;pre}.
\end{align*}
This operator squares to $-\Delta_{g^{pre}_{\zeta;t}}$ and anti-commutes with the operator $D_{\zeta;t}^{pre}$. Thus, we deduce that 
\begin{align*}
    \m{ker}_{L^2}(D_{\zeta;t}^{pre})=\m{ker}_{L^2}(D_{\zeta;t}^{op;pre})\und{1.0cm}\m{im}_{L^2}(D_{\zeta;t}^{pre})=\m{im}_{L^2}(D_{\zeta;t}^{op;pre}).
\end{align*}

In the following we will use the resolution data to construct operators on $X$, $N_0$ and $N_\zeta$, whose analytic behaviour will determine the one of $D^{pre}_{\zeta;t}$.\\

Using the gluing morphisms $\Gamma^t_\zeta$ and $\hat{\Gamma}^t_\zeta$ we define 
\begin{equation*}
    \begin{tikzcd}
        &(N_\zeta,g^t_\zeta)\sqcup (X,g)\arrow[dl,"\bigsqcup_{ \Gamma^t_\zeta}",swap,dotted]\arrow[dr,"\bigsqcup_{\overline{\Gamma}^t_\zeta}",dotted]&&\\
        (X_{\zeta;t},g^{pre}_{\zeta;t})&&(\overline{X}_t,\overline{g}_t)\arrow[r,equal]&(N_0,g_0)        
    \end{tikzcd}
\end{equation*}
the spaces $X_{\zeta;t}$ and $\overline{X}_t$ as well as the Dirac bundles
\begin{align*}
    \pi_t\colon &(\wedge^\bullet X_{\zeta;t},\m{cl}_{g^{pre}_{\zeta;t}},g^{pre}_{\zeta;t},\nabla^{g^{pre}_{\zeta;t}})\rightarrow (X_{\zeta;t},g^{pre}_{\zeta;t})\\\overline{\pi}_t\colon &(\wedge^\bullet\overline{X}_t,\m{cl}_{\overline{g}_t},\overline{g}_t,\nabla^{\overline{g}_t})\rightarrow (\overline{X}_t,\overline{g}_t).
\end{align*}

\begin{defi}
We will define the operators
\begin{align*}
    D& \colon\Omega^{\bullet}(X)\rightarrow \Omega^{\bullet}(X)\und{1.0cm}
    &\widehat{D}_0 \colon\Omega^{\bullet}(N_0)\rightarrow \Omega^{\bullet}(N_0)\\
        \widehat{D}_{0;V}& \colon\Omega^{\bullet}(N_0)\rightarrow \Omega^{\bullet}(N_0)\und{0.7cm}
        &\widehat{D}^t_\zeta \colon\Omega^{\bullet}(N_\zeta)\rightarrow \Omega^{\bullet}(N_\zeta).
\end{align*}

We will refer to $D=\m{d}+\m{d}^*$ as the orbifold signature operator, $\widehat{D}_0=\m{d}^{1,0}+\m{d}^{0,1}+\left(\m{d}^{1,0}+\m{d}^{0,1}\right)^*$ as the normal operator\footnote{This work is related to the normal operator $N(D)$ used in \cite{mazzeo1991elliptic}}, $D_{0,V}=\m{d}^{0,1}+\left(\m{d}^{0,1}\right)^*$ as the vertical or indicial operator and $\widehat{D}^t_\zeta=\m{d}^{1,0}+\m{d}^{0,1}+t^{-1}\cdot \left(\m{d}^{1,0}+\m{d}^{0,1}\right)^*+t\cdot\m{cl}(\Omega_\zeta)$ as the resolution of the normal operator.
\end{defi}

Understanding the geometric behaviour of $D_{\zeta;t}^{pre}$ will essentially reduces to understand the behaviour of the orbifold Hodge--de Rham, the resolution of the normal operator and their interaction.\\

In order to prove the exactness of gluing we will need to construct a functional analytic realisation of the normal operator that is an isomorphism on a comeagre set of the parameters, the so called rates characterising the realisation. The meagre set on which the normal operator fails to be an isomorphism are the so called indical roots or critical rates. These rates are determined by homogeneous solutions of the $\mathbb{R}_{>0}$-invariant indicial operator $I(D)=D_{0,V}$.\\

The analytic relation of $\widehat{D}_0$ as an isomorphism of Banach spaces will determine Fredholm maps, realising the two operators $D$ and $\widehat{D}^t_\zeta$ on function spaces \say{asymptotic} to the ones realising the normal operator. The sum of the indices of these two realisations coincides with the index of the operator $D^{t}$ and is in particular invariant under wall-crossings of the rate. This holds even though the indices of the two realisation depend on the connected component of the space of rates.\\

By using these fundamental insights we will further be able to construct, in certain cases, \say{adapted} norms and we construct a uniformly bounded right-inverse.

\subsection{The Cohomology of Resolved Orbifolds}
\label{The Cohomology of Resolved Orbifolds}

We first record the topological effect of replacing the singular orbifold neighbourhood by its smooth resolution. This calculation is independent of the analytic choices made later, but it determines the
finite-dimensional spaces which appear in the Hodge theory of the resolved spaces. In particular, the dimension of the harmonic forms on $X_{\zeta;t}$ is fixed by the topology of the resolution, whereas their
analytic representatives depend on the collapsing geometry.\\

Let $S\subset X$ denote the singular stratum and let $\nu_\zeta\colon N_\zeta\rightarrow S$ be the local resolution model. Topologically, the resolved space $X_{\zeta;t}$ is obtained by removing a tubular neighbourhood of $S$ in $X$ and gluing in a truncated neighbourhood of the exceptional locus
$\Upsilon_\zeta\subset N_\zeta$. Thus the new cohomology classes on $X_{\zeta;t}$ are precisely those carried by the exceptional geometry, modulo the classes already present on the stratum $S$. The following
statement makes this precise.
\begin{prop}
\label{cohomofXt}
There exists a short exact sequence of singular cohomology groups
\begin{equation*}
    \begin{tikzcd}
        \m{H}^\bullet(X,\mathbb{R})\arrow[r,hook]&\m{H}^\bullet(X_{\zeta;t},\mathbb{R})\arrow[r,two heads]&\m{H}^\bullet(\Upsilon_\zeta,\mathbb{R})/\m{H}^\bullet(S,\mathbb{R})
    \end{tikzcd}
\end{equation*}
and hence, 
\begin{align*}
    b_k(X_{\zeta;t})=b_k(X)+b_k(\Upsilon_\zeta)-b_k(S).
\end{align*}
The cohomology $\m{H}^\bullet(X)$ is the cohomology of the coarse moduli space of the orbifold $X$.
\end{prop}

\begin{proof}
We compute the cohomology of $X_{\zeta;t}$ using relative cohomology\\

\adjustbox{scale=0.55,center}{
    \begin{tikzcd}
        ...\arrow[r]&\m{H}^{k-1}(\m{Tub}_{5\epsilon}(S))\arrow[r]\arrow[dd,hook]&\m{H}^{k}(X,\m{Tub}_{5\epsilon}(S))\arrow[r]\arrow[dd,equal]&\m{H}^{k}(X)\arrow[r]\arrow[dd,hook]&\m{H}^{k}(\m{Tub}_{5t^{-1}\epsilon}(S))\arrow[dd,hook]\arrow[r]&...\\
        &&&&&\\
        ...\arrow[r]&\m{H}^{k-1}(Tub_{5t^{-1}\epsilon}(\Upsilon_{t^2\cdot\zeta}))\arrow[r]\arrow[dd,two heads]&\m{H}^{k}(X_{\zeta;t},Tub_{5t^{-1}\epsilon}(\Upsilon_{t^2\cdot\zeta}))\arrow[r]&\m{H}^{k}(X_{\zeta;t})\arrow[r]\arrow[dd,two heads]&\m{H}^{k}(Tub_{5t^{-1}\epsilon}(\Upsilon_{t^2\cdot\zeta}))\arrow[r]\arrow[dd,two heads]&...\\
        &&&&&\\
        ...\arrow[r,two heads]&\m{H}^{k-1}(Tub_{5t^{-1}\epsilon}(\Upsilon_{t^2\cdot\zeta}))/\m{H}^{k-1}(\m{Tub}_{5\epsilon}(S))&&\m{H}^{k}(Tub_{5t^{-1}\epsilon}(\Upsilon_{t^2\cdot\zeta}))/\m{H}^{k}(\m{Tub}_{5\epsilon}(S))\arrow[r,hook]&\m{H}^{k}(Tub_{5t^{-1}\epsilon}(\Upsilon_{t^2\cdot\zeta}))/\m{H}^{k}(\m{Tub}_{5\epsilon}(S))\arrow[r,two heads]&...\\
    \end{tikzcd}}
Here we used that $Tub_{5t^{-1}\epsilon}(\Upsilon_{t^2\cdot\zeta})$ retracts onto $\Upsilon_\zeta$ and $\m{Tub}_{5\epsilon}(S)$ onto $S$.
\end{proof}

This description will later be compared with Hodge theory. On the analytic side, the classes coming from $X$ are represented by harmonic forms on the orbifold part, while the additional classes are represented
by vertically harmonic forms on the ACF resolution model. The role of the gluing theorem is then to decide when these two collections of local harmonic forms glue to genuine harmonic forms on $X_{\zeta;t}$ without creating an obstruction.\\

The remaining question is analytic rather than topological. Although the Betti numbers of $X_{\zeta;t}$ are independent of $t$, the harmonic representatives of these classes may concentrate in different regions as $t\to0$. Classes inherited from $X$ are expected to converge to orbifold harmonic forms, while exceptional classes concentrate along the resolution region and are modelled by $L^2$-harmonic forms on the ACF fibres. To make this comparison uniform in $t$, we now introduce
weighted Hölder spaces adapted to the three geometric regimes of the resolution; the CFS orbifold region, the CF normal cone, and the ACF resolution model.

\subsection{Weighted Function Spaces on Orbifold Resolutions}
\label{Weighted Function Spaces on Orbifold Resolutions}

We now introduce the function spaces in which the Hodge--de Rham operator will be studied uniformly along the resolution family. The main point is that the geometry of $X_{\zeta;t}$ is not uniformly bounded in ordinary Hölder norms as $t\to0$. Near the singular stratum the metric develops a conically fibred structure, while in the resolution region one sees a rescaled ACF model. Thus a single unweighted Hölder norm cannot simultaneously describe the orbifold region, the normal cone, and the resolving space.\\

The appropriate norms have to keep track of the different geometric scales. On the conically fibred model $N_0$, radial derivatives and derivatives tangent to the fibres scale differently from derivatives along the stratum $S$. This anisotropy is reflected in the definition of the CF-weighted Hölder norms below. The weight records the expected growth or decay in the conical direction, while the splitting of the connection into horizontal and vertical parts records how many derivatives fall in each direction.\\

These weighted spaces serve two purposes. First, they give Fredholm realisations of the model operators away from the critical rates determined by the indicial operator. Second, they are compatible with the gluing maps used to compare forms on $X_{\zeta;t}$ with forms on the orbifold part and on the ACF resolution model. This compatibility is what makes it possible to formulate uniform elliptic estimates for the Hodge--de Rham operator on the degenerating family.\\

We begin with the conically fibred model. Let $N_0$ be the normal cone over the singular stratum $S$, equipped with its splitting into horizontal and vertical directions. The corresponding connection will be denoted by
\begin{align*}
    \nabla^{\oplus_0}=\nabla^{\oplus_0;(1,0)}+\nabla^{\oplus_0;(0,1)} .
\end{align*}
Here $\nabla^{\oplus_0;(1,0)}$ denotes differentiation along the stratum direction and $\nabla^{\oplus_0;(0,1)}$ denotes differentiation in the conical fibre direction. The weighted norms below assign
different powers of the conical radius depending on how many derivatives are taken in the vertical direction.

\begin{nota}
Let $\m{word}^+_{i,j}(A,B)$ denote the sum of all words of length $j$ in $A$ and $B$ counting $i$-times the letter $A$. In particular, $\m{word}^+_{i,j}(A,B)=\m{word}^+_{j-i,j}(B,A)$. 
\end{nota}

\begin{defi}
 We define the space $ \Omega^{\bullet;k}_{\m{CF};\beta}(N_0)$ to be the space of all $C^k$-forms $\Psi$ such that 
\begin{align*}
    \left|\left|w_{\m{CF};\beta}\left\{(\nabla^{\oplus})^j\right\}\Psi\right|\right|_{C^0}<\infty
\end{align*}
where 
\begin{align*}
    w_{\m{CF};\beta}\left\{(\nabla^{\oplus_0})^j\right\}\coloneqq \bigoplus_{i=0}^jw_{\m{CF};(\beta-j+i)}\m{word}^+_{i,j}(\nabla^{\oplus_0;(1,0)},\nabla^{\oplus_0;(0,1)}) 
\end{align*}
and equip the spaces $ \Omega^{\bullet;k}_{\m{CF};\beta}(N_0)$ with the metrics
\begin{align*}
    \left|\left|\Psi\right|\right|_{C^{k}_{\m{CF};\beta}}\coloneqq&\sum_{j=0}^k\left|\left|w_{\m{CF};\beta}\left\{(\nabla^{\oplus_0})^j\right\}\Psi\right|\right|_{C^0_{g_0}}
\end{align*}
The space $ \Omega^{\bullet}_{\m{CF};\beta}(N_0)$ is defined to be the intersection of all $ \Omega^{\bullet;k}_{\m{CF};\beta}(N_0)$ for all $k$. Furthermore, we define the $w_{\m{CF};\beta}$-weighted Hölder spaces $  \Omega^{\bullet;k,\alpha}_{\m{CF};\beta}(N_0)$ as the completions with respect to the norms 
\begin{align*}
    \left|\left|\Psi\right|\right|_{C^{k,\alpha}_{\m{CF};\beta}}\coloneqq&\sum_{j=0}^k\left|\left|w_{\m{CF};\beta}\left\{(\nabla^{\oplus_0})^j\right\}\Psi\right|\right|_{C^0_{g_0}}+[(\nabla^{\oplus_0})^j\Psi]_{C^{0,\alpha}_{\beta-CF}}\\
[(\nabla^{\oplus_0})^j\Psi]_{C^{0,\alpha}_{\beta-CF}}\coloneqq&\underset{B_{g_0}}{\m{sup}}\left\{\bigoplus_{i=0}^jw_{\m{CF};(\beta-j+i-\alpha)}(x,y)\frac{\left|\nabla^{\oplus_0}_{i,j}\Psi(x)-\Pi^{x,y}\nabla^{\oplus_0}_{i,j} \Psi(y)\right|_{g}}{d_{g_0}(x,y)^\alpha}\right\}
\end{align*}
where $\nabla^{\oplus_0}_{i,j}\coloneqq\m{word}^+_{i,j}(\nabla^{\oplus_0;(1,0)},\nabla^{\oplus_0;(0,1)}) $, $\Pi^{x,y}$ denotes the parallel transport along the geodesic connecting $x$ and $y$, and 
\begin{align*}
    B_{g_0}(U)\coloneqq\{x,y\in U|\,x\neq y,\,\m{dist}_{g_0}(x,y)\leq w_{\m{CF};(\beta-j+i-\alpha)}(x,y)\}\subset U\times U
\end{align*}
\end{defi}

Using these function space the operator $\widehat{D}_0$ can be realised as a unbounded operator 
\begin{align*}
    \widehat{D}_0\colon &\m{dom}_{  \Omega^{\bullet;k+1,\alpha}_{\m{CF};\beta}}(\widehat{D}_0)\subset   \Omega^{\bullet;k+1,\alpha}_{\m{CF};\beta}(N_0)\rightarrow  \Omega^{\bullet;k,\alpha}_{\m{CF};\beta-1}(N_0).
\end{align*}
We thus define the following Banach space.

\begin{defi}
We will denote the domain of $\widehat{D}_0$ by
\begin{align*}
    \Omega^{\bullet;k+1,\alpha}_{\m{CF};\beta}(\widehat{D}_0;\m{APS})=\left\{\left.\Psi\in  \Omega^{\bullet;k+1,\alpha}_{\m{CF};\beta}(N_0;\m{APS})\right|\widehat{D}_0\Psi\in   \Omega^{\bullet;k,\alpha}_{\m{CF};\beta-1}(N_0)\right\}
\end{align*}
and equip it with the graph norm 
\begin{align*}
    \left|\left|\Psi\right|\right|_{  C^{k+1,\alpha}_{\m{CF};\beta}(\widehat{D}_0)}=\left|\left|\Psi\right|\right|_{C^{k+1,\alpha}_{\m{CF};\beta}(\nabla^{\otimes_0})}+\left|\left|\widehat{D}_0\Psi\right|\right|_{C^{k,\alpha}_{\m{CF};\beta-1}(\nabla^{\otimes_0})}.
\end{align*}
\end{defi}

\begin{rem}
    The space $\Omega^{\bullet;k+1,\alpha}_{\m{CF};\beta}(N_0;\m{APS})$ denotes $C^{k,\alpha}_{\m{CF};\beta}$-regular differential forms subject to APS-type boundary conditions. These need to be imposed to realise $\widehat{D}_0$ as a left-semi Fredholm operator. Its precise definition can be found in \cite{majewskiDirac}.
\end{rem}

\begin{defi}
    We define the set of critical rates of $\widehat{D}_0$, $D$ and $\widehat{D}^t_{\zeta}$ to be
\begin{align*}
        \lambda\in\mathcal{C}(\widehat{D}_0)\coloneqq\left\{\lambda+\frac{m-1}{2}-\bullet\in\sigma\left(D_{\mathbb{S}^{m-1}/\Gamma} \colon H^1\Omega^{\bullet}(\mathbb{S}^{m-1}/\Gamma)\rightarrow L^2\Omega^{\bullet}(\mathbb{S}^{m-1}/\Gamma)\right)\right\}.
    \end{align*}
\end{defi}

In the following we will extend the notion of CF-weighted Hölder spaces to CFS- and ACF-weighted Hölder spaces. For the precise definition we refer to \cite{majewskiDirac}.\\

Given the exponential tubular neighbourhood $\m{exp}_{g_\Phi} \colon\m{Tub}_{5\epsilon}(S)\hookrightarrow X$, we define the CFS-weighted Hölder spaces 
\begin{align*}
    \Omega^{\bullet;k+1,\alpha}_{\m{CFS};\beta;\epsilon}(X)\und{1.0cm}\Omega^{\bullet;k+1,\alpha}_{\m{CFS};\beta;\epsilon}(X;\m{APS})
\end{align*}
by interpolating between the CF-weighted norms on the tubular neighbourhood and \say{unweighted} Hölder norms on the compact part of the $\m{Spin}(7)$-orbifold.\\

The spaces $(N_\zeta,\Phi^t_\zeta,g^t_\zeta)$ are ACF, i.e. asymptotic to $(N_0,\Phi^t_0,g^t_0)$. By interpolating between the CF-weighted Hölder norms on the non-compact parts of $N_\zeta$ and the \say{unweighted} Hölder norms on the compact region, we define the ACF-weighted Hölder norms
\begin{align*}
    \Omega^{\bullet;k+1,\alpha}_{\m{ACF};\beta;t}(N_\zeta)\und{1.0cm}\Omega^{\bullet;k+1,\alpha}_{\m{ACF};\beta;t}(\widehat{D}^t_\zeta;\m{APS}).
\end{align*}
By the construction, the spaces $\Omega^{\bullet;k+1,\alpha}_{\m{ACF};\beta;t}(\widehat{D}^t_\zeta;\m{APS})$ contain forms, that fibrewise restrict to $C^{k+1,\alpha}_{\m{AC};\beta}$-regular forms. Here AC-denotes asymptotically conically weighted Hölder spaces as defined in \cite{lockhart1985elliptic}. The vertical Dirac operator of $\widehat{D}^t_\zeta$, i.e. the vertical Hodge--de Rham operator, can be realised as semi-Fredholm map
\begin{align*}
    \widehat{D}_{\zeta;V} \colon\Omega^{\bullet;k+1,\alpha}_{\m{AC};\beta}(\nu_\zeta^{-1}(s))\rightarrow \Omega^{\bullet;k,\alpha}_{\m{AC};\beta-1}(\nu_\zeta^{-1}(s)).
\end{align*}
outside a meagre set of critical rates $\beta\in\mathcal{C}(\widehat{D}_0)$. The function space $\Omega^{\bullet;k+1,\alpha}_{\m{ACF};\beta}(N_\zeta)$ are well-behaved with respect to push-forwards along $\nu_\zeta$, and hence, we identify $\Omega^{\bullet;k+1,\alpha}_{\m{ACF};\beta}(N_\zeta)$ with the Hölder space of section of an infinite dimensional Banach bundle of vertically AC-Hölder continuos differential forms. Using this perspective, we are able to realise the vertical Hodge--de Rham operator as a smooth homomorphism of Banach bundles over $S$. Using the AC-version of the Hodge-decomposition of differential forms \cite{lockhart1987fredholm}, we are able to define the spaces 
\begin{align*}
    &\Omega^{\bullet;k+1,\alpha}_{t}(S,\mathcal{C}o\mathcal{I}^{\bullet;k+1,\alpha}_{\m{AC};\beta}(N_\zeta/S)),\\
    &\Omega^{\bullet;k+1,\alpha}_{t}(S,\mathcal{H}^{\bullet}_{\m{AC};\beta}(N_\zeta/S)),\\
    &\Omega^{\bullet;k,\alpha}_{t}(S,\mathcal{I}^{\bullet;k,\alpha}_{\m{AC};\beta-1}(N_\zeta/S))\\
    \und{0.2cm}&\Omega^{\bullet;k,\alpha}_{t}(S,\mathcal{C}o\mathcal{H}^{\bullet}_{\m{AC};\beta-1}(N_\zeta/S))
\end{align*}
of vertical differential forms that are in the coimage, kernel, image and cokernel of the vertical Hodge--de Rham. By AC elliptic regularity vertical harmonicity or coharmonicity implies vertical smoothness.
\begin{rem}
In the notation of \cite{majewskiDirac} we have isomorphisms
\begin{align*}
    \mathcal{K}_{\m{AC};\beta}(\pi_\zeta/\nu_\zeta)\cong\wedge^\bullet T^*S\otimes \mathcal{H}^\bullet_{\m{AC};\beta}(N_\zeta/S);
\end{align*}
analogously for the other bundles.
\end{rem}

Let in the following $\pi_{\mathcal{C}o\mathcal{I};\beta}$, $\varpi_{\mathcal{H};\beta}$, $\pi_{\mathcal{I};\beta-1}$ and $\varpi_{\mathcal{C}o\mathcal{H};\beta}$ denote projections onto the subbundles and let $\iota$ and $\pi$ be two real numbers that will be determined in Lemma \ref{valueofphi} and Theorem \ref{lineargluingHodgedeRham}.

\begin{defi}
Let $-\iota<\kappa<\pi$. We define the adiabatic norms adapted to the adiabatic family of Dirac operators $\widehat{D}^t_{\zeta}$ to be 
\begin{align}    
\label{adiabaticnormszeta}
    \left|\left|\Psi\right|\right|_{\mathfrak{D}^{k+1,\alpha}_{\m{ACF};\beta;t}}\coloneqq&\left|\left|\pi_{\mathcal{C}o\mathcal{I};\beta}\Psi\right|\right|_{C^{k+1,\alpha}_{\m{ACF};\beta;t}}+t^{-\kappa}\cdot \left|\left|\varpi_{\mathcal{K};\beta}\Psi\right|\right|_{C^{k+1,\alpha}_{t}}\\
    \left|\left|\Psi\right|\right|_{\mathfrak{C}^{k,\alpha}_{\m{ACF};\beta-1;t}}\coloneqq&\left|\left|\pi_{\mathcal{I};\beta-1}\Psi\right|\right|_{C^{k,\alpha}_{\m{ACF};\beta-1;t}}+t^{-\kappa}\cdot \left|\left|\varpi_{\mathcal{C}o\mathcal{K};\beta-1}\Psi\right|\right|_{C^{k,\alpha}_{t}}.
\end{align}
\end{defi}

\begin{defi}
We define the spaces
\begin{align*}
      \mathfrak{C}\Omega^{\bullet;k;\alpha}_{\m{ACF};\beta-1;t}(N_\zeta)\coloneqq&\left(\Omega^{\bullet;k;\alpha}_{\m{ACF};\beta-1;t}(N_\zeta),\left|\left|.\right|\right|_{\mathfrak{C}^{k,\alpha}_{\m{ACF};\beta;t}}\right)\\
      \mathfrak{D}\Omega^{\bullet;k+1,\alpha}_{\m{ACF};\beta;t}(\widehat{D}^t_{\zeta};\m{APS})\coloneqq&\left(\m{dom}_{  \mathfrak{D}^{k+1,\alpha}_{\m{ACF};\beta;t}}(\widehat{D}^t_{\zeta};\m{APS}),\left|\left|.\right|\right|_{  \mathfrak{D}^{k+1,\alpha}_{\m{ACF};\beta;t}(\widehat{D}^t_{\zeta})}\right)
\end{align*}
where
\begin{align*}
    \left|\left|\Psi\right|\right|_{  \mathfrak{D}^{k+1,\alpha}_{\m{ACF};\beta;t}(\widehat{D}^t_{\zeta})}=\left|\left|\Psi\right|\right|_{\mathfrak{D}^{k+1,\alpha}_{\m{ACF};\beta;t}}+\left|\left|\widehat{D}^t_{\zeta}\Psi\right|\right|_{\mathfrak{C}^{k,\alpha}_{\m{ACF};\beta-1;t}}.
\end{align*}
denotes the graph norm.
\end{defi}

Finally, we are able to construct the adapted function spaces on $X_\zeta$ which will be used in the construction of the torsion-free $\m{Spin}(7)$-structure in section \ref{Existence of torsion-free Spin(7) Structures on Orbifold Resolutions}.

\begin{nota}
    In the following we will need to vary the width of the tubular neighbourhood of the singular stratum. Hence, we will set 
\begin{align*}
    \epsilon\sim t^\lambda
\end{align*}
for a $1>\lambda\geq 0$.
\end{nota}

\begin{defi}
    We define the function spaces 
\begin{align*}
    \mathfrak{D}\Omega^{\bullet;k+1,\alpha}_{\beta;t}(X_{\zeta;t})\und{1.0cm}\mathfrak{C}\Omega^{\bullet;k,\alpha}_{\beta-1;t}(X_{\zeta;t})
\end{align*}
by completing $\Omega^{\bullet}(X_{\zeta;t})$  with respect to the norms
    \begin{align*}
        \left|\left|\Psi\right|\right|_{\mathfrak{D}^{k+1,\alpha}_{\beta;t}}=&\left|\left|\delta^*_t\left((1-\chi_4)\cdot\Psi\right)\right|\right|_{\mathfrak{D}^{k+1,\alpha}_{\m{ACF};\beta;t}}+\left|\left|\chi_2\cdot\Psi\right|\right|_{C^{k+1,\alpha}_{\m{CFS};\beta;\epsilon}}\\
\left|\left|\Psi\right|\right|_{\mathfrak{C}^{k,\alpha}_{\beta-1;t}}=&\left|\left|\delta^*_t\left((1-\chi_4)\cdot\Psi\right)\right|\right|_{\mathfrak{C}^{k,\alpha}_{\m{ACF};\beta-1;t}}+\left|\left|\chi_2\cdot\Psi\right|\right|_{C^{k,\alpha}_{\m{CFS};\beta-1;\epsilon}}.
    \end{align*}
\end{defi}

\subsection{Vertically Harmonic Forms}
\label{Vertically Harmonic Forms}

The following Lemma proves that the vertical harmonic bundles exist, the fibres are given by smooth elements that decay at least by a power of $1-m$, i.e. 
\begin{align*}
    \mathcal{H}_{\m{AC};\beta}(N_\zeta/S)=\mathcal{H}_{L^2}(N_\zeta/S)
\end{align*}
for $\beta\in(1-m,0)$.

\begin{lem}
\label{adiabaticmircokernels}
The kernel of
\begin{align*}
    \widehat{D}_{\zeta;V}(s)\colon   \Omega^{\bullet;k+1,\alpha}_{\m{AC};\beta}(\nu_\zeta^{-1}(s))\rightarrow  \Omega^{\bullet;k,\alpha}_{\m{AC};\beta-1}(\nu_\zeta^{-1}(s))
\end{align*}
independent of $0<\alpha<1$, $k$ and $\beta\in[1-m,0)$. In particular, when $\beta<-m/2$ then $\mathcal{H}_{\m{AC};\beta}\cong\mathcal{H}_{L^2}$ by \cite{lockhart1985elliptic}.
\end{lem}

By the elliptic regularity of uniform elliptic asymptotically conical operator \cite{marshal2002deformations}, the first part of the statement holds. To prove the independence in $\beta$ we need the following result.

\begin{prop}
\label{diracspinorsthatdecaydecayfast}
Let $\Psi\in \Omega^{\bullet}(\nu^{-1}_\zeta(s))$. If $\widehat{D}_{\zeta;V}\Psi=0$ and $\Psi=o(1)$ then $\nabla^k\Psi=\mathcal{O}(r^{1-m-k})$.
\end{prop}

\begin{proof}
We will prove this statement similar to the proof of Proposition 5.10 in \cite{walpuski2012G2}. Notice that $\widehat{D}_{\zeta;V}$ is a Dirac operator and hence, the Kato inequality
\begin{align}
\label{Katoinequalty}
    |\m{d}|\Psi||\leq\sqrt{\frac{m-1}{m}}|\nabla^{g_{\zeta;V}}\Psi|
\end{align}
holds on the complement of the vanishing locus of $\Psi$. By using the Weitzen-böck formula 
\begin{align*}
    (\widehat{D}_{\zeta;V})^2=(\nabla^{g_{\zeta;V}})^*\nabla^{g_{\zeta;V}}+\mathcal{R}_{g_{V;\zeta}},
\end{align*}
$|F^{tw.}_{\nabla^{g_{\zeta;V}}}|=\mathcal{O}(r^{-m})$, and the identity
\begin{align*}
    \Delta^{g_{V,\zeta}}|\Psi|^2+2|\nabla^{g_{\zeta;V}}\Psi|^2=2\left<(\nabla^{g_{\zeta;V}})^*\nabla^{g_{\zeta;V}}\Psi,\Psi\right>
\end{align*}
we obtain the following estimate on the complement of the vanishing locus of $\Psi$
\begin{align*}
    \frac{2(m-1)}{m-2} \Delta^{g_{\zeta;V}}|\Psi|^{\frac{m-2}{m-1}}\leq&|\Psi|^{-\frac{m+1}{m}}\left(\Delta^{g_{\zeta;V}}|\Psi|^{2}+\frac{2m}{m-1}|\m{d}|\Psi||^2\right)\\
    \leq&|\Psi|^{-\frac{m}{m-1}}\left(\Delta^{g_{\zeta;V}}|\Psi|^{2}+2|\nabla^{g_{\zeta;V}}\Psi|^2\right)\\
    =&2|\Psi|^{-\frac{m}{m-1}}\left<(\nabla^{g_{\zeta;V}})^*\nabla^{g_{\zeta;V}}\Psi,\Psi\right>\\
    \leq&2|\Psi|^{-\frac{m}{m-1}}\left(\left<(\widehat{D}_{\zeta;V})^*\widehat{D}_{\zeta;V}\Psi,\Psi\right>+\left<\mathcal{R}_{g_{V;\zeta}}\Psi,\Psi\right>\right)\\
    \leq&\mathcal{O}(r^{-m})|\Psi|^{\frac{m-2}{m-1}}
\end{align*}
Let $f=|\Psi|^{\frac{m-2}{m-1}}$ on $(V_\Psi)^c=\{m\in \nu_\zeta^{-1}(s)|\Psi(m)\neq0\}$. By the above
\begin{align*}
    \Delta^{g_{V,\zeta}}f\lesssim f\cdot w_{\m{AC};m}.
\end{align*}
Since $f$ is bounded by \cite[Theorem 8.3.6(a)]{joyce2000compact}\footnote{Here we used a Sobolev embedding of $W^{2,p}\hookrightarrow C^{0,\alpha}$ for big enough $p>1$.}, there exists a $g=\mathcal{O}(r^{2-m})$ such that 
\begin{align*}
    \Delta g=\left\{\begin{array}{ll}
         (\Delta^{g_{V,\zeta}} f)^+& \text{on } (V_\Psi)^c\\
         0& V_\Psi
    \end{array}\right.
\end{align*}
where $(.)^+$ denotes the positive part. As $g$ is superharmonic and decays at infinity, the maximum principle implies its nonnegativity. Further, the function $f-g$ is subharmonic on $(V_\Psi)^c$, decays at infinity and is nonnegative on the boundary of $(V_\Psi)^c$. Again, using the maximum principle $f\leq g=\mathcal{O}(r^{2-m})$ and we deduce that
\begin{align*}
    |\Psi|=\mathcal{O}(r^{1-m}).
\end{align*}
\end{proof}

\begin{defi}
    Define the bundle of conically $\beta$-\textbf{logarithmic solutions} of $\widehat{D}_{0;V}$ by
\begin{align*}
        \mathcal{L}_{\m{C};\beta}(\widehat{D}_{0;V})\coloneqq\left\{\Phi=\left.\sum^n_{i=0}\m{log}^i(r)\Phi_i\in \Gamma_{\m{C};\beta}(\nu_0^{-1}(s),\widehat{E}_0)\right|\widehat{D}_{0;V}\Phi=0,\Phi_i\text{ is }\beta-\text{homogen.}\right\}
    \end{align*}
\end{defi}

\begin{rem}
From now on we assume that there exists a critical rate $\mu_{\mathcal{H}}\in(-\infty,1-m]$ such that the kernels 
\begin{align*}
    \widehat{D}_{\zeta;V}(s)\colon   \Omega^{\bullet;k+1,\alpha}_{\m{AC};\beta}(\nu_\zeta^{-1}(s))\rightarrow  \Omega^{\bullet;k,\alpha}_{\m{AC};\beta-1}(\nu_\zeta^{-1}(s))
\end{align*}
are isomorphic for all $\beta\in[\mu_{\mathcal{H}},0)$. This implies that for $\beta_2>\beta_1>\mu_{\mathcal{H}}$ there exists an exact sequence  
\begin{equation*}
    \begin{tikzcd}
	\mathcal{C}o\mathcal{H}_{\m{AC};\beta_1}(N_\zeta/S) &  \mathcal{C}o\mathcal{H}_{\m{AC};\beta_2}(N_\zeta/S) & {\bigoplus_{\beta_1<\beta<\beta_2}\mathcal{L}_{\m{C},\beta}(\widehat{D}_{0;V})}  
	\arrow[from=1-1, hook, to=1-2]
	\arrow[from=1-2,two heads, to=1-3]
\end{tikzcd}
\end{equation*}
Furthermore, there exists an isomorphism of vector bundles
\begin{align*}
    \mathcal{L}_{\m{C};\beta}(\widehat{D}_{0;V})\cong\mathcal{E}_{\beta+\frac{m-1}{2}-\delta(\widehat{E}_0)},
\end{align*}
where $\mathcal{E}_\lambda$ is the bundle $\wedge^\bullet T^*S$ twisted by the vertical $\lambda$-eigenspace bundle of the Hodge--de Rham operator on the normal cone.
\end{rem}

\begin{lem}
\label{imporveddecay}
Let $m=4$, $\Gamma\subset\m{SU}(2)$ and $(N_\zeta,g_{\zeta;V})$ be fibrewise ALE.
Then $\mathcal{H}_{\m{AC};\beta}(N_\zeta/S)\cong \mathcal{H}^2_{-}(N_\zeta/S)$ and $\mu_{\mathcal{H}}=-4$. 
\end{lem}

\begin{proof}

In some cases we can even further improve the Kato inequality \eqref{Katoinequalty}. The only vertical harmonic forms on $N_\zeta$ are anti-selfdual two forms. Following the results of Walter Seaman \cite{seaman1991} we improve the Kato inequality to 
\begin{align}
\label{improvedKatoinequalty}
    |\m{d}|\Psi||\leq\sqrt{\frac{2}{3}}|\nabla^{g_{\zeta;V}}\Psi|.
\end{align}
This is achieved by a closer inspection of the Weitzenböck-curvature. Then by following the arguments of the proof of Proposition \ref{diracspinorsthatdecaydecayfast} with $|\Psi|^{2/3}$ being replaced by $|\Psi|^{1/2}$ we deduce the statement. Results of \cite{lockhart1987fredholm} and \cite{joyce2000compact} ensure the existence of an isomorphism of flat vector bundles
\begin{align*}
    \underline{\mathbb{R}}\oplus\mathcal{H}^\bullet(N_\zeta/S)\cong\m{H}^\bullet(N_\zeta/S).
\end{align*}  
\end{proof}

\begin{lem}
\label{bundlesexist}
    The bundles $\mathcal{H}_{\m{AC};\beta}(N_\zeta/S)$ and $\mathcal{C}o\mathcal{H}_{\m{AC};\beta}(N_\zeta/S)$ always exist and coincide with 
    \begin{align*}
        \mathcal{H}_{\m{AC};\beta}(N_\zeta/S)=\zeta^*\mathcal{H}_{\m{AC};\beta}(\mathfrak{M}/\mathfrak{P})
    \end{align*}
    and
    \begin{align*}
        \mathcal{C}o\mathcal{H}_{\m{AC};\beta}(N_\zeta/S)=\zeta^*\mathcal{C}o\mathcal{H}_{\m{AC};\beta}(\mathfrak{M}/\mathfrak{P}).
    \end{align*}
\end{lem}

\begin{proof}
    The existence of the bundles of vertically AC-harmonic and AC-coharmonic forms on $\mathfrak{M}$ follows directly from the algebraic construction of $\mathfrak{M}$.
\end{proof}

\begin{ex}
    Let $N_\zeta$ be as above with $\Gamma=\mathbb{Z}_2$. Then in \cite[Thm. 3.26]{platt} it was shown that $\mu_{\mathcal{K}}=-4$ and $\mu_{\mathcal{C}o\mathcal{K}}=-2$ and
    \begin{align*}
        rk(\mathcal{C}o\mathcal{K}_{\m{AC};\beta-1})=\left\{\begin{array}{ll}
             0,&\text{if } \beta\in(0,-2) \\
             6,&\text{if } \beta\in(-2,-4)
        \end{array}\right.
    \end{align*}
\end{ex}

\subsection{Uniform Elliptic Theory}
\label{Uniform Ellitpic Theory}

The following sections will be devoted on sketching the main ingredients needed to state the main theorems of the author's work in \cite{majewskiDirac}.\\

The bundles of vertically harmonic and coharmonic bundles are finite dimensional Dirac bundles\footnote{For $\mathcal{C}o\mathcal{H}_{\m{AC};\beta}$ we need $\beta\notin\mathcal{C}(\widehat{D}_0)$.}. The operator $\widehat{D}^t_\zeta$ restricts to a family of elliptic operators on $S$ of Dirac type. These families $\mathfrak{D}^t_{\mathcal{K};\beta}$ and $\mathfrak{D}^t_{\mathcal{C}o\mathcal{K};\beta-1}$ admit well-defined adiabatic limits
\begin{align*}
    \mathfrak{D}_{\mathcal{K};\beta}\coloneqq&\lim_{t\to0}\mathfrak{D}^t_{\mathcal{K};\beta}\\
    \mathfrak{D}_{\mathcal{C}o\mathcal{K};\beta-1}\coloneqq&\lim_{t\to0}\mathfrak{D}^t_{\mathcal{C}o\mathcal{K};\beta-1}
\end{align*}
which will be referred to as the adiabatic residues of $\widehat{D}^t_\zeta$.

\begin{defi}
We will refer to the space 
\begin{align*}
    \mathfrak{Ker}_{\m{ACF};\beta}(\widehat{D}^0_\zeta)=\m{ker}\left(\mathfrak{D}_{\mathcal{K};\beta}\right)\und{.5cm}\mathfrak{CoKer}_{\m{ACF};\beta-1}(\widehat{D}^0_\zeta)=\m{coker}\left(\mathfrak{D}_{\mathcal{C}o\mathcal{K};\beta-1}\right).
\end{align*}
as the adiabatic kernel and adiabatic cokernel.
\end{defi}

\begin{defi}
    A resolution $\rho_{t^2\cdot\zeta}\colon (N_{t^2\cdot\zeta},\Phi_{t^2\cdot\zeta},g_{t^2\cdot\zeta})\dashrightarrow (N_0,\Phi_0,g_0)$ is called isentropic if 
    \begin{align*}
        \m{H}^{\bullet}(N_\zeta)\cong\mathfrak{Ker}_{\m{ACF};\beta}(\widehat{D}^0_\zeta)\oplus\m{H}^\bullet(S).
    \end{align*}
\end{defi}
    An isentropic process in statistical mechanics is an adiabatic process that is reversible. In the same manner an isentropic resolution can be modelled by its adiabatic limit.\\

The adiabatic residues have a geometric interpretation. Let $(N_\zeta,g^t_\zeta)$ satisfy 
\begin{align}
    \m{H}^\bullet(N_\zeta/S)\cong\underline{\mathbb{R}}\oplus\mathcal{H}^\bullet(N_\zeta/S).
\end{align}
The adiabatic residue $\mathfrak{D}_{\mathcal{K};\beta}=\mathfrak{D}_{\mathcal{K};L^2}$ of the Hodge--de Rham operator on $(N_\zeta,g_\zeta)$ is given by the Gauß-Manin-Hodge--de Rham operator
\begin{align*}
    \mathfrak{D}_{GM} \colon\bigoplus_{p+q=\bullet}\Omega^p\left(S, \mathcal{H}^q(N_\zeta/S)\right)\rightarrow\bigoplus_{p+q=\bullet}\Omega^p\left(S, \mathcal{H}^q(N_\zeta/S)\right)
\end{align*}
given by the Hodge--de Rham operator twisted by the flat Gauß-Manin connection on $\mathcal{H}^\bullet(N_\zeta/S)\rightarrow S$. In particular, this operator can be seen a folding of the elliptic complex induced by the local system
\begin{align*}
    L^2GM^\bullet(\nu_\zeta\colon N_\zeta\rightarrow S)=(\Omega^{\bullet}(S,\mathcal{H}^\bullet(N_\zeta/S)),\m{d}_{GM}).
\end{align*}

By classical results of algebraic topology, we can identify the adiabatic kernel with the second page of the Leray-Serre-spectral sequence. 

\begin{thm}[\cite{bott1982differential}]
There exists a spectral sequence $E^{p,q}_N$ associated to the double complex
\begin{align*}
    (K^{\bullet_1,\bullet_2},\m{d})=(\check{C}^{\bullet_1}(\nu_\zeta^{-1}\mathcal{U},\Omega^{\bullet_2}),\delta+(-1)^{\bullet_1}\m{d}_{dR})
\end{align*}
converging to the cohomology of $N_\zeta$, i.e. $\tot{p+q=\bullet}{}{E^{p,q}_\infty}=\m{H}^\bullet(N_\zeta)$. Here $\mathcal{U}$ is a good cover of $S$. Moreover, the second page of this spectral sequence can be identified with 
\begin{align*}
    (\check{\m{H}}^{\bullet_1}(\nu_\zeta^{-1}\mathcal{U},\m{H}^{\bullet_2}(N_\zeta/S)),\m{d}^{p,q}_2).
\end{align*}
\end{thm}

\begin{prop}
\label{L2E2}
Let us assume that $\Gamma\subset\m{SU}(m/2)$ and  $\m{H}^\bullet(N_\zeta/S)\cong\underline{\mathbb{R}}\oplus\mathcal{H}^\bullet(N_\zeta/S)$. There exists an isomorphism 
\begin{align*}
    E^{p,q}_2(N_\zeta)\cong\left\{\begin{array}{ll}
         \m{H}^p(S)&\text{if } q=0 \\
         \m{H}^p(S,\mathcal{H}^q(N_\zeta/S))&\text{else}
    \end{array}\right.
\end{align*}
i.e. 
\begin{align*}
     E^{\bullet,\bullet}_2(N_\zeta)\cong\m{H}^\bullet(S)\oplus\m{H}^\bullet(L^2GM^\bullet(\nu_\zeta\colon N_\zeta\rightarrow S)).
\end{align*}
Moreover, for $0<N$ the differentials $d_{2N}=0$. In particular,
    \begin{align*}
        \mathfrak{Ker}_{\m{ACF};\beta}(\widehat{D}^0_\zeta)\cong L^2E^{\bullet,\bullet}_2(N_\zeta).
    \end{align*}
\end{prop}

\begin{nota}
    We will denote by 
    \begin{align*}
\beta_i(N_\zeta)\coloneqq \sum_{p+q=i}\m{dim}(E^{p,q}_2(N_\zeta))-b_i(S).
    \end{align*}    
\end{nota}

\begin{cor}
\label{collapsingss} 
    A $\rho_{t^2\cdot\zeta}\colon (N_{t^2\cdot\zeta},\Phi_{t^2\cdot\zeta},g_{t^2\cdot\zeta})\dashrightarrow (N_0,\Phi_0,g_0)$ is isentropic, if and only if the spectral sequence $E^{\bullet,\bullet}_{N}$ collapses on the second page.
\end{cor}

\begin{defi}
    We define the approximate kernel and cokernel of $\widetilde{D}^{pre}_{\zeta;t}$ by\\
    
\begin{minipage}{\linewidth}
    \begin{align*}
        \m{xker}_{\beta}(\widetilde{D}^{pre}_{\zeta;t})=\begin{array}{c}
             \m{ker}\left(\mathfrak{D}_{\mathcal{K};\beta}\colon  \Omega^{\bullet;k+1,\alpha}_{t}\left(S,\mathcal{H}^\bullet_{\m{AC};\beta}(N_\zeta/S)\right)\rightarrow  \Omega^{\bullet;k,\alpha}_{t}\left(S, \mathcal{H}^\bullet_{\m{AC};\beta}(N_\zeta/S)\right)\right)\\
             \oplus\\
             \m{ker}\left(D \colon\Omega^{\bullet;k+1,\alpha}_{\m{CFS};\beta;\epsilon}(X;\m{APS})\rightarrow \Omega^{\bullet;k,\alpha}_{\m{CFS};\beta-1;\epsilon}(X)\right)
        \end{array}
    \end{align*}
        \end{minipage}
    and
    
\resizebox{1.0\linewidth}{!}{
\begin{minipage}{\linewidth}
    \begin{align*}
        \m{xcoker}_{\beta-1}(\widetilde{D}^{pre}_{\zeta;t})=\begin{array}{c}
             \m{coker}\left(\mathfrak{D}_{\mathcal{C}o\mathcal{K};\beta-1}\colon  \Omega^{\bullet;k+1,\alpha}_{t}\left(S,\mathcal{C}o\mathcal{H}^\bullet_{\m{AC};\beta}(N_\zeta/S)\right)\rightarrow  \Omega^{\bullet;k,\alpha}_{t}\left(S,\mathcal{C}o \mathcal{H}^\bullet_{\m{AC};\beta}(N_\zeta/S)\right)\right)\\
             \oplus\\
             \m{coker}\left(D \colon\Omega^{\bullet;k+1,\alpha}_{\m{CFS};\beta;\epsilon}(X;\m{APS})\rightarrow \Omega^{\bullet;k,\alpha}_{\m{CFS};\beta-1;\epsilon}(X)\right).
        \end{array}
    \end{align*}
    \end{minipage}}\\
    
\end{defi}

Before we state the main result of \cite{majewskiDirac} we need to prove the following technical Lemma. This Lemma essentially ensures that the adiabatic limit of the operator $\widehat{D}^t_\zeta$ is \say{well-behaved}, i.e. that a uniform elliptic estimate holds on the complement of the adiabatic kernel.

\begin{lem}
\label{valueofphi}
    Given $\Psi\in \Omega^{\bullet;k,\alpha}_{\m{ACF};\beta-1}(N_\zeta)$, we can bound
    \begin{align*}
        \left|\left|L_{\m{AC};-\beta-m-1}\{\widehat{D}_{\zeta;V};\widehat{D}^t_{\zeta;H}\}\pi_{\mathcal{K};\beta}\Psi\right|\right|_{C^{k,\alpha}_{\m{ACF};\beta-1}}\lesssim t^\phi\cdot\left|\left|\pi_{\mathcal{K};\beta}\Psi\right|\right|_{C^{k+1,\alpha}_{\m{ACF};\beta}}.
    \end{align*}
for $\phi\geq 1$.
\end{lem}

\begin{proof}
    Notice, that the de Rham differential $\m{d}=\m{d}^{1,0}+\m{d}^{0,1}+\m{d}^{2,-1}$ on a fibration satisfies the identities 
    \begin{align*}
        \{\m{d}^{1,0},\m{d}^{1,0}\}=&-2\{\m{d}^{0,1},\m{d}^{2,-1}\},\hspace{0.3cm}\{\m{d}^{0,1},\m{d}^{0,1}\}=0,\hspace{0.3cm}\{\m{d}^{2,-1},\m{d}^{2,-1}\}=0\\
        \{\m{d}^{1,0},\m{d}^{0,1}\}=&0,\hspace{0.3cm}\{\m{d}^{1,0},\m{d}^{2,-1}\}=0.
    \end{align*}
    The same identities hold for the codifferential $\m{d}^*=\m{d}^{-1,0}+\m{d}^{0,-1}+\m{d}^{-2,1}$. The statement follows from a direct computation for $L_{\m{AC};-\beta-m-1}\{\widehat{D}_{\zeta;V};\widehat{D}^t_{\zeta;H}\}$.
\end{proof}

Before we state the main result of \cite{majewskiDirac} we need to prove the following technical Lemma. This Lemma essentially ensures that the adiabatic limit of the operator $\widehat{D}^t_\zeta$ is \say{well-behaved}, i.e. that a uniform elliptic estimate holds on the complement of the adiabatic kernel.

\begin{thm}\cite[Thm. 9.1]{majewskiDirac}
\label{lineargluingHodgedeRham}
Let $\beta\notin\mathcal{C}(\widehat{D}_0)$ and 
\begin{align*}
        \iota\geq&\phi-m/2-\beta\und{1.0cm}\pi\geq\phi+ m/2+\beta-1-\alpha
\end{align*}
whereby $\phi\geq 1$. The map 
\begin{align}
    D^{pre}_{\zeta;t} \colon \mathfrak{D}\Omega^{\bullet;k+1,\alpha}_{\beta;t}(X_{\zeta;t})\rightarrow\mathfrak{C}\Omega^{\bullet;k,\alpha}_{\beta-1;t}(X_{\zeta;t})
\end{align}
is a bounded Fredholm map, whose kernel and cokernel satisfy

\begin{equation}
\label{lineargluingexactsequenceHodgedeRham}
\adjustbox{scale=0.9,center}{
    \begin{tikzcd}
        \mathcal{H}^\bullet(X_{\zeta;t},g^{pre}_{\zeta;t})\arrow[r,"i_{\beta;t}",hook]&
             \m{xker}_{\beta}(D^{pre}_{\zeta;t})\arrow[r,"\m{ob}_{\beta;t}"]&\m{xcoker}_{\beta-1}(D^{pre}_{\zeta;t})\arrow[r,"p_{\beta;t}",two heads]&\mathcal{H}^\bullet(X_{\zeta;t},g^{pre}_{\zeta;t}).
    \end{tikzcd}}
\end{equation}
Furthermore, the uniform elliptic estimate 
    \begin{align}
        \label{unifromDiracbelow}\left|\left|\Psi\right|\right|_{\mathfrak{D}^{k+1,\alpha}_{\beta;t}}\lesssim \left|\left|D^{pre}_{\zeta;t}\Psi\right|\right|_{\mathfrak{C}^{k,\alpha}_{\beta-1;t}}
    \end{align}
    holds for all $\Psi\in\m{xker}_\beta(D_{\zeta;t}^{pre})^\perp$. Hence, if $\m{ob}_{\beta;t}=0$ there exists uniform bounded map
\begin{align}
    R^{pre}_{\zeta;t} \colon\mathfrak{C}\Omega^{\bullet;k,\alpha}_{\beta-1;t}(X_{\zeta;t},\wedge^\bullet T^*X_{\zeta;t})\rightarrow \mathfrak{D}\Omega^{\bullet;k+1,\alpha}_{\beta;t}(X_{\zeta;t})
\end{align}
that restricts to a right-inverse.
\end{thm} 

So far we have only established uniform bounds for the right-inverse of the Hodge--de Rham operator for the naive-preglued-metric. However, the next Proposition ensures that the results extend to the Hodge--de Rham operator associated $\m{Spin}(7)$-structure $\Phi^{pre}_{\zeta;t}$.

\begin{prop}\cite[Prop. 9.4]{majewskiDirac}
\label{tameuniformboundedness}
Let $\rho_t\colon(X_t,g_t)\dashrightarrow (X,g)$ be a $(k+1)$-tame smooth Gromov--Hausdorff resolution with respect to 
\begin{align*}
    \rho_{\zeta;t}\colon(X_{\zeta;t},g^{pre}_{\zeta;t})\dashrightarrow (X,g).
\end{align*}
The $L^2$-projection induces a natural identification
\begin{align*}
    \mathcal{H}^\bullet(X_{\zeta;t},g^{pre}_{\zeta;t})\cong\mathcal{H}^\bullet(X_{\zeta;t},g_t).
\end{align*}
and Hodge theory yields a further identification 
\begin{align*}
    \mathcal{H}^\bullet(X_{\zeta;t},g^{pre}_{\zeta;t})\cong\m{H}^\bullet(X_{\zeta;t})\cong \m{H}^\bullet(X)\oplus \m{H}^\bullet(N_\zeta)/\m{H}^\bullet(S).
\end{align*}
If moreover $\m{ob}_{\beta;t}=0$, then 
\begin{align*}
    \mathcal{H}^\bullet(X_{\zeta;t},g^{pre}_{\zeta;t})\cong \m{H}^\bullet(X)\oplus \m{H}^\bullet(S,\mathcal{H}^\bullet(N_\zeta/S))
\end{align*}
and $D_{t}$ admits a uniform bounded right-inverse, i.e. there exists a map Fredholm map
\begin{align}
    R_{t} \colon\mathfrak{C}\Omega^{\bullet;k,\alpha}_{\beta-1;t}(X_{\zeta;t})\rightarrow \mathfrak{D}\Omega^{\bullet;k+1,\alpha}_{\beta;t}(X_{\zeta;t})
\end{align}
satisfying 
\begin{align*}
    \left|\left|R_{t}\Psi\right|\right|_{\mathfrak{C}^{k+1,\alpha}_{\beta;t}}\lesssim \left|\left|\Psi\right|\right|_{\mathfrak{C}^{k,\alpha}_{\beta-1;t}}.
\end{align*}
\end{prop}

\begin{rem}
By the exactness of \eqref{lineargluingexactsequenceHodgedeRham} we know that an obstruction to $\m{ob}_{\beta;t}=0$ is the surjectivity of the map $\iota_{\beta;t}$. As $\iota_{\beta;t}$ is injective, it is surjective if and only if $\m{dim}\left(\mathcal{H}^\bullet(X_{\zeta;t},\Phi^{pre}_{\zeta;t})\right)=\m{dim}(\m{xker}_\beta(D^{pre}_{\zeta;t}))$. If $\beta<0$ is larger than the biggest negative critical rate, then by Proposition \ref{cohomofXt} and \eqref{lineargluingexactsequenceHodgedeRham} the former holds true if and only if the resolution $\rho_{t;\zeta}\colon (X_{\zeta,t},\Phi^{pre}_{\zeta;t})\dashrightarrow (X,\Phi)$ is isentropic. By Corollary \ref{collapsingss} this is equivalent to the collapsing of the Leray-Serre-spectral sequence on the second page.
\end{rem}

\section{Adiabatic Spin(7)-Orbifold Resolutions}
\label{Adiabatic Spin(7)-Orbifold Resolutions}

The goal of this section is to construct adiabatic smooth Gromov--Hausdorff resolutions of $\m{Spin}(7)$-structures. Building on the analytic framework developed in \cite{majewskiDirac}, we will identify the local data needed to resolve CF-$\m{Spin}(7)$ normal cones along the singular strata of a $\m{Spin}(7)$-orbifold, and construct suitable ACF-$\m{Spin}(7)$ spaces using techniques from symplectic and algebraic geometry. By interpolating between these ACF-$\m{Spin}(7)$-structures and the $\m{Spin}(7)$-orbifold structure, we construct a family of $\m{Spin}(7)$-structures converging to the $\m{Spin}(7)$-orbifold, whose torsion vanishes in the orbifold limit.


\subsection{Tamed Isentropic Smooth Gromov--Hausdorff Resolutions}
\label{Tamed Isentropic Smooth Gromov--Hausdorff Resolutions}

We begin by describing the geometric idea which motivates the analytic set-up of this chapter. Let $(X,\Phi)$ be a compact torsion-free $\m{Spin}(7)$-orbifold. Suppose that $(X,\Phi)$ arises as the Gromov--Hausdorff limit of a family of smooth torsion-free $\m{Spin}(7)$-manifolds
\begin{align*}
        \rho_{\zeta,t}\colon (X_{\zeta,t},\Phi_{\zeta,t})
        \dashrightarrow
        (X,\Phi),
\end{align*}
where $t>0$ denotes the resolution scale and $\zeta$ denotes a parameter defining a resolution of $X$. Away from the singular set of $X$, the maps
\begin{align*}
        \rho_{t}\colon X_{\zeta,t}\dashrightarrow X
\end{align*}
identify $X_{\zeta,t}$ with $X$ up to an error which tends to zero as $t\to 0$. Near the singular strata, the geometry is instead modelled on suitable resolutions of the corresponding normal cones depending $\zeta$.\\

The compactness theory of the moduli space $\mathfrak{Spin}(7)[X_{\zeta,t}]$ discussed in Section \ref{Compactifications of Moduli Spaces of Spin(7)-Manifolds and Orbifold-Degenerations}, in the spirit of Anderson, Cheeger, Naber and Jiang\cite{anderson1990convergence,cheeger1986collapsing,cheeger1990collapsing,cheegernaber2015regularity,jiang2020l2curvatureboundsmanifolds}, suggests that such a degeneration should be described by a smooth background resolution together with a family of adiabatic Ricci-flat, or equivalently adiabatic torsion-free, geometric structures. In the present setting we therefore do not try to construct the torsion-free structures directly. Instead, we first consider a family of smooth $\m{Spin}(7)$-structures
\begin{align*}
        \Phi^{pre}_{\zeta,t}\in \Gamma(X_{\zeta,t},Cay_+(X_{\zeta,t}))
\end{align*}
which interpolate between the orbifold structure $\Phi$ on $X$ and the chosen local resolution models, depending on $\zeta$, near the singular strata. These structures should be regarded as pre-glued $\m{Spin}(7)$-structures which encode the expected Gromov--Hausdorff geometry of the resolution.\\

The aim is then to find conditions on $\Phi^{pre}_{\zeta,t}$ which guarantee a posteriori the existence of the genuine torsion-free $\m{Spin}(7)$-structure $\Phi_{\zeta,t}$ on $X_{\zeta,t}$ which it tames\footnote{In the sense of Definition \ref{ktame}.} by $\Phi^{pre}_{\zeta,t}$. The existence theorem in Section \ref{Existence of torsion-free Spin(7) Structures on Orbifold Resolutions}
will be proved; if the pre-glued structure satisfies the required geometric conditions, it will satisfy uniform estimates, has sufficiently small torsion and in particular, then it can be uniquely deformed to a torsion-free $\m{Spin}(7)$-structure $\Phi_{\zeta,t}$.\\

Let $D_{\Phi^{pre}_{\zeta,t}}\coloneqq\m{d}+\m{d}^{*_{g_{\Phi^{pre}_{\zeta,t}}}}$ denote the Hodge--de Rham operator associated with the metric induced by $\Phi^{pre}_{\zeta,t}$ and assume that the resolution $(X_{\zeta,t},\Phi^{pre}_{\zeta,t})$ is isentropic, in the sense that the obstruction map $\m{ob}_{0,t}=0$ in the linear gluing sequence for $D_{\Phi^{pre}_{\zeta,t}}$ vanishes. This seems to be a rather technical assumption, but we will see in the following section, how this assumption is related to a conjecture on the string cohomology of the orbifold.\\ 

Since we assumed isentropicity, Proposition \ref{tameuniformboundedness} gives a canonical identification of finite-dimensional kernels
\begin{align*}
        \m{ker}(D_{\Phi^{pre}_{\zeta,t}})
        \cong
        \m{xker}_{0}(D^{pre}_{\zeta,t}).
\end{align*}
By the description of the extended kernel in the tame limit, this space is naturally identified with the fibre product
\begin{align*}
        \m{xker}_{0}(D^{pre}_{\zeta,t})
        \cong
        \mathfrak{Ker}_{\m{ACF};0}(\widehat{D}^{0}_{\zeta})
        \times_{\m{ker}_{\m{CF};0}(\widehat{D}_{0})}
        \m{ker}_{\m{CFS};0}(D_{\Phi}).
\end{align*}
Here $\m{ker}_{\m{CFS};0}(D_{\Phi})$ is the contribution from the orbifold limit, $\mathfrak{Ker}_{\m{ACF};0}(\widehat{D}^{0}_{\zeta})$ is the contribution from the resolved normal models, and the fibre product imposes the matching condition along the common conically fibred limit.\\

Now suppose that $\Phi_{\zeta,t}$ is a torsion-free $\m{Spin}(7)$-structure on $X_{\zeta,t}$ which is tamed by $\Phi^{pre}_{\zeta,t}$. Since torsion-freeness is equivalent to
\begin{align*}
        D_{\Phi_{\zeta,t}}\Phi_{\zeta,t}=0,
\end{align*}
the form $\Phi_{\zeta,t}$ defines an element
\begin{align*}
        \Phi_{\zeta,t}
        \in
        \m{ker}(D_{\Phi_{\zeta,t}})
        \cap
        \Gamma(X_{\zeta,t},Cay_+(X_{\zeta,t})).
\end{align*}
Using the tame identification of kernels, and the fact that $\Phi_{\zeta,t}$ is hence a small deformation of $\Phi^{pre}_{\zeta,t}$, we may identify $\Phi_{\zeta,t}$ with an element
\begin{align*}
        (\Phi_{\zeta}\times_{\Phi_0}\Phi)
        \in
        \mathfrak{Ker}_{\m{ACF};0}(\widehat{D}^{0}_{\zeta})
        \times_{\m{ker}_{\m{CF};0}(\widehat{D}_{0})}
        \m{ker}_{\m{CFS};0}(D_{\Phi}).
\end{align*}
The second component is the original orbifold $\m{Spin}(7)$-structure $\Phi$ on $X$. The first component is an asymptotically conically fibred $\m{Spin}(7)$-structure on the resolved normal model,
\begin{align*}
        \Phi_{\zeta}
        \in
        \mathfrak{Ker}_{\m{ACF};0}(\widehat{D}^{0}_{\zeta})
        \cap
        \Gamma(N_{\zeta},Cay_+(N_{\zeta})),
\end{align*}
which resolves the conical fibred $\m{Spin}(7)$-structure
\begin{align*}
        \Phi_{0}
        \in
        \m{ker}_{\m{CF};0}(\widehat{D}_{0})
        \cap
        \Gamma(N_{0},Cay_+(N_{0})).
\end{align*}
The fibre-product condition precisely says that $\Phi_{\zeta}$ and $\Phi$ have the same asymptotic normal cone $\Phi_0$.\\

Consequently, we argue a posteriori that in order to construct a tame $\m{Spin}(7)$-orbifold resolution we will need to construct an ACF-$\m{Spin}(7)$-structure satisfying
\begin{empheq}[box=\colorbox{blue!10}]{align}
\label{adiabaticACFSpin7}
        \Phi_\zeta
        \in
        \mathfrak{Ker}_{\m{ACF};0}(\widehat{D}^{0}_\zeta)
        \cap
        \Gamma(N_\zeta,Cay_+(N_\zeta)),
        \qquad
        \Phi_\zeta\asymp \Phi_0 .
\end{empheq}

\subsection{Spin(7)-Normal Cones}
\label{Spin(7)-Normal Cones}
Let $(X,\Phi)$ be a $\m{Spin}(7)$-orbifold and $S\subset X$ a singular stratum. Recall, that the normal cone structure $\Phi_0$ of $\Phi$ with respect to $S$ is the CF-$\m{Spin}(7)$-structure on the normal cone $\nu_0\colon N_0\rightarrow S$ induced by the restriction of $\Phi$ to $S$. Given a tubular neighbourhood $ \m{exp}_{g_\Phi}\colon\m{Tub}_{5\epsilon}(S)\hookrightarrow X$, we can identify $\Phi_0$ as the leading order term, in an expansion 
\begin{align*}
    \m{exp}_{g_\Phi}^*\Phi=\Phi_0+\Phi_{hot}\und{1.0cm}|\Phi_{hot}|_{g_{\Phi_0}}=\mathcal{O}(r)
\end{align*}
in terms of the distance $r$ to the singular stratum. In the following section we will analyse the structure $\Phi_0$ in more detail. We will show that it defines an element in 
\begin{align*}
    \Phi_0\in \m{ker}_{CF,0}(\widehat{D}_0)\cap \Gamma(N_0,Cay_+(N_0))
\end{align*}
and relate this behaviour to Donaldson's notion of adiabatic torsion-free, fibred, exceptional holonomy metrics \cite{donaldson2016adiabaticlimitscoassociativekovalevlefschetz}.

\begin{lem}
\label{expansionOmegaProperties}
    Let $\m{exp}_{g_\Phi} \colon\m{Tub}_{5\epsilon}(S)\hookrightarrow (X,\Phi)$ be the exponential neighbourhood of $S$ with respect to $g_\Phi$ and let $\Phi$ be torsion-free. The torsion-free $\m{Spin}(7)$-structure and the corresponding Riemannian structure $g_\Phi$ expands as
    \begin{align*}
        \m{exp}_{g_\Phi}^*\Phi&=\Phi_0+\Phi_{hot}&\und{1.0cm}|\Phi_{hot}|_{g_0}&=\mathcal{O}(r^2)\\
        \m{exp}_{g_\Phi}^*g_\Phi&=g_0+g_{hot}&\und{1.0cm}|g_{hot}|_{g_0}&=\mathcal{O}(r^2).
    \end{align*}
    Moreover, 
    \begin{align}
    \label{dOmegahot}
        \nabla^{\oplus}\Phi_0=0\und{1.0cm}
        \m{d}\Phi_{hot}=-\m{d}^{2,-1}\Phi_0.
    \end{align}
\end{lem}

\begin{proof}
    The first property follows from the fact that the singular strata of a $\m{Spin}(7)$-orbifold are locally the fixed point sets of isometries and hence totally geodesic, i.e. $\m{II}_{S,g_\Phi}=0$.
    Relative to the singular stratum $S$ the $\m{Spin}(7)$-structure $\Phi$ expands as 
    \begin{align*}
        \m{exp}_{g_\Phi}^*\Phi=\Phi_0+\Phi_1+\mathcal{O}(r^2),
    \end{align*}
    where $\Phi_1$ on $V_i\in \mathfrak{X}^{1,0}(N_0)$ and $W\in \mathfrak{X}^{0,1}(N_0)$ is given by 
    \begin{align*}
        \left<\Phi_1(V_1,...,V_4),W\right>=-\sum_{i=1}^4\Phi(V_1,...,A_{g_\Phi,S}(V_i,W),...,V_4),
    \end{align*}
    where $\left<A_{g_\Phi,S}(W)V_i,V_i\right>=\left<\m{II}_{g_\Phi,S}(V_i,V_j),W\right>$ denote the Weingarten map of $S$ with respect to $g_\Phi$. Further, we used that terms of the form $\Phi(V_1,..,W,..,V_4)$ vanish for all fixed-point sets of $\m{Spin}(7)$-isometries.\\
    
    Furthermore, the Levi-Civita connection $\nabla^{g_\Phi}$ on $X$ induces a connection $\nabla^{\oplus_0}$ on $N_0$ that is metric with respect to $g_0$ and whose torsion can be identified with the curvature of the Ehresmann connection $H_0$. Since the singular strata of Riemannian orbifolds are totally geodesic, as they are locally fixed point sets of isometries the Levi-Civita connection $\nabla^{g_\Phi}$ expands as 
\begin{align*}
    \nabla^{g_\Phi}=\nabla^{\oplus_0}+\mathcal{O}(r);
\end{align*}
the connection $\nabla^{\oplus_0}=\nu_0^*\nabla^{g_s}\oplus \nabla^{\nu}$ is given by the direct sum connection of the Levi-Civita connection on $S$ lifted along $H_0$ and the fibrewise flat connection.
The torsion-free-ness of $\Phi$ implies that 
\begin{align*}
    \nabla^{g_\Phi}\Phi=\nabla^{\oplus_0}\Phi_0+\mathcal{O}(r)
\end{align*}
and hence that $\Phi_0$ is parallel with respect to $\nabla^{\oplus_0}$. The latter implies that the holonomy of the connection $\nabla^{\oplus_0}$ is contained in the intersection of $\m{Spin}(7)\cap \m{SO}(H)\times\m{SO}(V)$.
\end{proof}

We will continue by analysing the normal cone structures depending on the codimension of the singular stratum explicitly and express the normal cone structures in natural tensors on the normal cone bundle. 

Recall from \ref{Riemannian Orbifolds as Conically Fibred Singular Spaces} that there exists a canonical $\m{N}_{\m{Spin}(7)}$-reduction of the $\m{Spin}(7)$-frame bundle
\begin{align*}
    \phi\colon \m{Fr}(X/S,\Phi)\rightarrow S
\end{align*}
of frames normalising the action of the isotropy group $\Gamma$.\\

\textbf{Codimension Four:} Let $S\subset X^{sing}$ be a singular stratum of codimension four, isotropy $\Gamma\subset\m{SU}(2)\cong\m{Sp}(1)$. There exist representations 
\begin{align*}
    \varrho_{\mathbb{R}^4}\colon &\m{N}_{\m{Spin}(7)}(\Gamma)\rightarrow \m{SO}(\mathbb{R}^4)\\
    \varrho_{\mathbb{H}/\Gamma}\colon &\m{N}_{\m{Spin}(7)}(\Gamma)\rightarrow \m{N}_{\m{SO}(\mathbb{H})}(\Gamma)\\
    \varrho_{\m{Im}(\mathbb{H})}\colon &\m{N}_{\m{Spin}(7)}(\Gamma)\rightarrow \m{SO}(\m{Im}(\mathbb{H})
\end{align*}
such that 
\begin{align*}
    TS\cong&\m{Fr}(X/S,\Phi)\times_{\m{N}_{\m{Spin}(7)}(\Gamma)}\mathbb{R}^4\\
    N_0\cong& \m{Fr}(X/S,\Phi)\times_{\m{N}_{\m{Spin}(7)}(\Gamma)}\mathbb{H}/\Gamma\\
    \wedge^2_+T^* S\cong&\m{Fr}(X/S,\Phi) \times_{\m{N}_{\m{Spin}(7)}(\Gamma)}\m{Im}(\mathbb{H})^*.
\end{align*}
Let $\left[\theta_{S}\wedge\theta_S\right]_+\in\Omega^2_{hor,+}(\m{Fr}(X/S,\Phi),\m{Im}(\mathbb{H})^*)^{\m{N}_{\m{Spin}(7)}(\Gamma)}$ be the self-dual Soldering form corresponding to the isomorphism 
\begin{align*}
    \m{Fr}(X/S,\Phi)\times_{\m{N}_{\m{Spin}(7)}(\Gamma)}\m{Im}(\mathbb{H})^*\cong \wedge^2_+T^* S.
\end{align*}
The CF-$\m{Spin}(7)$-structure on the normal cone can be expressed as 
\begin{align*}
    \Phi_0=\Phi^{4,0}_0+\Phi^{2,2}_0+\Phi^{0,4}_0
\end{align*}
whereby lifted to $\m{Fr}(X/S,\Phi)\times \mathbb{H}/\Gamma$ the components can be expressed as
\begin{align*}
    \hat{\Phi}^{4,0}_0=\phi^*\m{vol}_{g_S},\hspace{0,5cm}\hat{\Phi}^{2,2}_0=-\m{tr}(\left[\theta_{S}\wedge\theta_S\right]_+\wedge\underline{\omega}_0)\und{0.5cm}\hat{\Phi}^{0,4}_0=\frac{1}{6}\m{tr}(\underline{\omega}_0\wedge\underline{\omega}_0).
\end{align*}
Here $\underline{\omega}_0$ denotes the standard $\m{SU}(2)$-structure on $\mathbb{H}$. Moreover, the CF-$\m{Spin}(7)$-structure satisfies 
\begin{align*}
    \m{d}\Phi_0=\m{d}^{2,-1}\Phi_0\Leftrightarrow \widehat{D}_{0}\Phi_0=0.
\end{align*}

\textbf{Codimension Six.} Let $S\subset X^{sing}$ be a singular stratum of codimension six, isotropy $\Gamma\subset\m{SU}(3)$. There exist representations 
\begin{align*}
    \varrho_{\mathbb{C}}\colon &\m{N}_{\m{Spin}(7)}(\Gamma)\rightarrow \m{SO}(\mathbb{C})\\
    \varrho_{\mathbb{C}^3/\Gamma}\colon &\m{N}_{\m{Spin}(7)}(\Gamma)\rightarrow \m{N}_{\m{SO}(\mathbb{C}^3)}(\Gamma)
\end{align*}
such that 
\begin{align*}
    TS\cong&\m{Fr}(X/S,\Phi)\times_{\m{N}_{\m{Spin}(7)}(\Gamma)}\mathbb{R}^2\\
    N_0\cong& \m{Fr}(X/S,\Phi)\times_{\m{N}_{\m{Spin}(7)}(\Gamma)}\mathbb{C}^3/\Gamma
\end{align*}
Let $\theta_{S}\in\Omega^1_{hor,+}(\m{Fr}(X/S,\Phi),\mathbb{C})^{\m{N}_{\m{Spin}(7)}(\Gamma)}$ be the Soldering form corresponding to the isomorphism $\m{Fr}(X/S,\Phi)\times_{\m{N}_{\m{Spin}(7)}(\Gamma)}\mathbb{C}^*\cong T^* S$. The CF-$\m{Spin}(7)$-structure on the normal cone can be expressed as 
\begin{align*}
    \Phi_0=\Phi^{2,2}_0+\Phi^{1,3}_0+\Phi^{0,4}_0
\end{align*}
whereby lifted to $\m{Fr}(X/S,\Phi)\times \mathbb{C}^3/\Gamma$ the components can be expressed as
\begin{align*}
    \hat{\Phi}^{2,2}_0=\phi^*\m{vol}_{g_S}\wedge\omega_0,\hspace{0,5cm}\hat{\Phi}^{1,3}_0=-\m{tr}(\theta_{S}\wedge\theta_0)\und{0.5cm}\hat{\Phi}^{0,4}_0=\frac{1}{2}\m{tr}(\omega_0\wedge\omega_0).
\end{align*}
Here $(\omega_0,\theta_0)$ denotes the standard $\m{SU}(3)$-structure on $\mathbb{C}^3$. Moreover, the CF-$\m{Spin}(7)$-structure satisfies 
\begin{align*}
    \m{d}\Phi_0=\m{d}^{2,-1}\Phi_0\Leftrightarrow \widehat{D}_{0}\Phi_0=0.
\end{align*}

\textbf{Higher Codimension:} For codimension seven and eight singular strata, the CF-$\m{Spin}(7)$-structures are trivial in the following sense. The normal cone bundle of codimension seven singular strata is trivialisable, i.e. 
\begin{align*}
    N_0\cong\mathbb{T}^1\times \m{Im}(\mathbb{O})
\end{align*}
and the normal cone structure is given by 
\begin{align*}
    \Phi_0=\Phi_0^{1,3}+\Phi_0^{0,4}=\theta\wedge\phi_0 +\psi_0
\end{align*}
where $(\phi_0,\psi_0)$ is the standard $G_2$-structure on $\m{Im}(\mathbb{O})$. For codimension eight-singular strata, i.e. conical singularities the normal cone structure is given by the standard $\m{Spin}(7)$-structure 
\begin{align*}
    \Phi_0=\left<.,.\times_0.\times_0.\right>_0.
\end{align*}
on $\mathbb{O}/\Gamma$.\\

We will continue by discussing the notion of fibred $\m{Spin}(7)$-structures and the notion of adiabatic torsion-freeness as introduced by Donaldson in \cite{donaldson2016adiabaticlimitscoassociativekovalevlefschetz}. We will see that the $\m{Spin}(7)$-normal cone structure of a singular stratum naturally defines an adiabatic torsion-free $\m{Spin}(7)$-structure. Later in this section, we will construct resolutions of the normal cone bundle and we will see that the notion of torsion-free $\m{Spin}(7)$-structures is the right setting for the construction of resolutions of $\m{Spin}(7)$-orbifolds.\\

Let $\pi\colon (X,\Phi)\rightarrow (B,g_B)$ be a $\m{Spin}(7)$-structure inducing a Riemannian fibration over $(B,g_B)$. The Riemannian structure induced by $\Phi$ induces an Ehresmann connection $H$ such that 
\begin{align*}
    g_\Phi=g^{2,0}_\Phi+g^{0,2}_\Phi=\pi^*g_B\oplus g^{0,2}_\Phi
\end{align*}
and $\Phi=\sum_{k=0}^4\Phi^{k,4-k}$. We define a family of $\m{Spin}(7)$-structures $\Phi^t$ on $X$ by 
\begin{align*}
    \Phi^t\coloneqq \sum_{k=0}^4t^k\Phi^{4-k,k}\Rightarrow g_{\Phi^t}=g^{2,0}_\Phi+t^2g^{0,2}_\Phi
\end{align*}
which geometrically corresponds to scaling the fibres by $t$. Donaldson's program on adiabatic torsion-free exceptional holonomy spaces aims at understanding the adiabatic limit of special holonomy structures, extracting the adiabatic special holonomy data and constructing torsion-free special holonomy spaces given such an adiabatic datum.

\begin{defi}
A fibred $\m{Spin}(7)$-structure $\pi\colon (X,\Phi)\rightarrow (B,g_B)$ is called \textbf{adiabatic torsion-free} if 
    \begin{align*}
        \m{d}\Phi^t=\mathcal{O}(t), 
    \end{align*}
    i.e. the family $\Phi^t$ is formally torsion-free in the adiabatic limit $\lim_{t\to 0}(X, \Phi^t)$. 
\end{defi}

\begin{lem}
    A fibred $\m{Spin}(7)$-structure $\pi\colon (X,\Phi)\rightarrow (B,g_B)$ is adiabatic torsion-free if and only if 
    \begin{align*}
        \m{d}\Phi=\m{d}^{2,-1}\Phi.
    \end{align*}
\end{lem}

\begin{proof}
    In the presence of the Ehresmann connection induced by $\Phi$, the de Rham differential splits into 
    \begin{align*}
        \m{d}=\m{d}^{1,0}+\m{d}^{0,1}+\m{d}^{2,-1}.
    \end{align*}
    We compute 
    \begin{align*}
        \left|\m{d}\Phi^t\right|_{g_{\Phi^t}}^2=\left|\m{d}^{1,0}\Phi+t^{-1}\cdot\m{d}^{0,1}\Phi+t\cdot\m{d}^{2,-1}\Phi\right|_{g_{\Phi}}^2
    \end{align*}
which only satisfies $\mathcal{O}(t)$ if $\m{d}^{1,0}\Phi=0$ and $\m{d}^{0,1}\Phi=0$.
\end{proof}

\begin{lem}\cite[Eq. (2.3)]{goette2014adiabatic}
    Let $\pi\colon (X,\Phi)\rightarrow (B,g_B)$ be a fibred $\m{Spin}(7)$-structure with minimal fibres. Then $\Phi$ is adiabatic torsion-free if and only if 
    \begin{align*}
        \widehat{D}\Phi\coloneqq\m{cl}_{g_{\Phi}}\circ \nabla^{\oplus}\Phi=0
    \end{align*}
    where $\m{cl}_{g_\Phi}\colon T^*X\otimes \wedge^\bullet T^*X\rightarrow \wedge^\bullet T^*X $ denote the Clifford multiplication.
\end{lem}

\subsection{Local Chen--Ruan Systems and the Generalised McKay Correspondence}
\label{Local Chen--Ruan-Systems and the Generalised McKay-Correspondence}

In this section we introduce the Chen--Ruan local systems associated with the singular strata of a $\m{Spin}(7)$-orbifold and relate them to the McKay-correspondence. We consider orbifolds whose isotropy groups act on the normal representations through finite subgroups of $\m{SU}(V)$. In codimension four this means that the local isotropy groups are conjugate to finite subgroups of
\begin{align*}
    \m{Sp}(1)\cong \m{SU}(2).
\end{align*}
The inertia orbifold records the conjugacy classes of these groups, and the age grading determines the corresponding degree shifts in Chen--Ruan cohomology; we follow the standard treatments of Chen--Ruan cohomology and inertia orbifolds in \cite{chen2004new,adem2007orbifolds,moerdijk2002orbifolds}.
For $\Gamma\subset \m{SU}(2)$ every non-trivial conjugacy class has age one, so the local Chen--Ruan contribution is concentrated in degree two.\\

By the McKay correspondence \cite{mckay1980graphs,reid2002mckay}, these degree-two classes are identified with the second cohomology of the ALE hyperkähler resolutions of $\mathbb{H}/\Gamma$. In the codimension-four case this identification is realised geometrically by Kronheimer's construction of ALE hyperkähler spaces \cite{kronheimer1989construction,kronheimer}. Thus,
over a codimension-four stratum $S$, one obtains a flat vector bundle $\mathfrak{H}^\bullet_\Gamma$ which can be interpreted either as the degree-two Chen--Ruan local system or as the bundle of McKay cohomology groups of the local ALE resolutions. Combining this identification with the hyperkähler parameter space gives the resolution
parameter bundle
\begin{align*}
    \mathfrak{P}_\Gamma\coloneqq \wedge^2_+T^*S\otimes \mathfrak{H}^2_\Gamma .
\end{align*}
Sections of $\mathfrak{P}_\Gamma$ assign to each point of $S$ a stability condition for the ALE resolution of the normal cone fibre.\\

Using the equivariant universal McKay family, we show that such sections produce ACF $\m{Spin}(7)$-resolutions of the normal cone bundle. Moreover, the adiabatic torsion-free condition is equivalent to the harmonicity of the corresponding $\mathfrak{H}^2_\Gamma$-valued self-dual two-form. Thus smooth adiabatic codimension-four resolutions are parametrised by regular harmonic representatives of classes in
\begin{align*}
\m{H}^2(S,\mathfrak{H}^2_\Gamma),
\end{align*}
which are precisely the local Chen--Ruan classes associated with the stratum.\\

While the later construction uses this result in codimension four, the formalism is designed to extend to any quotient $V/\Gamma$ whose crepant Calabi--Yau or hyperkähler resolutions are governed by McKay-type quiver moduli. In such cases the same mechanism identifies Chen--Ruan local systems with the cohomology of local resolutions and turns quiver stability parameters into geometric resolution parameters. This broader perspective includes the construction of moduli spaces via stability conditions and quiver varieties in the work of King \cite{king1994moduli}, Crawley--Boevey
\cite{crawley2001geometry}, Craw--Ishii \cite{craw2004flops}, and Yamagishi \cite{yamagishi2025moduliI,yamagishi2025moduliII}.

\subsubsection{Chen--Ruan Cohomology}
\label{Chen--Ruan-Cohomology}

We begin by recalling the part of Chen--Ruan cohomology which is relevant for the local structure of a $\m{Spin}(7)$-orbifold near a singular stratum. Let $S\subset X$ be a connected singular stratum with isotropy group $\Gamma$. Locally near $S$, the orbifold is modelled on a quotient
\begin{align*}
        H\oplus V/\Gamma ,
\end{align*}
where $H$ is the tangent representation along the stratum and $V$ is the normal representation. In the cases relevant for us, $\Gamma$ acts trivially on $H$ and as a finite subgroup of $\m{Sp}(V)$ or $\m{SU}(V)$ on $V$.\\

In the $\m{Spin}(7)$-context, the codimension four and codimension six strata are of this form, with
\begin{align*}
        \Gamma\subset \m{Sp}(1)\cong \m{SU}(2)
        \qquad\text{or}\qquad
        \Gamma\subset \m{SU}(3),
\end{align*}
respectively, while in codimension eight, there are possible finite subgroups that are not contained in $\m{Sp}(2)$ or $\m{SU}(4)$.\\

The Chen--Ruan grading is determined by the action of $\Gamma$ on the normal representation $V$.

\begin{defi}
Let $g\in \m{SU}(V)$ be an element of finite order and let $\{e^{2\pi \m{i}\theta_1},\dots,e^{2\pi \m{i}\theta_{m/2}}\}$
be the eigenvalues of $g$ on $V$, with $\theta_j\in[0,1)$. The \textbf{age} of $g$ is defined by
\begin{align*}
        \m{age}_V(g)
        \coloneqq
        \sum_{j=1}^{m/2} \theta_j .
\end{align*}
Since $g\in \m{SU}(V)\subset\m{SO}(V)$, we have $\m{det}(g)=1$  and hence $\m{age}_V(g)\in\{0,1,...,m/2\}$.
\end{defi}

Moreover, the age only depends on the conjugacy class of $g$. Hence, for every finite subgroup $\Gamma\subset \m{SU}(V)$, it defines a function
\begin{align*}
        \m{age}_V\colon \m{Con}(\Gamma)\longrightarrow \mathbb{Z}_{\geq 0}.
\end{align*}

\begin{rem}
Equivalently, if $\log_{[0,2\pi)}$ denotes the branch of the logarithm whose
eigenvalues have imaginary parts in $[0,2\pi)$, then
\begin{align*}
        \m{age}_V(g)
        =
        \frac{1}{2\pi\m{i}}
        \m{tr}\bigl(\log_{[0,2\pi)}(g)\bigr).
\end{align*}
\end{rem}

\begin{lem}
Let $\Gamma\subset \m{SU}(V)\subset \m{Spin}(7)$ be a finite subgroup. Then
the normaliser $\m{N}_{\m{Spin}(7)}(\Gamma)$ acts through $\m{N}_{\m{Spin}(7)}(\Gamma)\rightarrow \m{N}_{\m{SO}(V)}(\Gamma)$ on $\m{Con}(\Gamma)$ by
conjugation. Moreover, the age function is invariant under this action.
\end{lem}

\begin{proof}
If $n\in \m{N}_{\m{Spin}(7)}(\Gamma)$, then
\begin{align*}
        n\Gamma n^{-1}=\Gamma.
\end{align*}
Therefore conjugation by $n$ defines an automorphism of $\Gamma$ and hence a
permutation of the set of conjugacy classes $\m{Con}(\Gamma)$. Since $g$ and
$ngn^{-1}$ are conjugate as linear transformations of the normal representation
$V$, they have the same eigenvalues with multiplicity. Consequently,
\begin{align*}
        \m{age}_V(ngn^{-1})=\m{age}_V(g).
\end{align*}
\end{proof}

\begin{defi}
For $l\in\mathbb{Z}_{\geq 0}$, define
\begin{align*}
        \m{Con}^l(\Gamma)
        \coloneqq
        \{(g)\in \m{Con}(\Gamma)\mid \m{age}_V(g)=l\}
\end{align*}
and set $\mathfrak{n}_{\Gamma,l}\coloneqq\#\m{Con}^l(\Gamma).$
\end{defi}

We now recall the inertia orbifold. We follow
\cite[Sec. 4]{adem2007orbifolds} and \cite[Chap. 6.4]{moerdijk2002orbifolds}.

\begin{defi}
Let $X$ be an orbifold and let $X^\bullet$ be a proper foliating Lie groupoid
presenting $X$. The \textbf{inertia Lie groupoid}
$\mathcal{I}X^\bullet$ is defined by
\begin{align*}
        \mathcal{I}X^0
        &\coloneqq
        \{\phi\in X^1\mid s(\phi)=t(\phi)\},\\
        \mathcal{I}X^1
        &\coloneqq
        \{(\kappa,\phi)\in X^1\times \mathcal{I}X^0
        \mid s(\kappa)=s(\phi)\},
\end{align*}
with source and target maps given by $s(\kappa,\phi)=\phi$ and $t(\kappa,\phi)=\kappa\phi\kappa^{-1}$ respectively.
\end{defi}

The groupoid $\mathcal{I}X^\bullet$ presents an orbifold, denoted
$\mathcal{I}X$, called the \textbf{inertia orbifold} of $X$. Equivalently, if
\begin{equation*}
\begin{tikzcd}
	{\Lambda(X^1)} && {X^1} \\
	\\
	{X^0} && {X^0\times X^0}
	\arrow[from=1-1, to=1-3]
	\arrow[from=1-1, to=3-1]
	\arrow["\lrcorner"{anchor=center, pos=0.125}, draw=none, from=1-1, to=3-3]
	\arrow["{(s,t)}"{description}, from=1-3, to=3-3]
	\arrow["\Delta"{description}, from=3-1, to=3-3]
\end{tikzcd}
\end{equation*}
denotes the space of loops of the presenting groupoid, then
\begin{align*}
        \mathcal{I}X^\bullet
        \cong
        X^\bullet\ltimes \Lambda(X^1),
\end{align*}
where $X^\bullet$ acts on $\Lambda(X^1)$ by conjugation.

\begin{ex}
If $X^\bullet=\Gamma\ltimes M$ is an action groupoid, then
\begin{align*}
        \mathcal{I}X^\bullet
        \cong
        \Gamma\ltimes
        \left(
            \bigsqcup_{g\in\Gamma}\m{Fix}(g,M)
        \right),
\end{align*}
where $h\in\Gamma$ maps $h\colon \m{Fix}(g,M)\rightarrow\m{Fix}(hgh^{-1},M), x\longmapsto h.x$. Thus the inertia orbifold decomposes according to conjugacy classes in
$\Gamma$.
\end{ex}

\begin{defi}
Let $X$ be an orbifold whose isotropy groups act on their normal
representations through finite subgroups of $\m{SU}(V)$. The
\textbf{Chen--Ruan cohomology} of $X$ with coefficients in a ring $R$ is the
graded $R$-module
\begin{align*}
        \m{H}^\bullet_{\m{CR}}(X,R)
        \coloneqq
        \bigoplus_{(g)}
        \m{H}^{\bullet-2\m{age}(g)}
        \bigl(\mathcal{I}X_{(g)},R\bigr),
\end{align*}
where the sum is taken over the connected components of the inertia orbifold,
labelled locally by conjugacy classes of the corresponding isotropy groups.
The function $\m{age}(g)$ is computed from the action of $g$ on the normal
representation.
\end{defi}

We now specialise this definition to a tubular neighbourhood of a singular
stratum. Let $S\subset X$ be a connected singular stratum with isotropy group
$\Gamma$ and normal representation $V$. Let
\begin{align*}
        Fr_{\m{Spin}(7)}(X/S)\longrightarrow S
\end{align*}
be the corresponding $\m{N}_{\m{Spin}(7)}(\Gamma)$-reduction of the
$\m{Spin}(7)$-frame bundle along $S$. Then the normal cone bundle is
\begin{align*}
        N_0
        \cong
        Fr_{\m{Spin}(7)}(X/S)
        \times_{\m{N}_{\m{Spin}(7)}(\Gamma)}
        V/\Gamma .
\end{align*}
The inertia orbifold of the normal cone bundle is described by
\begin{align*}
        \mathcal{I}N_0
        \cong
        Fr_{\m{Spin}(7)}(X/S)
        \times_{\m{N}_{\m{Spin}(7)}(\Gamma)}
        \left(
            \bigsqcup_{(g)\in\m{Con}(\Gamma)}
            V^g/C(g)
        \right),
\end{align*}
where $V^g\coloneqq \m{Fix}(g,V)$ and $C(g)\subset\Gamma$ denotes the centraliser of $g$. Since the age is invariant under the normaliser action, the above decomposition
can be grouped according to the age. Define
\begin{align*}
        N_{\Gamma,l}
        \coloneqq
        Fr_{\m{Spin}(7)}(X/S)
        \times_{\m{N}_{\m{Spin}(7)}(\Gamma)}
        \left(
            \bigsqcup_{(g)\in\m{Con}^l(\Gamma)}
            V^g/C(g)
        \right).
\end{align*}
Consequently, the inertia orbifold $\mathcal{I}X$ can be covered by 
\begin{align*}
        \mathcal{I}X
        \cong\mathcal{I}(X\backslash S) \cup 
        \bigsqcup_{l\geq 0} N_{\Gamma,l}.
\end{align*}
The component corresponding to the identity element has age zero and is the ordinary local contribution
\begin{align*}
        N_{\Gamma,0}\cong N_0 .
\end{align*}
The components with $l>0$ give the additional Chen--Ruan contributions. For later use it is convenient to record the base of each local inertia component. Set
\begin{align*}
        S_{\Gamma,l}
        \coloneqq
        Fr_{\m{Spin}(7)}(X/S)
        \times_{\m{N}_{\m{Spin}(7)}(\Gamma)}
        \m{Con}^l(\Gamma).
\end{align*}
Then there is a natural finite covering map
\begin{align*}
        \pi_{\Gamma,l}\colon S_{\Gamma,l}\longrightarrow S
\end{align*}
whose degree is $\mathfrak{n}_{\Gamma,l}$, up to the monodromy induced by the normaliser action. The space $N_{\Gamma,l}$ is naturally fibred over $S_{\Gamma,l}$, with fibre $V^g/C(g)$ over the point labelled by $(g)$.

\begin{defi}
The \textbf{Chen--Ruan local system of age $l$} associated with the stratum
$S$ is
\begin{align*}
        \mathfrak{H}^{2l}_\Gamma
        \coloneqq
        (\pi_{\Gamma,l})_*\underline{\mathbb{R}}_{S_{\Gamma,l}}[2l].
\end{align*}
Equivalently,
\begin{align*}
        \mathfrak{H}^{2l}_\Gamma
        \cong
        Fr_{\m{Spin}(7)}(X/S)
        \times_{\m{N}_{\m{Spin}(7)}(\Gamma)}
        \mathbb{R}^{\m{Con}^l(\Gamma)}[2l].
\end{align*}
We also write
\begin{align*}
        \mathfrak{H}^\bullet_\Gamma
        \coloneqq
        \bigoplus_{l>0}\mathfrak{H}^{2l}_\Gamma .
\end{align*}
\end{defi}

With this notation, the local Chen--Ruan contribution of the stratum $S$ is
\begin{align*}
        \bigoplus_{l>0}
        \m{H}^{\bullet-2l}(S,\mathfrak{H}^{2l}_\Gamma).
\end{align*}
Thus, if $X$ has a single singular stratum $S$ of this type, then as a graded vector space
\begin{align}
\label{CRcohomolocalsystems}
        \m{H}^\bullet_{\m{CR}}(X,\mathbb{R})
        \cong\m{H}^\bullet(X,\mathbb{R})
        \oplus\bigoplus_{0<l}\m{H}^{\bullet-2l}(S,\mathfrak{H}^{2l}_\Gamma).
\end{align}

\begin{rem}
The above construction is formally close to the local description of intersection
cohomology. Indeed, the local systems
\begin{align*}
\mathfrak{H}^{2l}_\Gamma(R)
\coloneqq(\pi_{\Gamma,l})_*\underline{R}_{S_{\Gamma,l}}[2l],
\und{1.0cm} \mathfrak{H}^\bullet_\Gamma(R)\coloneqq
\bigoplus_{l>0}\mathfrak{H}^{2l}_\Gamma(R),
\end{align*}
may be regarded as the coefficient systems carried by the non-trivial local inertia components of $X$. The Chen--Ruan grading shifts these local systems by twice the age. Equivalently, the age function defines a perversity on the strata of the inertia orbifold. One can therefore use the pair consisting of the local system $\mathfrak{H}^\bullet_\Gamma(R)$ and this age-induced perversity to define an intersection-cohomology-type complex adapted to the orbifold stratification.\\

In this sense, the Chen--Ruan cohomology should be viewed as the part of such a theory which records the local cohomology supported along the singular strata. We will not use this intersection cohomology interpretation directly, but it is useful conceptually. The same age shift which defines the Chen--Ruan grading also determines the perversity with which the local contribution of a singular stratum is inserted into the global cohomology theory.
\end{rem}

In particular, suppose that $S$ has codimension four. Then
$V\cong\mathbb{H}\cong\mathbb{C}^2 $ and $\Gamma\subset\m{Sp}(1)\cong \m{SU}(2)$. If $g\neq 1$, then the eigenvalues of $g$ are of the form $\{e^{2\pi\m{i}\theta},e^{2\pi\m{i}(1-\theta)}\}$
with $0<\theta<1$. Hence
\begin{align*}
        \m{age}(g)=1
\end{align*}
for every non-trivial $g\in\Gamma$. Therefore all non-trivial local Chen--Ruan contributions over $S$ have degree shift two. In this case $\mathfrak{H}^\bullet_\Gamma=\mathfrak{H}^2_\Gamma$
and
\begin{align*}
        \m{H}^\bullet_{\m{CR}}(X,\mathbb{R})\cong\m{H}^\bullet(X,\mathbb{R})\oplus\m{H}^{\bullet-2}(S,\mathfrak{H}^2_\Gamma).
\end{align*}

\subsubsection{Local McKay-Duality}
\label{Local McKay-Duality}

We now reinterpret the Chen--Ruan local systems introduced above in terms of the McKay correspondence. The point is that, for a quotient singularity $V/\Gamma$, the Chen--Ruan cohomology parametrises orbifold resolutions of the quotient orbifold equipped with (Q)ALE-hyperkähler- or (Q)ALE-Calabi--Yau-structures respectively. \\

In this paper we will use this picture primarily in codimension four. This is the case in which the local singularity is modelled on
\begin{align*}
\mathbb{H}/\Gamma,\qquad\Gamma\subset \m{Sp}(1)\cong\m{SU}(2)
\end{align*}
and the resolution theory is completely controlled by Kronheimer's \cite{kronheimer1989construction,kronheimer}
gravitational instanton construction. It is also the case most directly relevant for the non-compactness phenomena in the moduli space of torsion-free $\m{Spin}(7)$-structures as discussed in Section \ref{The Moduli Space of Torsion-Free Spin(7)-Structures and Smooth Gromov--Hausdorff Resolutions}.\\

Nevertheless, the formulation below is deliberately written in a way which separates the representation-theoretic input from the special features of codimension four. Whenever the algebraic geometry of crepant/hyperkähler resolutions of $V/\Gamma$ is controlled by a McKay-type quiver construction, the same formal mechanism applies. Namely, the conjugacy classes of $\Gamma$
define the local Chen--Ruan systems, the quiver construction identifies these systems with the cohomology of the local resolutions, and the space of stability parameters becomes the fibre of a resolution-parameter bundle over the singular stratum. In the Calabi--Yau case this includes the familiar description of crepant resolutions via $\Gamma$-constellations and GIT quotients, while in the hyperkähler case it includes Nakajima--Kronheimer
quiver varieties.\\

Thus the codimension-four case should be viewed as the model situation in which all analytic and geometric ingredients are available unconditionally. Higher codimension cases can be incorporated into the same framework provided one has a sufficiently uniform family of Calabi--Yau or hyperkähler resolutions of $V/\Gamma$, together with the required equivariance under the normaliser of $\Gamma$ and the corresponding McKay identification of local cohomology.

\begin{rem}
Let $\Gamma \subset \m{SU}(m)$ be a finite subgroup, so that the quotient $\mathbb{C}^{m/2}/\Gamma$ has Gorenstein singularities. The existence and geometry of crepant resolutions of $\mathbb{C}^{m/2}/\Gamma$ depend strongly on the dimension.
\begin{itemize}
    \item For $m = 2$, these are the classical Kleinian singularities. Crepant resolutions always exist and are unique; they correspond to the minimal resolutions of rational surface singularities.
    \item For $m = 3$, the singularities are Gorenstein and always admit crepant resolutions, though in general such resolutions are not unique and may be related by flops. 
    \item For $m = 4$, crepant resolutions do not need to exist. Even when they do, they may be non-unique or disconnected in the birational geometry sense.
\end{itemize}
\end{rem}

We keep the discussion/notations and definitions as general as possible but will focus on the codimension four case, i.e. $\Gamma\subset \m{SU}(2)$.\\

Let $V$ be either a complex Hermitian vector space or a quaternionic Hermitian vector space. We write
\begin{align*}
G=\m{SU}(V)\coloneqq \m{Aut}(V,\omega_0,\Upsilon_0)
\oder{1.0cm}
G=\m{Sp}(V)\coloneqq \m{Aut}(V,\underline{\omega}_0),
\end{align*}
depending on the context, and let $\Gamma\subset G$ be a finite subgroup. We denote by
\begin{align*}
    R_\Gamma\coloneqq \mathbb{C}[\Gamma]
\end{align*}
the regular representation, by $\m{Rep}(\Gamma)$ the representation ring, and by $\m{Irr}(\Gamma)$ the set of irreducible complex representations of $\Gamma$. We also set $\m{Irr}_0(\Gamma)\coloneqq\m{Irr}(\Gamma)\backslash 1$
where $\mathbf{1}$ denotes the trivial representation.\\

The elementary character theory of finite groups gives a canonical equality of cardinalities
\begin{align*}
\#\m{Con}(\Gamma)=\#\m{Irr}(\Gamma).
\end{align*}
More precisely, the irreducible characters form a basis of the space of class functions on $\Gamma$. Thus the vector space generated by conjugacy classes may be identified, after extending scalars to $\mathbb{C}$, with the dual of the representation ring. In the McKay correspondence this representation theoretic identification is refined geometrically.

\begin{defi}
Let $ \Gamma \subset \m{SU}(V) $ be a finite subgroup. The associated \textbf{McKay quiver} $ Q_\Gamma $ is defined as follows:
\begin{itemize}
    \item The set of vertices $ Q_\Gamma^0 $ is given by the isomorphism classes of irreducible representations $ \varrho \in \mathrm{Irr}(\Gamma) $.
    \item The number of arrows from $ \varrho $ to $ \sigma $ is determined by the multiplicities $ a_{\varrho\sigma} $ in the decomposition
    \begin{align*}
        V \otimes R_\varrho \cong \bigoplus_{\sigma \in \m{Irr}(\Gamma)} a_{\varrho\sigma} R_\sigma.
    \end{align*}
\end{itemize}
\end{defi}

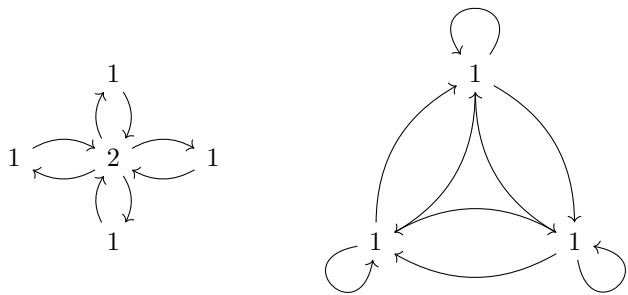
\begin{figure}
    \centering
    \begin{tikzcd}
	& 1 &&&& {1} \\
	1 & 2 & 1 && {} && {} \\
	& 1 &&& {1} && {1}
	\arrow[bend left, from=1-2, to=2-2]
	\arrow[from=1-6, to=1-6, loop, in=125, out=55, distance=10mm]
	\arrow[bend left, from=1-6, to=3-5]
	\arrow[bend left, from=1-6, to=3-7]
	\arrow[bend left, from=2-1, to=2-2]
	\arrow[bend left, from=2-2, to=1-2]
	\arrow[bend left, from=2-2, to=2-1]
	\arrow[bend left, from=2-2, to=2-3]
	\arrow[bend left, from=2-2, to=3-2]
	\arrow[bend left, from=2-3, to=2-2]
	\arrow[bend left, from=3-2, to=2-2]
	\arrow[bend left, from=3-5, to=1-6]
	\arrow[from=3-5, to=3-5, loop, in=260, out=190, distance=10mm]
	\arrow[bend left, from=3-5, to=3-7]
	\arrow[bend left, from=3-7, to=1-6]
	\arrow[bend left, from=3-7, to=3-5]
	\arrow[from=3-7, to=3-7, loop, in=350, out=280, distance=10mm]
\end{tikzcd}
    \caption{McKay Quiver of the group $D_4\subset\m{SU}(2)$ and the group $\mathbb{Z}_3\subset\m{SU}(3)$.}
    \label{D4quiver}
\end{figure}

\begin{figure}[h!]
\centering
\text{$A_n$ }\dynkin[scale=0.5, edge length=1.2cm]{A}{}\hspace{1.0cm}\text{$D_n$ }\dynkin[scale=0.5, edge length=1.2cm]{D}{}
\vspace{1em}
\text{$E_6$ }\dynkin[scale=0.5, edge length=1.2cm]{E}{6}\\
\text{$E_7$ }
\dynkin[scale=0.5, edge length=1.2cm]{E}{7}
\vspace{1em}
\text{$E_8$ }\dynkin[scale=0.5, edge length=1.2cm]{E}{8}
\caption{Dynkin diagrams of types $A_n$, $D_n$, and the exceptional types $E_6$, $E_7$, and $E_8$.}
\end{figure}

\begin{figure}[h!]
    \centering
\begin{tikzcd}
	{\varrho_0} & {\varrho_1} & {\varrho_2} & \cdots & {\varrho_{n-1}} & {\varrho_n} \\
	&&& {} \\
	{} & \bullet & \bullet & \cdots & \bullet & \bullet
	\arrow[bend left, dotted, from=1-1, to=1-2]
	\arrow[bend left, dotted, from=1-2, to=1-1]
	\arrow[bend left, from=1-2, to=1-3]
	\arrow[bend left, from=1-3, to=1-2]
	\arrow[bend left, from=1-3, to=1-4]
	\arrow[bend left, from=1-4, to=1-3]
	\arrow[bend left, from=1-4, to=1-5]
	\arrow["{\text{McKay Correspondence}}"{description}, squiggly, from=1-4, to=3-4]
	\arrow[bend left, from=1-5, to=1-4]
	\arrow[bend left, from=1-5, to=1-6]
	\arrow[bend left, from=1-6, to=1-5]
	\arrow[no head, from=3-2, to=3-3]
	\arrow[no head, from=3-3, to=3-4]
	\arrow[no head, from=3-4, to=3-5]
	\arrow[no head, from=3-5, to=3-6]
\end{tikzcd}
    \caption{The McKay correspondence of $\mathbb{Z}_n\subset\m{SU}(2)$. The quiver $Q_{\mathbb{Z}_n}$ corresponds to the extended $A_{n-1}$ Dynkin diagram.}
    \label{fig:enter-label}
\end{figure}

In the case $\Gamma\subset \m{Sp}(1)$, the McKay quiver $Q_\Gamma$ is the extended Dynkin diagram of type ADE. If $A_\Gamma$ denotes
the adjacency matrix of this graph, then
\begin{align*}
C_\Gamma\coloneqq 2I-A_\Gamma
\end{align*}
is the affine Cartan matrix. Removing the vertex corresponding to the trivial representation gives the finite ADE Cartan matrix. We denote the corresponding real Cartan space by
\begin{align*}
    \mathfrak{h}_\Gamma\cong\mathbb{R}^{\m{Irr}_0(\Gamma)}.
\end{align*}
The Cartan form on $\mathfrak{h}_\Gamma$ will again be denoted by $C_\Gamma$.

\begin{defi}
The space of integral stability parameters is
\begin{align*}
    \Theta_{\mathbb{Z}}(\Gamma)\coloneqq\left\{\theta\in \m{Hom}{\mathbb{Z}}(\m{Rep}(\Gamma),\mathbb{Z})|\theta(R\Gamma)=0\right\}.
\end{align*}
We set $\Theta_{\m{Im}(\mathbb{C})}(\Gamma)
\coloneqq \Theta_{\mathbb{Z}}(\Gamma)\otimes_{\mathbb{Z}}\mathbb{R}$ and $\Theta_{\m{Im}(\mathbb{H})}(\Gamma)
\coloneqq \m{Im}(\mathbb{H})\otimes_{\mathbb{R}}\Theta_{\m{Im}(\mathbb{C})}(\Gamma)$.
\end{defi}

Using the decomposition of the regular representation into irreducibles, one obtains natural identifications
\begin{align*}
\Theta_{\m{Im}(\mathbb{C})}(\Gamma)
\cong\mathfrak{z}(\mathfrak{pu}(R_\Gamma)^\Gamma)\cong\mathfrak{h}_\Gamma
\cong\mathbb{R}^{\m{Irr}_0(\Gamma)} .
\end{align*}
Indeed, an element of the centre of $\mathfrak{pu}(R_\Gamma)^\Gamma$ acts by a scalar on each non-trivial isotypical summand of $R_\Gamma$, subject to the trace-zero condition imposed by the regular representation. Moreover, the age grading defines a grading on the $\Theta_{\m{Im}(\mathbb{C})}$ via its identification with $\mathbb{R}^{\m{Irr}_0(\Gamma)}\cong \mathbb{R}^{\m{Con}^{0<}(\Gamma)}$.\\

The wall arrangement in the parameter space is the finite union of hyperplanes on which the corresponding quiver quotient is singular. We denote it by
    \begin{align*}
        \mathcal{W}_\Gamma\coloneqq \left\{\theta\in \Theta_{\m{Im}(\mathbb{C})}(\Gamma)\left|\vec{v}\in \mathbb{N}^{\m{Irr}(\Gamma)}\text{ such that }v_\rho\leq \dim(\rho),\text{ and }\sum_{\rho\in \m{Irr}_0(\Gamma)}(v_{\rho}-v_{\rho_0})\m{dim}(\rho)\theta_\rho=0\right.\right\}
    \end{align*}
and write
\begin{align*}
\Theta^{\m{reg}}_{\m{Im}(\mathbb{C})}(\Gamma)
\coloneqq
\Theta_{\m{Im}(\mathbb{C})}(\Gamma)\backslash \mathcal{W}_\Gamma .
\end{align*}
In the hyperkähler case we similarly set
\begin{align*}
\Theta^{\m{reg}}_{\m{Im}(\mathbb{H})}(\Gamma)
\coloneqq\Theta_{\m{Im}(\mathbb{H})}(\Gamma)\backslash
(\m{Im}(\mathbb{H})\otimes \mathcal{W}_\Gamma).
\end{align*}

The following proposition will be used in the case
\begin{align*}
\Gamma\subset \m{Sp}(1)\cong \m{SU}(2).
\end{align*}
In this case Kronheimer's construction gives a complete description of the hyperkähler ALE resolutions of $\mathbb{H}/\Gamma$ in terms of the McKay-quiver of $\Gamma$ and its stability parameters \cite{kronheimer1989construction,kronheimer}. We package this construction as a universal family over the hyperkähler parameter space
\begin{align*}
\Theta_{\m{Im}(\mathbb{H})}(\Gamma).
\end{align*}
The equivariance under the normaliser of $\Gamma$ will be essential below, because it allows the universal family to be associated to the $\m{N}_{\m{Spin}(7)}(\Gamma)$-reduction over a singular stratum.\\

Similar statements hold, with suitable modifications, in higher-dimensional situations whenever the relevant crepant Calabi--Yau or hyperkähler resolutions of $V/\Gamma$ are described by moduli spaces of McKay-quiver representations. For finite subgroups
\begin{align*}
\Gamma\subset \m{SU}(3),
\end{align*}
this is known for projective crepant resolutions of $\mathbb{C}^3/\Gamma$; Craw--Ishii \cite{craw2004flops} proved the result for Abelian subgroups, while Yamagishi\cite{yamagishi2025moduliI,yamagishi2025moduliII} proved the Craw--Ishii conjecture for arbitrary finite subgroups of
$\m{SL}_3(\mathbb{C})$, showing that every projective crepant resolution is a moduli space of $\theta$-stable $\Gamma$-constellations for a suitable generic stability condition. For symplectic quotient singularities there are analogous Nakajima-quiver descriptions in the wreath-product case
\begin{align*}
\Gamma_n=\Gamma\wr\mathfrak{S}_n
\subset \m{Sp}(n), \qquad \Gamma\subset \m{SL}_2(\mathbb{C}),
\end{align*}
where the corresponding symplectic resolutions are Hilbert schemes of points on the minimal resolution of $\mathbb{C}^2/\Gamma$ and their birational models are obtained by variation of GIT for Nakajima quiver varieties \cite{nakajima1994instantons,nakajima1999lectures,bellamy2018birational}. In particular, this gives examples in
\begin{align*}
\m{Sp}(2)
\end{align*}
by taking $n=2$. In dimension four, i.e. for subgroups of
\begin{align*}
\m{SU}(4),
\end{align*}
there is no comparable general theorem for all finite subgroups. The same formalism applies only in those cases where a crepant Calabi--Yau resolution, or a hyperkähler resolution in the symplectic case, is known and is realised by a suitable McKay-quiver moduli space.\\

Thus the codimension-four case treated below is the unconditional model case. The higher-dimensional analogues should be read as conditional on the existence of an equivariant universal family of crepant resolutions with the same quiver-theoretic properties.

\begin{prop}[Universal Moduli Spaces of McKay-Quiver Representations \cite{kronheimer1989construction,kronheimer,majewskithesis}]
\label{universalMcKaymoduli}
Let $\Gamma\subset\m{Sp}(1)$ be a finite subgroup. 
\begin{itemize}
    \item There exists a stratified fibration, smooth over the regular stability conditions $\Theta^{reg}_{\m{Im}(\mathbb{H})}(\Gamma)\coloneqq \Theta_{\m{Im}(\mathbb{H})}(\Gamma)\backslash\mathcal{W}_\Gamma$
\begin{equation*}
\begin{tikzcd}
	{\mathcal{Q}_\Gamma} && \\
	\\
	{\mathcal{M}_\Gamma} && {\Theta_{\mathrm{Im}(\mathbb{H})}(\Gamma)}
	\arrow["\pi_\Gamma"{description}, from=1-1, to=3-1]
	\arrow["\mu_\Gamma"{description}, from=1-1, to=3-3]
	\arrow["\kappa_\Gamma"{description}, from=3-1, to=3-3]
\end{tikzcd}
\end{equation*}
The map $\pi_\Gamma$ is a $\m{PU}(R_\Gamma)^{\Gamma}$-principal bundle. There exists a decomposition 
\begin{align*}
    T\mathcal{Q}_\Gamma\cong K_\Gamma\oplus \hat{H}_\Gamma\oplus V\pi_\Gamma
\end{align*}
the connection $K_\Gamma\oplus \hat{H}_\Gamma$ connection one form $\mathfrak{a}_\Gamma\in \mathcal{A}(\pi_\Gamma)$ is $\m{PU}(R_\Gamma)^{\Gamma}$-invariant and descends to a connection $\hat{H}_\Gamma\oplus V\kappa_\Gamma$ on $\kappa_\Gamma$.
\item There exists $\underline{\omega}_{\mathcal{M}_\Gamma}\in \Omega^{0,2}(\mathcal{M}_\Gamma,\m{Im}(\mathbb{H}))$ such that 
\begin{align*}
    \m{d}\underline{\omega}_{\mathcal{M}_\Gamma}=\m{d}^{1,0}\underline{\omega}_{\mathcal{M}_\Gamma}=-\left<\kappa_\Gamma^*\theta_{\m{Im}(\mathbb{H})}\wedge \pi_{\mathfrak{z}(\mathfrak{pu}(R_\Gamma)^\Gamma)}F_{\mathfrak{a}_\Gamma}\right>.
\end{align*}
Moreover, the fibres $(M_\zeta,\underline{\omega}_\zeta)$ of $\kappa_\Gamma$ are stratified hyperkähler spaces, and smooth manifolds for $\zeta\in \Theta^{reg}_{\m{Im}(\mathbb{H})}(\Gamma)$. 
\item The group $\m{N}_{\m{SO}(V)}(\Gamma)$ acts on $\mathcal{M}_\Gamma$ such that all the maps, connections and the vertical hyperkähler form are equivariant. Moreover, the group $\mathbb{R}_{>0}$ acts via a universal dilation map $\delta_t\colon \mathcal{M}_\Gamma\rightarrow \mathcal{M}_\Gamma$, equivariant with respect to the dilation $\kappa_\Gamma\circ \delta_t=\delta_{t^2}\circ \kappa_\Gamma$ such that 
\begin{align*}
    \delta_t^*\underline{\omega}_{t^2\zeta}=t^2\underline{\omega}_{\zeta}.
\end{align*}
The identification of the central fibre $(\kappa^{-1}(0),\underline{\omega}_{0})\coloneqq (M_0,\underline{\omega}_0)\cong (V/\Gamma,\underline{\omega}_0)$ is equivariant with respect to both actions by $\m{N}_{\m{SO}(V)}(\Gamma)$ and by $\mathbb{R}_{>0}$ 
\item 
 There exists a universal resolution map 
\begin{align*}
    \rho_{\mathcal{M}_\Gamma}\colon \mathcal{M}_\Gamma\dashrightarrow M_0
\end{align*}
such that $\rho_\zeta\colon (M_\zeta,\underline{\omega}_\zeta)\dashrightarrow (M_0,\underline{\omega}_0)\cong (V,\underline{\omega}_0)$ are ALE of rate $\nu=-4$, i.e.
\begin{align*}
    \nabla^{k}((\rho_\zeta)_*\underline{\omega}_{\zeta}-\underline{\omega}_0)= \mathcal{O}(r^{-4-k})
\end{align*}
and $\mathfrak{a}_\Gamma$ 
is an ASD-instanton of the same rate\footnote{i.e. $\nabla^{k}F_{\mathfrak{a}_\zeta}=\mathcal{O}(r^{-4-k})$}. 
\item There exists an isomorphism of graded vector spaces 
\begin{align*}
    \m{H}^\bullet(M_\zeta)\cong \m{H}^{4-\bullet}_c(M_\zeta)\cong \m{H}^\bullet_{\m{CR}}(V/\Gamma)\cong (\mathbb{R}[0]\oplus \mathfrak{h}^{\bullet}_\Gamma,C_\Gamma)
\end{align*}
In particular, $\m{H}_2(M_\zeta,\mathbb{Z})\cong \Theta(\Gamma)$ is generated by $\underline{\omega}_\zeta$-holomorphic spheres\footnote{There exists a $q\in \m{Im}(\mathbb{H})$ such that $C_i\subset \rho_\zeta^{-1}([0/\Gamma])$ is $\left<\underline{\omega}_\zeta,q\right>$-holomorphic. Moreover, $C_i$ is holomorphic Lagrangian with respect to $\left<\underline{\omega}_\zeta,q^\perp\right>$.}, forming the exceptional set; intersecting according to the Cartan-form $C_\Gamma$.
\item The Gromov--Hausdorff distance of $(M_{\zeta},\underline{\omega}_\zeta)$ and $(M_{\zeta'},\underline{\omega}_{\zeta'})$ is proportional to $|\zeta-\zeta'|$ and hence $(M_{t^2\zeta},\underline{\omega}_{t^2\zeta})\xrightarrow[GH]{t\to 0}(M_0,\underline{\omega}_0)$. 
\item There exists a universal distance function $\mathfrak{w}_{\m{AC}}\colon \mathcal{M}_\Gamma\rightarrow[0,\infty)$ such that $w_{\m{AC;\zeta}}=\mathfrak{w}_{\m{AC}}|_{M_\zeta}\asymp _{r\to \infty} r$ such that the $\underline{\omega}_\zeta$-Hodge--de Rham operator satisfies the uniform weighted AC Schauder estimate 
\begin{align*}
    ||\eta||_{C^{k+1,\alpha}_{\m{AC},\underline{\omega}_\zeta,\beta}}\leq C_{k,\alpha,\beta}\left(||D_{\underline{\omega}_\zeta}\eta||_{C^{k,\alpha}_{\m{AC},\underline{\omega}_\zeta,\beta}}+||\eta||_{C^{0}_{\m{AC},\underline{\omega}_\zeta,\beta}}\right).
\end{align*}
Here $C_{k,\alpha,\beta}>0$ is independent of $\zeta\in \Theta^{\m{reg}}_{\m{Im}(\mathbb{H})}(\Gamma)$. The kernel of $D_{\underline{\omega}_\zeta}$ consists of smooth two forms in $\mathcal{O}(r^{-4})$ and can be identified with 
\begin{align*}
    \mathcal{H}^2_{-4}(M_\zeta)\cong \mathfrak{h}^2_\Gamma.
\end{align*}
\end{itemize}
 \end{prop}

 \begin{figure}[h!]
  \centering
  \includegraphics[width=0.4\linewidth]{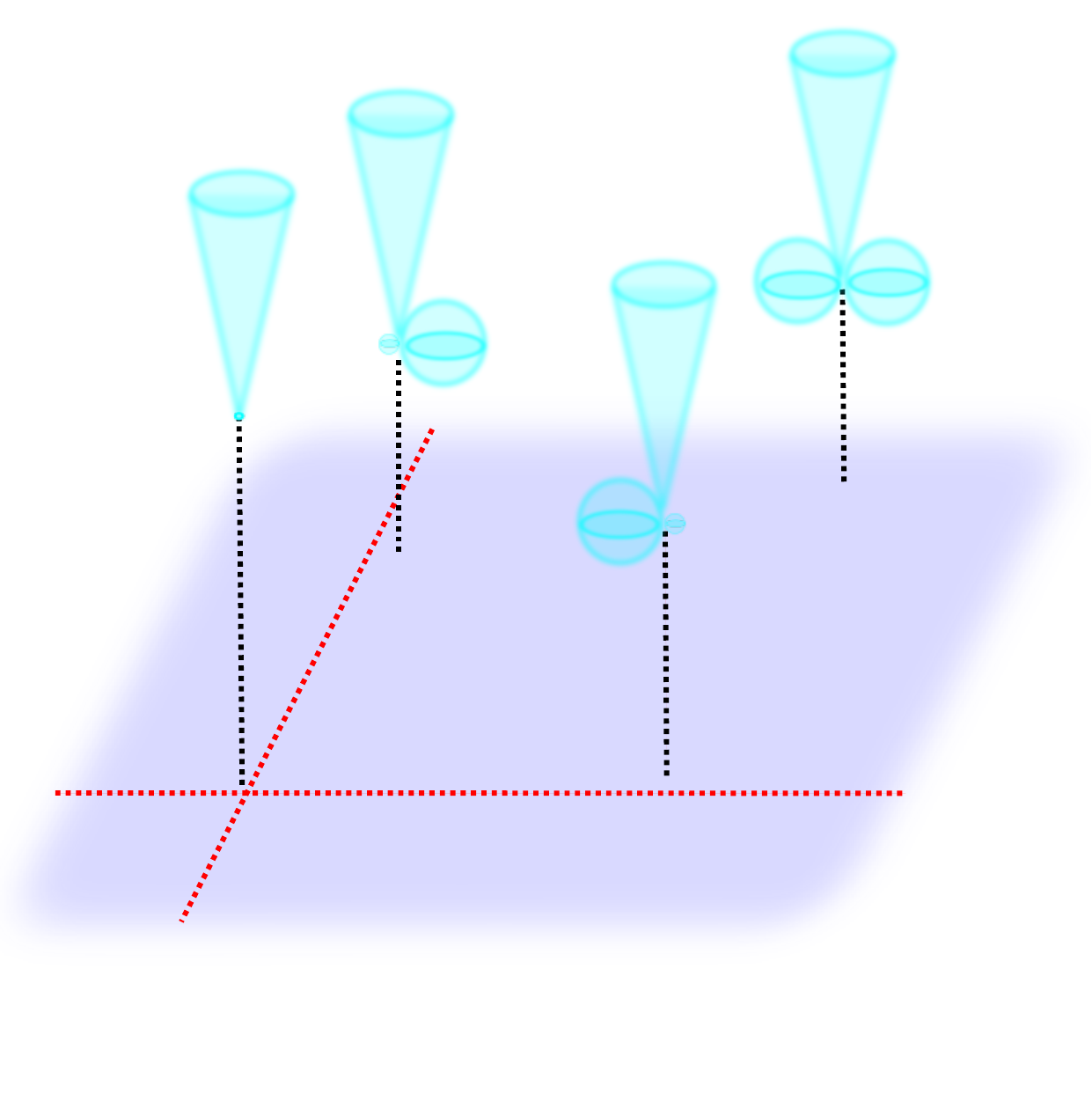}
  \caption{Parameter space of resolutions of $V/\Gamma$.}
\end{figure}

Using the decomposition of $W/\Gamma\cong H\oplus V/\Gamma$ we naturally are able to identify the normaliser of $\Gamma$ in $\m{Spin}(7)$ with the group
\begin{align*}
    \m{N}_{\m{Spin}(7)}(\Gamma)\cong \m{Spin}(7)\cap \m{SO}(H)\times \m{N}_{\m{SO}(V)}(\Gamma)\cong \m{Sp}(1)_{\top}\m{Sp}(1)_{+}\m{N}_{\m{Sp}(1)_{\perp}}(\Gamma)
\end{align*}
By identifying $\mathbb{H}\cong V$, we deduce that $\m{Sp}(1)_+=\m{Sp}(1)_{R}$ and $\m{Sp}(1)_{\perp}=\m{Sp}(1)_{L}$, we recover the action of $\m{N}_{\m{Spin}(7)}(\Gamma)$ factors through $\m{N}_{\m{SO}(4)}(\Gamma)$. The following table includes all possible normalizers and quiver symmetry groups of finite subgroups of $\m{Sp}(1)$.
\begin{align*}
\begin{array}{c|c|c|c}
\text{ADE type} 
& \Gamma\subset \m{Sp}(1) 
& N_{\m{SO}(4)}(\Gamma) 
& \m{Aut}(\Delta_\Gamma)
\\ \hline
A_0 
& C_1=\{1\} 
& \m{SO}(4) 
& 1
\\[4pt]
A_1 
& C_2=\{\pm 1\} 
& \m{SO}(4) 
& 1
\\[4pt]
A_{m-1},\ m\geq 3 
& C_m 
& \m{Pin}(2)\times_{\mathbb{Z}_2} \m{Sp}(1)_R
& \mathbb Z_2
\\[12pt]
D_4 
& Q_8=\m{Dic}_2 
& 2O\times_{\mathbb{Z}_2} \m{Sp}(1)_R 
& S_3
\\[12pt]
D_{n+2},\ n\geq 3 
& \m{Dic}_n 
&\m{Dic}_{2n}\times_{\mathbb{Z}_2} \m{Sp}(1)_R 
& \mathbb Z_2
\\[12pt]
E_6 
& 2T 
& 2O\times_{\mathbb{Z}_2} \m{Sp}(1)_R
& \mathbb Z_2
\\[12pt]
E_7 
& 2O 
& 2O\times_{\mathbb{Z}_2} \m{Sp}(1)_R
& 1
\\[12pt]
E_8 
& 2I 
& 
2I\times_{\mathbb{Z}_2} \m{Sp}(1)_R 
& 1
\end{array}
\end{align*}

Now using the identification of $\m{Im}(\mathbb{H})\cong \wedge^2 _+V^*\cong \wedge^2_-H^*$ as $\m{N}_{\m{Spin}(7)}(\Gamma)$-representations, we define the parameter bundle 
\begin{align*}
    \mathfrak{P}_\Gamma\coloneqq \wedge^2_{g_S,+}T^*S\otimes \mathfrak{H}^2_{\Gamma}.
\end{align*}
Naturally, points of $\mathfrak{P}_\Gamma$ correspond to hyperkähler ALE spaces resolving the normal-cone fibre of its base point in $s$. In particular, using the equivariance established in Proposition \ref{universalMcKaymoduli}, we conclude that there exists a bundle $\mathfrak{N}_\Gamma\coloneqq \m{Fr}(X/S,\Phi)\times_{\m{N}_{\m{Spin}(7)}(\Gamma)}\mathcal{M}_\Gamma$ such that 
\begin{equation*}
    \begin{tikzcd}
	{\mathfrak{N}_\Gamma} && \\
	\\
	{\mathfrak{P}_\Gamma} && S
	\arrow["{\mathfrak{k}_\Gamma}"{description}, from=1-1, to=3-1]
	\arrow["{\mathfrak{n}_\Gamma}"{description}, from=1-1, to=3-3]
	\arrow["{\mathfrak{p}_\Gamma}"{description}, from=3-1, to=3-3]
\end{tikzcd}
\end{equation*}
whose $\mathfrak{n}_\Gamma$-fibres are given by the universal moduli spaces $\mathcal{M}_\Gamma$. Moreover, as the action of $\m{N}_{\m{Spin}(7)}(\Gamma)$ preserves the stratified structure of $\mathcal{M}_\Gamma$ by means of preserving the walls $\m{Im}(\mathbb{H})\otimes\mathcal{W}_\Gamma$, there exists a stratification of $\mathfrak{P}_S$ given by 
\begin{align*}
    \wedge^2_+T^*S\otimes \mathfrak{W}_{\Gamma}\subset \mathfrak{P}_\Gamma
\end{align*}
of singular parameters. These parameters correspond to partial ALE resolutions of the normal cone fibres.\\

We further equip the bundle $\mathfrak{k}_\Gamma \colon\mathfrak{N}_\Gamma\rightarrow \mathfrak{P}_\Gamma $ with the structure of a stratified Riemannian fibration by 
    \begin{align*}
        g_{\mathfrak{N}_\Gamma}\coloneqq g_{\mathfrak{N}_\Gamma}^{2,0}+g_{\mathfrak{N}_\Gamma}^{0,2}=\mathfrak{n}^*_\Gamma g_S+[g_{\mathcal{M}_\Gamma}].
    \end{align*}
and endow it with the \textbf{universal Cayley form }
    \begin{align*}
        \Phi_{\mathfrak{N}_\Gamma}\coloneqq\Phi_{\mathfrak{N}_\Gamma}^{4,0}+\Phi_{\mathfrak{N}_\Gamma}^{2,2}+\Phi_{\mathfrak{N}_\Gamma}^{0,4}
    \end{align*}
    given by \footnote{The Soldering form of $S$ is a form $\theta_S\in \Omega^1(S,TS)$. We identify $[\theta_S\wedge\theta_S]_+$ with the Soldering form of the self-dual two forms on $S$.}
    \begin{align*}
        \Phi_{\mathfrak{N}_\Gamma}^{4,0}\coloneqq&\mathfrak{n}^*_\Gamma\m{vol}_S\\
        \Phi_{\mathfrak{N}_\Gamma}^{2,2}\coloneqq&-\left<\mathfrak{n}^*_\Gamma\left[\theta_{S}\wedge\theta_S\right]_+\wedge[\underline{\omega}_{\mathcal{M}_\Gamma}]\right>\\
        \Phi_{\mathfrak{N}_\Gamma}^{0,4}\coloneqq&\frac{1}{6}\left<[\underline{\omega}_{\mathcal{M}_\Gamma}]\wedge[\underline{\omega}_{\mathcal{M}_\Gamma}]\right>,
    \end{align*}
    and the \textbf{universal ACF-distance} function 
    \begin{align*}
        w_{\m{AC}}=[\mathfrak{w}_{\m{AC}}].
    \end{align*}
    The universal resolution map $\rho_{\mathcal{M}_\Gamma}\colon \mathcal{M}_\Gamma\dashrightarrow M_0$ induces a \textbf{universal resolution map} 
\begin{equation*}
    \begin{tikzcd}
	{{\rho_{\mathfrak{N}_\Gamma}}\colon\mathfrak{N}_\Gamma} & N_0
	\arrow[dashed, two heads, from=1-1, to=1-2]
\end{tikzcd}
\end{equation*}
Notice, that the dilation action $\mathbb{R}_{>0}$-action 
    \begin{align*}
        \delta_t \colon\mathfrak{P}_\Gamma\rightarrow \mathfrak{P}_\Gamma,\, \zeta\mapsto t^2\cdot\zeta
    \end{align*}
    lifts to an action 
    \begin{align*}
        \delta_t \colon\mathfrak{N}_\Gamma\rightarrow\mathfrak{N}_\Gamma
    \end{align*}
   such that 
   \begin{align*}
       \delta_t^*\Phi_{\mathfrak{N}_\Gamma}=\Phi_{\mathfrak{N}_\Gamma}^{4,0}+t^2\cdot\Phi_{\mathfrak{N}_\Gamma}^{2,2}+t^4\cdot\Phi_{\mathfrak{N}_\Gamma}^{0,4}\und{1.0cm}\delta_t^*w_{\m{AC}}=tw_{\m{AC}}.
   \end{align*}
In order to construct an ACF-resolution of the normal cone bundle we have to pick a point in each fibre of $\mathfrak{P}_\Gamma$ varying smoothly. Hence, such resolution are characterised by section 
\begin{align*}
    \zeta\in \Gamma(S,\mathfrak{P}_\Gamma)\cong \Omega^2_+(S,\mathfrak{H}_\Gamma).
\end{align*}
The resolution is then constructed via the pullback
\begin{equation*}
    \begin{tikzcd}
	&& {\mathfrak{N}_\Gamma} && \\
	{N_\zeta} &&&& {N_0} \\
	&& {\mathfrak{P}_\Gamma} \\
	S &&&& S
	\arrow[two heads, from=1-3, to=3-3]
	\arrow[hook, from=2-1, to=1-3]
	\arrow[dotted, from=2-1, to=2-5]
	\arrow["\lrcorner"{anchor=center, pos=0.125}, draw=none, from=2-1, to=3-3]
	\arrow[two heads, from=2-1, to=4-1]
	\arrow[hook', from=2-5, to=1-3]
	\arrow["\lrcorner"{anchor=center, pos=0.125, rotate=-90}, draw=none, from=2-5, to=3-3]
	\arrow[two heads, from=2-5, to=4-5]
	\arrow["\zeta"{description}, hook, from=4-1, to=3-3]
	\arrow["0"{description}, hook', from=4-5, to=3-3]
\end{tikzcd}
\end{equation*}
whereby 
    \begin{align*}
        (N_\zeta,\Phi_\zeta)\coloneqq (\zeta^*\mathfrak{N}_\Gamma,\zeta^*\Phi_{\mathfrak{N}_\Gamma})
    \end{align*}
Naturally, the base point of each fibre of $\mathfrak{P}_\Gamma$ corresponds to its zero section and we identify the bundle
\begin{align*}
    (N_0,\Phi_0)\cong 0^*(\mathfrak{N}_\Gamma,\Phi_{\mathfrak{N}_\Gamma}).
\end{align*}
The space $(N_\zeta,\Phi_\zeta)$ is a smooth noncompact manifold with a smooth $\m{Spin}(7)$-structure for all $\zeta$ such that 
\begin{align*}
    \m{im}(\zeta)\cap \wedge^2T^*_+S\otimes \mathfrak{W}_\Gamma=\emptyset,
\end{align*}
and outside an exceptional set $\Upsilon_\zeta\coloneqq \rho_\zeta^{-1}(0_S)$, which is of codimension two, the resolution map $\rho_\zeta$ defines a diffeomorphism to the regular part of the normal cone bundle.\\

The following proposition states the asymptotic properties of the constructed $\m{Spin}(7)$-structure $\Phi_\zeta$.

\begin{prop}
\label{ACFstructuresfrompullbacks4}
    Let $\zeta\in \Gamma(S,\mathfrak{P}^{reg}_\Gamma)$. The space $(N_\zeta,\Phi_\zeta)$ is an ACF-$\m{Spin}(7)$-resolution 
    \begin{align*}
        \rho_{\zeta}\colon (N_\zeta,\Phi_\zeta)\dashrightarrow(N_0,\Phi_0).
    \end{align*}
    of rate $-4$ with respect to $0^*w_{\m{AC}}=\m{dist}_{g_0}(S)=r$. Moreover, pulling back the universal dilation morphism yields a diffeomorphism 
    \begin{align*}
        \delta_t\colon(N_{t^2\cdot\zeta},\Phi_{t^2\cdot\zeta})\cong (N_\zeta,\Phi^t_\zeta).
    \end{align*}
\end{prop}

\begin{proof}
By \cite[Thm. 8.2.3]{joyce2000compact} there exists an one form $\tau\in\Omega^1(\mathbb{H}/\Gamma,\m{Im}(\mathbb{H}))$ such that 
\begin{align*}
    (\pi_{\zeta(s)})_*\underline{\omega}_{\zeta(s)}-\underline{\omega}_0=\m{d}\underline{\tau}_{\zeta(s)}
\end{align*}
and 
\begin{align*}
    \nabla^k\m{d}\underline{\tau}_\zeta=\mathcal{O}(r^{-4-k})
\end{align*}
we can deduce that 
\begin{align*}
    (\nabla^{\Phi_0})^k\left((\rho_\zeta)_*\Phi_\zeta-\Phi_0\right)=\mathcal{O}(r^{-4})\und{1.0cm}(\nabla^{\Phi_0,V})^k\left((\rho_\zeta)_*\Phi_\zeta-\Phi_0\right)=\mathcal{O}(r^{-4-k})
\end{align*}
as 
\begin{align*}
    (\rho_\zeta)_*\hat{\Phi}_\zeta-\hat{\Phi}_0=&-\left<(\kappa_0)^*\left[\theta_{S}\wedge\theta_S\right]_+\wedge(\rho_\zeta)_*((\kappa^*\overline{\zeta})^*\m{d}^{0,1}\underline{\tau})\right>\\
    &+\frac{1}{6}\left<\left(\tilde{0}^*\underline{\omega}_{\mathcal{M}_\Gamma}+(\rho_\zeta)_*(\kappa^*\overline{\zeta})^*\underline{\omega}_{\mathcal{M}_\Gamma}\right)^{0,2,0}\wedge(\rho_\zeta)_*(\kappa^*\overline{\zeta})^*(\m{d}^{0,1}\underline{\tau})\right>.
\end{align*}
\end{proof}

\subsubsection{Orbifold McKay-Correspondence}
\label{Orbifold McKay-Correspondence}

We now relate the local McKay correspondence to the adiabatic torsion-free condition for resolutions of the normal cone. In the previous sections we constructed, for every regular section
\begin{align*}
\zeta\in\Gamma(S,\mathfrak{P}_\Gamma)
\cong
\Omega^2_+(S,\mathfrak{H}^2_\Gamma),
\end{align*}
an ACF $\m{Spin}(7)$-resolution
\begin{align*}
\rho_\zeta\colon
(N_\zeta,\Phi_\zeta)
\dashrightarrow
(N_0,\Phi_0)
\end{align*}
of the normal cone bundle along the codimension-four stratum $S$. The purpose of this section is to identify the condition under which this resolution is adiabatic torsion-free.\\

The answer is cohomological. Under the local McKay correspondence, the parameter bundle
\begin{align*}
\mathfrak{P}_\Gamma=\wedge^2_+T^S\otimes\mathfrak{H}^2_\Gamma
\end{align*}
is built from the degree-two Chen--Ruan local system. Thus a section $\zeta$ determines a class
\begin{align*}
    [\zeta]\in \m{H}^2(S,\mathfrak{H}^2_\Gamma)\subset \m{H}^4(\m{H}^\bullet(S,\mathfrak{H}^\bullet_\Gamma))
\end{align*}
whenever it is closed. The theorem below says that the adiabatic torsion-free condition is precisely the differential-geometric representative of this cohomological condition: the leading torsion terms vanish exactly when $\zeta$ is harmonic with respect to the flat McKay local system and the metric $g_S$.

\begin{thm}
\label{codimensionfourMcKayDuality}
    Let $[\zeta]\in \m{H}^\bullet(S,\mathfrak{H}^\bullet_\Gamma)$ corresponds to a $g_S$-harmonic $\mathfrak{H}_\Gamma$-twisted differential form $\zeta\in \Gamma(S,\mathfrak{P}_\Gamma)\subset \Omega^\bullet(S,\mathfrak{H}^\bullet_\Gamma)$. Then the ACF-space 
    \begin{align*}
        \rho_\zeta\colon (N_\zeta,\Phi_\zeta)\dashrightarrow (N_0,\Phi_0)
    \end{align*}
    is adiabatic torsion-free. Conversely, $\rho_\zeta\colon (N_\zeta,\Phi_\zeta)\dashrightarrow (N_0,\Phi_0)$ is adiabatic torsion-free if $\zeta\in \Omega^\bullet(S,\mathfrak{H}^\bullet_\Gamma)$ is harmonic and hence determines a unique class $[\Phi]\oplus[\zeta]\in \m{H}^4_{\m{CR}}(X,\mathbb{R})\cong \m{H}^4_{\m{dR}}(X,\mathbb{R})\oplus \m{H}^2(S,\mathfrak{H}^2_\Gamma)$.
\end{thm}

\begin{proof}
We define $\overline{\zeta}\in\Omega^0(\m{Fr}(X/S,\Phi),\Theta_{\m{Im}(\mathbb{H})}(\Gamma))^{\m{N}_{\m{Spin}(7)}(\Gamma)}$ via 
\begin{align*}
    \zeta=\left<\overline{\zeta},\left[\theta_{S}\wedge\theta_S\right]_+\right>
\end{align*}
and the identification of $\Theta_{\m{Im}(\mathbb{H})}(\Gamma)=\m{Im}(\mathbb{H})\otimes\Theta_{\m{Im}(\mathbb{C})}(\Gamma)\cong \m{Im}(\mathbb{H})\otimes \mathfrak{z}(\mathfrak{pu}(R_\Gamma)^{\Gamma})$.\\

We will lift the forms from $N_\zeta$ to the bundle $\overline{\zeta}^*\mathcal{M}_\Gamma$, where $\overline{\zeta}\colon \m{Fr}(X/S,\Phi)\rightarrow \Theta_{\m{Im}(\mathbb{H})}$ denote the $\m{N}_{\m{Spin}(7)}(\Gamma)$-equivariant map lifting $\zeta$.

Now, consider the pullback of $\kappa \colon\mathcal{M}_\Gamma\rightarrow \Theta_{\m{Im}(\mathbb{H})}$ along $\overline{\zeta}$.
    \begin{equation*}
         \begin{tikzcd}
            \overline{\zeta}^*\mathcal{M}_\Gamma\arrow[r,"\kappa^*\overline{\zeta}"]\arrow[d,swap,"\overline{\zeta}^*\kappa"]&\mathcal{M}_\Gamma\arrow[d,"\kappa"]\\
            \m{Fr}(X/S,\Phi)\arrow[r,"\overline{\zeta}"]&\Theta_{\m{Im}(\mathbb{H})}(\Gamma)
        \end{tikzcd}
    \end{equation*}

As the group $\m{N}_{\m{Spin}(7)}(\Gamma)$ acts freely from the right on $\m{Fr}(X/S,\Phi)\times \mathcal{M}_\Gamma$, it acts freely on $\overline{\zeta}^*\mathcal{M}_\Gamma$ and the projection $\psi_\zeta \colon\overline{\zeta}^*\mathcal{M}_\Gamma\rightarrow N_\zeta$ defines a principal $\m{N}_{\m{Spin}(7)}(\Gamma)$-bundle.

\begin{equation*}
    \begin{tikzcd}
        &\overline{\zeta}^*\mathcal{M}_\Gamma\arrow[dl,swap,"\psi_\zeta"]\arrow[r,"\kappa^*\overline{\zeta}"]\arrow[d,"\overline{\zeta}^*\kappa"]&\mathcal{M}_\Gamma\arrow[d,"\kappa"]\\
        N_\zeta\arrow[dr,"\nu_\zeta"]&\m{Fr}(X/S,\Phi)\arrow[r,"\overline{\zeta}"]\arrow[d,"\phi"]&\Theta_{\m{Im}(\mathbb{H})}(\Gamma)\\
        &S&
    \end{tikzcd}
\end{equation*}
where 
\begin{align*}
    \nu_\zeta\colon N_\zeta=\overline{\zeta}^*\mathcal{M}_\Gamma/\m{N}_{\m{Spin}(7)}(\Gamma)\rightarrow S
\end{align*}
denotes the quotient mapping onto $S$.\\

The Levi-Civita connection on $S$ pulls back to a connection $(\overline{\zeta}^*\kappa)^*\varphi=\varphi_\zeta$ on $\psi_\zeta \colon\overline{\zeta}^*\mathcal{M}_\Gamma\rightarrow N_\zeta$. Furthermore, 
the $\m{N}_{\m{Spin}(7)}(\Gamma)$-invariant connection ${\check{H}}$ on $\kappa \colon\mathcal{M}_\Gamma\rightarrow \Theta_{\m{Im}(\mathbb{H})}(\Gamma)$ gives rise to a splitting of the horizontal subbundle of $(\overline{\zeta}^*\kappa)^*\varphi$. Consequently, the tangent bundle of $\overline{\zeta}^*\mathcal{M}_\Gamma$ decomposes into 
\begin{align*}
    T\overline{\zeta}^*\mathcal{M}_\Gamma\cong&(\overline{\zeta}^*\kappa)^*H_S\oplus\overline{\zeta}^*K\oplus\underline{\mathfrak{n}_{\m{Spin}(7)}(\Gamma)}_{\overline{\zeta}^*\mathcal{M}_\Gamma}\\
    \eqqcolon&T^{1,0,0}\overline{\zeta}^*\mathcal{M}_\Gamma\oplus T^{0,1,0}\overline{\zeta}^*\mathcal{M}_\Gamma\oplus T^{0,0,1}\overline{\zeta}^*\mathcal{M}_\Gamma.
\end{align*}
In particular, $(\overline{\zeta}^*\kappa)^*H_S$ defines an $\m{N}_{\m{Spin}(7)}(\Gamma)$ equivariant connection on $\phi\circ\overline{\zeta}^*\kappa \colon\overline{\zeta}^*\mathcal{M}_\Gamma\rightarrow S$, and thus descends to an Ehresmann connection $H_\zeta$ on $\nu_\zeta\colon N_\zeta\rightarrow S$.\\

Differential forms on $\overline{\zeta}^*\mathcal{M}_\Gamma$ decompose into
\begin{align*}
    \Omega^{p,q,r}(\overline{\zeta}^*\mathcal{M}_\Gamma)=\Gamma(\overline{\zeta}^*\mathcal{M}_\Gamma,\wedge^p(\overline{\zeta}^*\kappa)^* H_S^* \otimes \wedge^p(\kappa^*\overline{\zeta})^*K^* \otimes\wedge^r\underline{\mathfrak{n}_{\m{Spin}(7)}(\Gamma)}_{\overline{\zeta}^*\mathcal{M}_\Gamma}^*)
\end{align*}
and the de Rham differential can be identified with 
\begin{equation*}
    \m{d}
    =\begin{cases}
        \m{d}^{1,0}&=\m{d}^{1,0,0}+\m{d}^{0,1,0}+\m{d}^{2,-1,0}\\
        \m{d}^{0,1}&=\m{d}^{0,0,1}\\
        \m{d}^{2,-1}&=\m{d}^{2,0,-1}.
    \end{cases}
\end{equation*}

We identify $\Phi_\zeta$ with a four form on $\overline{\zeta}^*\mathcal{M}_\Gamma$ given by 
\begin{align*}
    \hat{\Phi}_\zeta^{4,0}&=(\phi\circ(\overline{\zeta}^*\kappa))^*\m{vol}_S\\
    \hat{\Phi}_\zeta^{2,2}&=-\left<(\overline{\zeta}^*\kappa)^*\left[\theta_{S}\wedge\theta_S\right]_+\wedge((\kappa^*\overline{\zeta})^*\underline{\omega}_{\mathcal{M}_\Gamma})^{0,2,0}\right>\\
    \hat{\Phi}_\zeta^{0,4}&=\frac{1}{6}\left<((\kappa^*\overline{\zeta})^*\underline{\omega}_{\mathcal{M}_\Gamma})^{0,2,0}\wedge((\kappa^*\overline{\zeta})^*\underline{\omega}_{\mathcal{M}_\Gamma})^{0,2,0}\right>.
\end{align*}
We will now proceed by computing the torsion $\m{d}\Phi_\zeta$ and prove that if $\zeta$ is harmonic, $\m{d}\Phi_\zeta=\m{d}^{2,-1}\Phi^{2,2}_\zeta+\m{d}^{2,-1}\Phi^{0,4}_\zeta$.

First we compute
\begin{align*}
    \psi_\zeta^*\m{d}\Phi_\zeta^{4,0}=&\m{d}(\overline{\zeta}^*\kappa\circ\phi)^*\m{vol}_S\\
    =&\m{d}(\overline{\zeta}^*\kappa\circ\phi)^*\m{vol}_S\\
    =&(\overline{\zeta}^*\kappa\circ\phi)^*\m{d}\m{vol}_S\\
    =&0
\end{align*}
To see that the harmonicity of $\zeta$ implies that $\Phi_\zeta$ defines an adiabatic torsion-free $\m{Spin}(7)$-structure, we compute

\begin{align*}
    \m{d}_{\varphi}\zeta=&\m{d}^{1,0}\left<\left[\theta_{S}\wedge\theta_S\right]_+,\overline{\zeta}\right>=\left<\left[\theta_{S}\wedge\theta_S\right]_+\wedge(\overline{\zeta}^*\theta_{\Theta_{\m{Im}(\mathbb{H})}(\Gamma)})^{1,0}\right>.
\end{align*}
where $\theta_{\Theta_{\m{Im}(\mathbb{H})}(\Gamma)}\in \Omega^1(\Theta_{\m{Im}(\mathbb{H})}(\Gamma),\m{Im}(\mathbb{H})\otimes \mathfrak{z}(\mathfrak{pu}(R_\Gamma)^{\Gamma}))$ denotes the Soldering form on $\Theta_{\m{Im}(\mathbb{H})}(\Gamma)$. We use that $\varphi$ is torsion-free, as well as
\begin{align*}
    \m{d}^{1,0,0}((\kappa^*\overline{\zeta})^*\underline{\omega}_{\mathcal{M}_\Gamma})^{0,2,0}=&(\m{d}(\kappa^*\overline{\zeta})^*\underline{\omega}_{\mathcal{M}_\Gamma})^{1,2,0}-\m{d}^{0,1,0}((\kappa^*\overline{\zeta})^*\underline{\omega}_{\mathcal{M}_\Gamma})^{1,1,0}\\
    =&((\kappa^*\overline{\zeta})^*\m{d}\underline{\omega}_{\mathcal{M}_\Gamma})^{1,2,0}\\
    &-\m{d}^{0,1,0}((\kappa^*\overline{\zeta})^*\underline{\omega}_{\mathcal{M}_\Gamma})^{1,1,0}\\
    =&-\left((\kappa^*\overline{\zeta})^*\left<\kappa^*\theta_{\Theta_{\m{Im}(\mathbb{H})}(\Gamma)}\wedge\pi_{\mathfrak{z}(\mathfrak{pu}(R_\Gamma)^{\Gamma})}F_{\mathfrak{a}_\Gamma}\right>\right)^{1,2,0}\\
    &-\m{d}^{0,1,0}((\kappa^*\overline{\zeta})^*\underline{\omega}_{\mathcal{M}_\Gamma})^{1,1,0}\\
    =&-\left<\overline{\zeta}^*(\kappa^*\overline{\zeta})^*\theta_{\Theta_{\m{Im}(\mathbb{H})}(\Gamma)}\wedge(\kappa^*\overline{\zeta})^*\pi_{\mathfrak{z}(\mathfrak{pu}(R_\Gamma)^{\Gamma})}F_{\mathfrak{a}_\Gamma}\right>^{1,2,0}\\
    &-\m{d}^{0,1,0}((\kappa^*\overline{\zeta})^*\underline{\omega}_{\mathcal{M}_\Gamma})^{1,1,0}\\
    =&-\left<(\overline{\zeta}^*(\kappa^*\overline{\zeta})^*\theta_{\Theta_{\m{Im}(\mathbb{H})}(\Gamma)})^{1,0,0}\wedge((\kappa^*\overline{\zeta})^*\pi_{\mathfrak{z}(\mathfrak{pu}(R_\Gamma)^{\Gamma})}F_{\mathfrak{a}_\Gamma})^{0,2,0}\right>\\
    &-\left<(\overline{\zeta}^*(\kappa^*\overline{\zeta})^*\theta_{\Theta_{\m{Im}(\mathbb{H})}(\Gamma)})^{0,1,0}\wedge((\kappa^*\overline{\zeta})^*\pi_{\mathfrak{z}(\mathfrak{pu}(R_\Gamma)^{\Gamma})}F_{\mathfrak{a}_\Gamma})^{1,1,0}\right>\\
    &-\m{d}^{0,1,0}((\kappa^*\overline{\zeta})^*\underline{\omega}_{\mathcal{M}_\Gamma})^{1,1,0}
\end{align*}
whereby 
\begin{align*}
    \m{d}^{0,1,0}((\kappa^*\overline{\zeta})^*\underline{\omega}_{\mathcal{M}_\Gamma})^{1,1,0}=-\left<(\overline{\zeta}^*(\kappa^*\overline{\zeta})^*\theta_{\Theta_{\m{Im}(\mathbb{H})}(\Gamma)})^{0,1,0}\wedge((\kappa^*\overline{\zeta})^*\pi_{\mathfrak{z}(\mathfrak{pu}(R_\Gamma)^{\Gamma})}F_{\mathfrak{a}_\Gamma})^{1,1,0}\right>
\end{align*}

and 
\begin{align*}
    \m{d}^{0,1,0}((\kappa^*\overline{\zeta})^*\underline{\omega}_{\mathcal{M}_\Gamma})^{0,2,0}=&(\m{d}(\kappa^*\overline{\zeta})^*\underline{\omega}_{\mathcal{M}_\Gamma})^{0,3,0}\\
    =&-\left<(\kappa^*\overline{\zeta})^*\left(\kappa^*\theta_{\Theta_{\m{Im}(\mathbb{H})}(\Gamma)}\wedge\pi_{\mathfrak{z}(\mathfrak{pu}(R_\Gamma)^{\Gamma})}F_{\mathfrak{a}_\Gamma}\right)\right>^{0,3,0}\\
    =&-\left<(\overline{\zeta}^*(\kappa^*\overline{\zeta})^*\theta_{\Theta_{\m{Im}(\mathbb{H})}(\Gamma)})^{0,1,0}\wedge((\kappa^*\overline{\zeta})^*\pi_{\mathfrak{z}(\mathfrak{pu}(R_\Gamma)^{\Gamma})}F_{\mathfrak{a}_\Gamma})^{0,2,0}\right>.
\end{align*}
Hence, we deduce that 

\begin{align*}
    (\m{d}^{1,0,0}+\m{d}^{0,1,0})((\kappa^*\overline{\zeta})^*\underline{\omega}_{\mathcal{M}_\Gamma})^{0,2,0}=&\left<(\overline{\zeta}^*\kappa)^*(\overline{\zeta}^*\theta_{\Theta_{\m{Im}(\mathbb{H})}(\Gamma)})^{1,0}\wedge((\kappa^*\overline{\zeta})^*\pi_{\mathfrak{z}(\mathfrak{pu}(R_\Gamma)^{\Gamma})}F_{\mathfrak{a}_\Gamma})^{0,2,0}\right>.
\end{align*}

Further, as the pull back of the Soldering form satisfies
\begin{align*}
    \m{d}(\overline{\zeta}^*\kappa)^*\left[\theta_{S}\wedge\theta_S\right]_+=&(\overline{\zeta}^*\kappa)^*\m{d}\left[\theta_{S}\wedge\theta_S\right]_+\\
    =&(\overline{\zeta}^*\kappa)^*[\varphi\overset{\circ}{\wedge}\left[\theta_{S}\wedge\theta_S\right]_+]\in\Omega^{2,1}(\overline{\zeta}^*\mathcal{M}_\Gamma,\m{Im}(\mathbb{H}))^{\m{N}_{\m{Spin}(7)}(\Gamma)}
\end{align*}
we deduce, that if $\zeta$ is harmonic
\begin{align*}
    \psi_\zeta^*\m{d}\Phi^{2,2}_\zeta=&-\left<(\overline{\zeta}^*\kappa)^*\left[\theta_{S}\wedge\theta_S\right]_+\wedge\left<(\overline{\zeta}^*\kappa)^*(\overline{\zeta}^*\theta_{\Theta_{\m{Im}(\mathbb{H})}(\Gamma)})^{1,0}\wedge((\kappa^*\overline{\zeta})^*\pi_{\mathfrak{z}(\mathfrak{pu}(R_\Gamma)^{\Gamma})}F_{\mathfrak{a}_\Gamma})^{0,2,0}\right>\right>\\
    &+\psi^*_\zeta\m{d}^{2,-1}\Phi^{2,2}_\zeta\\
    =&-\left<(\overline{\zeta}^*\kappa)^*\underbrace{\left<\left[\theta_{S}\wedge\theta_S\right]_+\wedge(\overline{\zeta}^*\theta_{\Theta_{\m{Im}(\mathbb{H})}(\Gamma)})^{1,0}\right>}_{\m{d}_{\varphi}\zeta=0}\wedge((\kappa^*\overline{\zeta})^*\pi_{\mathfrak{z}(\mathfrak{pu}(R_\Gamma)^{\Gamma})}F_{\mathfrak{a}_\Gamma})^{0,2,0}\right>\\
    &+\psi^*_\zeta\m{d}^{2,-1}\Phi^{2,2}_\zeta\\
    =&\psi^*_\zeta\m{d}^{2,-1}\Phi^{2,2}_\zeta
\end{align*}
Further,
\begin{align*}
    \psi_\zeta^*\m{d}^{1,0}\Phi_\zeta^{0,4}=&\m{d}^{1,0,0}\hat{\Phi}_\zeta^{0,4}\\
    =&\frac{1}{6}\m{d}^{1,0,0}\left<((\kappa^*\overline{\zeta})^*\underline{\omega}_{\mathcal{M}_\Gamma})^{0,2,0}\wedge((\kappa^*\overline{\zeta})^*\underline{\omega}_{\mathcal{M}_\Gamma})^{0,2,0}\right>\\
    =&\frac{1}{3}\left<\m{d}^{1,0,0}((\kappa^*\overline{\zeta})^*\underline{\omega}_{\mathcal{M}_\Gamma})^{0,2,0}\wedge((\kappa^*\overline{\zeta})^*\underline{\omega}_{\mathcal{M}_\Gamma})^{0,2,0}\right>\\
    =&-\frac{1}{3}\left<\left<(\overline{\zeta}^*\kappa)^*(\overline{\zeta}^*\theta_{\Theta_{\m{Im}(\mathbb{H})}(\Gamma)})^{1,0}\wedge((\kappa^*\overline{\zeta})^*\pi_{\mathfrak{z}(\mathfrak{pu}(R_\Gamma)^{\Gamma})}F_{\mathfrak{a}_\Gamma})^{0,2,0}\right>\wedge((\kappa^*\overline{\zeta})^*\underline{\omega}_{\mathcal{M}_\Gamma})^{0,2,0}\right>\\
    =&-\frac{1}{3}\left<(\overline{\zeta}^*\kappa)^*(\overline{\zeta}^*\theta_{\Theta_{\m{Im}(\mathbb{H})}(\Gamma)})^{1,0}\wedge\left<((\kappa^*\overline{\zeta})^*\pi_{\mathfrak{z}(\mathfrak{pu}(R_\Gamma)^{\Gamma})}F_{\mathfrak{a}_\Gamma})^{0,2,0}\wedge((\kappa^*\overline{\zeta})^*\underline{\omega}_{\mathcal{M}_\Gamma})^{0,2,0}\right>\right>\\
    =&-\frac{1}{3}\left<(\overline{\zeta}^*\kappa)^*(\overline{\zeta}^*\theta_{\Theta_{\m{Im}(\mathbb{H})}(\Gamma)})^{1,0}\wedge\left<((\kappa^*\overline{\zeta})^*(\underbrace{\pi_{\mathfrak{z}(\mathfrak{pu}(R_\Gamma)^{\Gamma})}F_{\mathfrak{a}_\Gamma}\wedge\underline{\omega}_{\mathcal{M}_\Gamma}}_{=0}))^{0,4,0}\right>\right>\\
    =&0
\end{align*}

vanishes as a consequence of $F_{\mathfrak{a}_\Gamma}$ being fibrewise ASD.
\end{proof}

\begin{thm}
\label{alladiabaticresolutions}
    Every smooth adiabatic torsion-free $\m{Spin}(7)$-ACF-resolution of rate $-4$ comes from the above construction.
\end{thm}

\begin{proof}
    Let $\rho'\colon (N',\Phi')\dashrightarrow (N_0,\Phi_0)$ be a smooth adiabatic torsion-free $\m{Spin}(7)$-ACF-resolution of rate $-4$. Let $\iota_s\colon \mathbb{H}/\Gamma\hookrightarrow N_0$ denote the fibre inclusion at $s\in S$. The restriction $M_s\coloneqq \rho'^{-1}(\mathbb{H}/\Gamma)$, defines a smooth resolution 
    \begin{align*}
        \rho'_s\colon M_s\dashrightarrow \mathbb{H}/\Gamma
    \end{align*}
    and metrically is a smooth ALE space of rate $-4$. Moreover, for every $0\neq \beta\in \wedge^2_+T^*_sS$, $\left<\beta,\Phi'|_{s}\right>\asymp_{r\to \infty} \left<\beta,\Phi_0|_{s}\right>$ is a torsion-free hyperkähler resolution as $\Phi'$ is adiabatic torsion-free and the pullback preserves the closed-ness. Consequently, using the universal moduli description of \cite{kronheimer1989construction} the fibre has to be diffeomorphic to the ADE-quiver variety of $\Gamma$ and $\left<\beta,\Phi'_s\right>$ can be assigned to a unique parameter inside the Weyl-chambers of $\m{H}^2(M')\cong \m{h}_\Gamma$. Since $(N',\Phi')$ is smooth, the assignment yields a class $[\Phi']\in \Omega^2_+(S,\m{H}^2(N'/S))$ such that $\m{im}([\Phi'])$ does not intersect the bundle walls. Moreover, by the equivalence of the universal moduli space $\mathcal{M}_\Gamma$ under the action of $\m{N}_{\m{Spin}(7)}(\Gamma)$ the bundle $\m{H}^2(N'/S)\cong \m{Fr}(X/S,\Phi)\times_{\m{N}_{\m{Spin}(7)}(\Gamma)} \mathfrak{h}_\Gamma\cong \mathfrak{H}^2_\Gamma$. By Theorem \ref{codimensionfourMcKayDuality} the corresponding section has to be harmonic and we conclude the statement.
\end{proof}

\begin{thm}
\label{isentropicityequalcrepantresolution}
    Let $(X,\Phi)$ be a compact $\m{Spin}(7)$-orbifold whose isotropy groups are conjugate to subgroups of $\m{SU}(n)$ and let 
    \begin{align*}
        \rho_{t}\colon (X_t,\Phi_t)\dashrightarrow (X,\Phi)
    \end{align*}
   be a $1$-tame $\m{Spin}(7)$-orbifold resolution with respect to an interpolating $\Phi_t'$. Then the isentropicity of $(X_t,\Phi_t')$ is equivalent to the isomorphism 
    \begin{align*}
        \m{H}^\bullet_{\m{CR}}(X,\mathbb{R})\cong \m{H}^\bullet(X_t,\mathbb{R}).
    \end{align*}
\end{thm}

\section{Existence of Torsion-Free $\m{Spin}(7)$ Structures on Orbifold Resolutions}
\label{Existence of torsion-free Spin(7) Structures on Orbifold Resolutions}

In the following section we will solve the fixpoint problem \eqref{fixpointproblem} to deform the pre-glued $\m{Spin}(7)$-structures on $X_{\zeta;t}$ constructed in Section \ref{Adiabatic Spin(7)-Orbifold Resolutions} to a torsion-free structure. In contrast to Joyce' general existence result of torsion-free $\m{Spin}(7)$ structures deforming those whose torsion has a small potential, we will use the adapted norms constructed in \cite{majewskiDirac} to prove an existence result for resolutions of $\m{Spin}(7)$-orbifolds. This existence result will only depend on the smallness of the torsion of the pre-glued $\m{Spin}(7)$-structure $\Phi^{pre}_{\zeta;t}$.

\subsection{Resolutions of Spin(7)-Orbifolds}
\label{Resolutions of Spin(7)-Orbifolds}
In the following section we will describe the data needed to resolve a $\m{Spin}(7)$-orbifold.
\begin{manualassumption}{1}
\label{assTypeiSpin7}    
Let us assume that there exists a family of torsion-free ACF-$\m{Spin}(7)$-spaces $(N_{t^2\cdot\zeta},\Phi_{t^2\cdot \zeta})$ resolving the family of CF-$\m{Spin}(7)$-orbifold $(N_0,\Phi_0)$, i.e. there exists maps 
\begin{equation*}
\begin{tikzcd}
        (N_\zeta,\Phi^t_\zeta)\arrow[d,dashed,"\rho_\zeta"]\arrow[r,"\delta_t"]&(N_{t^2\cdot\zeta},\Phi_{t^2\cdot\zeta})\arrow[d,dashed,"\rho_{t^2\cdot\zeta}"]\\
        (N_0,\Phi^t_0)\arrow[r,"\delta_t"]&(N_0,\Phi_0)
\end{tikzcd}
\end{equation*}
whose exceptional set $\rho_\zeta^{-1}(0)=\Upsilon_\zeta$ is of codimension $>0$ such that
\begin{align*}
    \Phi_\zeta\in \mathfrak{Ker}_{\m{ACF};0}(\widehat{D}_{\zeta})\cap\Gamma(N_\zeta,Cay_+(N_\zeta).
\end{align*}
In particular, the torsion of $\Phi^t_\zeta$ satisfies both
\begin{align*}
    \Phi_{t^2\cdot\zeta}\xrightarrow[C^\infty_{loc}]{t\to0}\Phi_0\und{1.0cm}\m{d}\Phi_{t^2\cdot\zeta}\xrightarrow[C^\infty_{loc}]{t\to0}\m{d}\Phi_0\,\text{ on }N_\zeta\backslash\Upsilon_\zeta.
\end{align*}
The Riemannian structure associated to $\Phi_{t^2\cdot\zeta}$ decomposes into 
\begin{align*}
    \delta^*_tg_{t^2\cdot \zeta}=\nu_\zeta^*g_S+t^2\cdot g_\zeta.
\end{align*}
Moreover $\Phi_{t^2\cdot \zeta}$ is of rate $\upsilon=-m$, i.e.
\begin{align*}
    (\nabla^{\Phi^t_{V,\nu}})^k((\rho_\zeta)_*\Phi_{t^2\cdot\zeta;V}-\Phi_{0;V})=&\mathcal{O}(r^{-m-k})\\
    (\nabla^{g_0})^k((\rho_\zeta)_*\Phi_{t^2\cdot\zeta}-\Phi_{0})=&\mathcal{O}(r^{-m}).
\end{align*}
\end{manualassumption}

\begin{defi}
    We interpolate between the torsion-free $\m{Spin}(7)$ orbifold structure on $X$ and the adiabatic torsion-free ACF $\m{Spin}(7)$-structure on $B_{3\epsilon}(N_\zeta)$ using the cut-off function ${\chi_3}$, i.e.
\begin{align}
    \label{pregluedspin7}\Phi^{pre}_{\zeta;t}=\Theta(\Phi_{t^2\cdot\zeta}\cup_t\Phi).
\end{align}
\end{defi}

\begin{lem}
\label{pregluedmetricclosetoDiracmetric}
Let $g^{pre}_{\zeta;t}\coloneqq g_{t^2\cdot\zeta}\cup_t g$ be the Riemannian metric on $X_{t;\zeta}$ constructed via interpolating the Riemannian orbifold structure and the ACF-structure on $N_{t^2\cdot\zeta}$. Then $g_{\Phi^{pre}_{\zeta;t}}$ is $2$-tame with respect to $g^{pre}_{\zeta;t}$, i.e. 
\begin{align*}
    \left|\left|\cup_t(g_{\Phi^{pre}_{\zeta;t}}-g^{pre}_{\zeta;t})\right|\right|_{C^{2}_{\m{ACF};-1;t}\oplus C^2_{\m{CFS};-1;\epsilon}}\lesssim t^\lambda.
\end{align*}
\end{lem}

\begin{proof}
The map $\gamma\colon Cay_+\rightarrow \m{Met}_+$ is analytic and hence, by using the expansion of $\Theta$ we deduce that  
    \begin{align*}
        \gamma(\Phi^{pre}_{\zeta;t})=&\gamma(\Phi)+T_\Phi\gamma\left(\pi_{\Phi,\top}\left((1-\chi_3)(\Phi_{t^2\cdot\zeta}-\Phi\right))\right)+h.o.t.\\
        \gamma(\Phi^{pre}_{\zeta;t})=&\gamma(\Phi_{t^2\cdot\zeta})+T_{\Phi_{t^2\cdot\zeta}}\gamma\left(\pi_{\Phi_{t^2\cdot\zeta},\top}\left(\chi_3(\Phi-\Phi_{t^2\cdot\zeta})\right)\right)+h.o.t.
    \end{align*}
Then on $X_{\zeta;t}\backslash U_{3\epsilon}$
    \begin{align*}
        g_{\Phi^{pre}_{\zeta;t}}-g^{pre}_{\zeta;t}=&\gamma(\Phi^{pre}_{\zeta;t})-\gamma(\Phi)-(1-\chi_3)(\gamma(\Phi_{t^2\cdot\zeta})-\gamma(\Phi))\\
        =&(1-\chi_3)\left\{T_{\Phi}\gamma\left(\pi_{\Phi,\top}\left(\Phi_{t^2\cdot\zeta}-\Phi\right)\right)-\gamma(\Phi_{t^2\cdot\zeta})-\gamma(\Phi)\right\}+h.o.t.\\
        =&(1-\chi_3)\left\{T_{\Phi}\gamma\circ\pi_{\Phi,\top}-T_{\Phi_{t;\sigma}}\gamma\right\}(\Phi_{t^2\cdot\zeta}-\Phi)+h.o.t.
    \end{align*}
where $\Phi_{t;\sigma}$ is a path of Cayley forms connecting $\Phi_{t^2\cdot\zeta}$ and $\Phi$ . In the same way, on $U_{4\epsilon}$ we have 
\begin{align*}
        g_{\Phi_{\zeta;t}}-g^{pre}_{\zeta;t}=&\gamma(\Phi^{pre}_{\zeta;t})-\gamma(\Phi_{t^2\cdot\zeta})-\chi_3(\gamma(\Phi)-\gamma(\Phi_{t^2\cdot\zeta}))\\
        =&\chi_3\left\{T_{\Phi_{t^2\cdot\zeta}}\gamma\left(\pi_{\Phi_{t^2\cdot\zeta},\top}\left(\Phi-\Phi_{t^2\cdot\zeta}\right)\right)-\gamma(\Phi)-\gamma(\Phi_{t^2\cdot\zeta})\right\}+h.o.t.\\
        =&\chi_3\left\{T_{\Phi_{t^2\cdot\zeta}}\gamma\circ \pi_{\Phi_{t^2\cdot\zeta},\top}-T_{\Phi_{s}}\gamma\right\}(\Phi_{t^2\cdot\zeta}-\Phi)+h.o.t.\quad.
\end{align*} 
The $C^{2}_{\m{ACF};-1;t}\oplus C^2_{\m{CFS};-1;\epsilon}$-bound follows directly.
\end{proof}

By now using the $2$-tameness of $g_{\Phi^{pre}_{\zeta;t}}$ Proposition \ref{tameuniformboundedness} implies the following corollary.

\begin{cor}
\label{uniformlyboundedrightinverse}
    If $\rho_{\zeta;t}\colon(X_{\zeta;t},\Phi^{pre}_{\zeta;t})\dashrightarrow (X,\Phi)$ is isentropic, the Hodge--de Rham operator $D_{\Phi^{pre}_{\zeta;t}}$ admits a right-inverse satisfying 
    \begin{align*}
        \left|\left|R_{\Phi^{pre}_{\zeta;t}}\Psi\right|\right|_{\mathfrak{D}^{2,\alpha}_{\beta;t}}\lesssim\left|\left|\Psi\right|\right|_{\mathfrak{C}^{1,\alpha}_{\beta;t}}.
    \end{align*}
\end{cor}

Before we start constructing adiabatic torsion-free ACF spaces $(N_\zeta,\Phi_\zeta)$ resolving the normal cone, and proving the existence result of torsion-free orbifold resolutions, we will prove an a priori bound on the torsion of the pre-glued $\m{Spin}(7)$-structure, assuming the data for Assumption \ref{assTypeiSpin7} is given.


\begin{nota}
    Let $\vartheta$ be a positive real number, satisfying the following 
        \begin{align}
        \label{vartheta}
        \vartheta < \min\left\{
                \begin{array}{l}
                \lambda(2- \beta),\\
                -\lambda\beta+ (1-\lambda)m, \\
                \frac{m}{2} - \kappa +  \lambda,\\
                \frac{m}{2} - \kappa +(1-\lambda)m - \lambda.
                \end{array}
                \right.
    \end{align}
\end{nota}

\begin{prop}[A Priori Bound]
\label{apprioriboundonlemma}
The pre-glued $\m{Spin}(7)$-structure $\Phi^{pre}_{\zeta;t}$ constructed by interpolating the orbifold $\m{Spin}(7)$-structure $\Phi$ and the adiabatic torsion-free ACF-structure $\Phi_\zeta$, resolving the CF-normal cone structure $\Phi_0$ assumed in Assumption \ref{assTypeiSpin7}, satisfies the priori bound 
\begin{align}
\label{aprioribound}
    \left|\left|\m{d}\Phi^{pre}_{\zeta;t}\right|\right|_{\mathfrak{C}^{0,\alpha}_{\beta-1;t}}\lesssim&t^{\vartheta}.
\end{align}
\end{prop}

\begin{rem}
We will establish this a priori bound in a general form, so that the optimal bound in specific examples can be easily determined.
\end{rem}

\begin{proof}
We begin by bounding the ACF-part of the torsion of $\Phi^{pre}_{\zeta;t}$. Recall, 
\begin{align*}
    \pi_{\Phi,\top}=\pi_{1\oplus 7\oplus 35;\Phi}
\end{align*}
is the projection onto the components that are tangential at $\Phi$ to the bundle of $\m{Spin}(7)$ structures and 
\begin{align*}
    \pi_{\Phi,\perp}=1-\pi_{\Phi,\top}
\end{align*}
denotes the projection on to the normal components.\\

We begin by bounding the ACF-part of the torsion of $\Phi^{pre}_{\zeta;t}$. Notice, that we can expand $\Phi^{pre}_{\zeta;t}$ by
\begin{align*}
    \Theta(\widetilde{\Phi}^{pre;t}_{\zeta})=&\Theta(\Phi^t_{\zeta}+\chi^t_3(\Phi^t-\Phi^t_\zeta))\\
    =&\Phi^t_{\zeta}+\pi_{\Phi^t_\zeta,\top}\left(\chi^t_3(\Phi^t-\Phi^t_\zeta)\right)+Q_{\Phi^t_\zeta}\left(\chi^t_3(\Phi^t-\Phi^t_\zeta)\right)\\
    =&\widetilde{\Phi}^{pre;t}_{\zeta}+\pi_{\Phi^t_\zeta,\perp}\left(\chi^t_3(\Phi^t-\Phi^t_\zeta)\right)+Q_{\Phi^t_\zeta}\left(\chi^t_3(\Phi^t-\Phi^t_\zeta)\right)
\end{align*}
where by $\widetilde{\Phi}^{pre;t}_{\zeta}=\Phi^t_{\zeta}+\chi^t_3(\Phi^t-\Phi^t_\zeta)=\delta_t^*(\Phi+(1-\chi_3)(\Phi_{t^2\cdot\zeta}-\Phi))$.\\

Hence, we bound the ACF-part of the torsion of $\Phi^{pre}_{\zeta;t}$ by\\
\begin{minipage}{\linewidth}
\begin{align*}
    \left|\left|(1-\chi^t_4)\m{d}\Phi^{pre;t}_{\zeta}\right|\right|_{\mathfrak{C}^{0,\alpha}_{\m{ACF};\beta-1;t}}\lesssim &\underbrace{\left|\left|(1-\chi^t_4)\m{d}\widetilde{\Phi}^{pre}_{\zeta;t}\right|\right|_{\mathfrak{C}^{0,\alpha}_{\m{ACF};\beta-1;t}}}_{(I)}\\
    &+\underbrace{\left|\left|(1-\chi^t_4)\m{d}\left(\pi_{\Phi^t_\zeta,\perp}\left(\chi^t_3(\Phi^t-\Phi^t_\zeta)\right)\right)\right|\right|_{\mathfrak{C}^{0,\alpha}_{\m{ACF};\beta-1;t}}}_{(II)}\\
    &+\underbrace{\left|\left|(1-\chi^t_4)\m{d}\left(Q_{\Phi^t_\zeta}\left(\chi^t_3(\Phi^t-\Phi^t_\zeta)\right)\right)\right|\right|_{\mathfrak{C}^{0,\alpha}_{\m{ACF};\beta-1;t}}}_{(III)}.
\end{align*}
\end{minipage}

Since,
\begin{align*}
        \m{d}\widetilde{\Phi}^{pre;t}_{\zeta}=&\m{d}\chi^t_3\wedge(\Phi^t-\Phi^t_\zeta)+(1-\chi^t_3)\m{d}\Phi^t_{\zeta}
\end{align*}
we are able to bound

\begin{align*}
    (I)=&\left|\left|(1-\chi^t_4)\m{d}\widetilde{\Phi}^{pre;t}_{\zeta}\right|\right|_{\mathfrak{C}^{0,\alpha}_{\m{ACF};\beta-1;t}}\\
    \lesssim&\max\left\{t^{-\lambda\beta+(1-\lambda)m},t^{m/2-\kappa-\lambda+(1-\lambda)m},\max\left\{t^{1+\lambda(1-\beta)},t^{1+m/2-\kappa}\right\}\right\}.
\end{align*}
We further bound 
\begin{align*}
    (II)\lesssim&\underbrace{\left|\left|\m{d}(1-\chi^t_4)\left(\pi_{\Phi^t_\zeta,\perp}\left(\chi^t_3(\Phi^t-\Phi^t_\zeta)\right)\right)\right|\right|_{C^{0,\alpha}_{\m{ACF};\beta-1;t}}}_{(i)}\\
    &+\underbrace{\left|\left|\m{d}\chi^t_4\wedge\left(\pi_{\Phi^t_\zeta,\perp}\left(\chi^t_3(\Phi^t-\Phi^t_\zeta)\right)\right)\right|\right|_{\mathfrak{C}^{0,\alpha}_{\m{ACF};\beta-1;t}}}_{(ii)}\\
\end{align*}
whereby 
\begin{align*}
    (i)\lesssim&\underbrace{\left|\left|(1-\chi^t_4)\m{d}\left(\pi_{\Phi^t_\zeta,\perp}\left(\chi^t_3(\Phi^t-\Phi^t_\zeta)\right)\right)\right|\right|_{C^{0,\alpha}_{\m{ACF};\beta-1;t}}}_{\lesssim t^{\lambda(1-\beta)}\max\{t^\lambda,t^{(1-\lambda)m}\}}\\
    &+\underbrace{\left|\left|\m{d}\chi^t_4\wedge\left(\pi_{\Phi^t_\zeta,\perp}\left(\chi^t_3(\Phi^t-\Phi^t_\zeta)\right)\right)\right|\right|_{C^{0,\alpha}_{\m{ACF};\beta-1;t}}}_{\lesssim t^{-\lambda\beta}\max\{t^{2\lambda},t^{(1-\lambda)m}\}},
\end{align*}
where we used that 
\begin{align*}
    \pi_{\Phi^t_\zeta,\perp}\Phi^t_{hot}=&\pi_{\Phi^t_0,\perp}\Phi^t_{hot}+(\pi_{\Phi^t_\zeta,\perp}-\pi_{\Phi^t_0,\perp})\Phi^t_{hot}\\
    =&\mathcal{O}(t^2r^2)+\mathcal{O}(r^{-m})
\end{align*}
and hence 

\begin{align*}
    (ii)\lesssim&\left|\left|\pi_{\mathcal{I};\beta-1}\left(\m{d}\chi^t_4\wedge\left(\pi_{\Phi^t_\zeta,\perp}\left(\chi^t_3(\Phi^t-\Phi^t_\zeta)\right)\right)\right)\right|\right|_{C^{0,\alpha}_{\m{ACF};\beta-1;t}}\\
    &+t^{-\kappa}\left|\left|\pi_{\mathcal{C}o\mathcal{K};\beta-1}\left(\m{d}\chi^t_4\wedge\left(\pi_{\Phi^t_\zeta,\perp}\left(\chi^t_3(\Phi^t-\Phi^t_\zeta)\right)\right)\right)\right|\right|_{C^{0,\alpha}_{t}}\\
    \lesssim& \max\left\{t^{-\lambda\beta}\max\{t^{2\lambda},t^{(1-\lambda)m}\},t^{m/2-\kappa}\max\{t^\lambda,t^{(1-\lambda)m-\lambda}\}\right\}
\end{align*}
Finally, we bound 
\begin{align*}
    (III)\lesssim&\left|\left|(1-\chi^t_4)\m{d}\left(Q_{\Phi^t_\zeta}\left(\chi^t_3(\Phi^t-\Phi^t_\zeta)\right)\right)\right|\right|_{C^{0,\alpha}_{\m{ACF};\beta-1;t}}\\
    &+\left|\left|\m{d}\chi^t_4\wedge\left(Q_{\Phi^t_\zeta}\left(\chi^t_3(\Phi^t-\Phi^t_\zeta)\right)\right)\right|\right|_{\mathfrak{C}^{0,\alpha}_{\m{ACF};\beta-1;t}}\\
    \lesssim&\max\left\{\left|\left|\m{d}\Phi^t_\zeta\right|\right|_{C^{0,\alpha}_{\m{ACF};\beta-1;t}},\left|\left|1\right|\right|_{C^{0,\alpha}_{\m{ACF};\beta-1;t}}\right\}\left|\left|\chi^t_3(\Phi^t-\Phi^t_\zeta)\right|\right|_{C^{1,\alpha}_{\m{ACF};0;t}}^2\\
    &+t^{-\kappa+m/2-\lambda}\left|\left|\chi^t_3(\Phi^t-\Phi^t_\zeta)\right|\right|_{C^{0,\alpha}_{\m{ACF};0;t}}^2\\
    \lesssim&\max\left\{t^{\lambda(1-\beta)}\max\{t^{2(1-\lambda)m},t^{2\lambda}\},t^{m/2-\kappa-\lambda}\max\{t^{(1-\lambda)m},t^{\lambda}\}\right\}
\end{align*}

Combing these estimates, we deduce that 

\begin{align*}
    \left|\left|(1-\chi^t_4)\m{d}\Phi^{pre}_{\zeta;t}\right|\right|_{\mathfrak{C}^{0,\alpha}_{\m{ACF};\beta-1;t}}\lesssim t^\vartheta.
\end{align*}
We will continue with the CFS-part of the a priori estimate on the torsion of the pre-glued $\m{Spin}(7)$-structure $\Phi^{pre}_{\zeta;t}$. We expand 
\begin{align*}
    \Phi^{pre}_{\zeta;t}=&\Theta(\widetilde{\Phi}^{pre}_{\zeta;t})\\
    =&\Theta(\Phi+(1-\chi_3)(\Phi_{t^2\cdot\zeta}-\Phi))\\
    =&\Phi+\pi_{\Phi,\top}\left((1-\chi_3)(\Phi_{t^2\cdot\zeta}-\Phi)\right)+Q_{\Phi}\left((1-\chi_3)(\Phi_{t^2\cdot\zeta}-\Phi)\right)
\end{align*}
and deduce that 
\begin{align*}
    \m{d}\Phi^{pre}_{\zeta;t}=\m{d}\pi_{\Phi,\top}\left((1-\chi_3)(\Phi_{t^2\cdot\zeta}-\Phi)\right)+\m{d}Q_{\Phi}\left((1-\chi_3)(\Phi_{t^2\cdot\zeta}-\Phi)\right).
\end{align*}

\begin{align*}
    \left|\left|\chi_2\m{d}\Phi^{pre}_{\zeta;t}\right|\right|_{C^{0,\alpha}_{\m{CFS};\beta-1;t}}   \lesssim&\left|\left|\chi_2\m{d}\pi_{\Phi,\top}\left((1-\chi_3)(\Phi_{t^2\cdot\zeta}-\Phi)\right)\right|\right|_{C^{0,\alpha}_{\m{CFS};\beta-1;t}}\\
    &+\left|\left|\chi_2\m{d}Q_{\Phi}\left((1-\chi_3)(\Phi_{t^2\cdot\zeta}-\Phi)\right)\right|\right|_{C^{0,\alpha}_{\m{CFS};\beta-1;t}}\\
    \lesssim&\max\left\{t^{\lambda(2-\beta)},t^{\lambda(2-\beta)+(1-\lambda)m}\right\}\lesssim t^{\lambda(2-\beta)}
\end{align*}
where we used Lemma \ref{expansionOmegaProperties} to expand 
\begin{align*} 
    \pi_{\Phi,\top}=\pi_{\tau,\Phi_0}+\left(\underbrace{\pi_{\Phi,\top}-\pi_{\tau,\Phi_0}}_{\mathcal{O}(r^2)}\right).
\end{align*}

Combining the ACF and CFS-part of the estimate we deduce the a priori bound 
\begin{align*}
    \left|\left|\m{d}\Phi^{pre}_{\zeta;t}\right|\right|_{\mathfrak{C}^{0,\alpha}_{\beta-1;t}}\lesssim&t^{\vartheta}
\end{align*}
on the torsion of the preglued $\m{Spin}(7)$-structure $\Phi^{pre}_{\zeta;t}$.
\end{proof}

\begin{rem}

As $\vartheta>0$ 
\begin{align*}
    \left|\left|\m{d}\Phi^{pre}_{\zeta;t}\right|\right|_{\mathfrak{C}^{0,\alpha}_{\beta-1;t}}\xrightarrow[t\to0]{}0.
\end{align*}
Hence, we will refer to 
\begin{align*}
    [\rho_{\zeta;t}\colon(X_{\zeta;t},\Phi^{pre}_{\zeta;t})\dashrightarrow (X,\Phi)]\in \cat{GHRes}(X,\Phi)
\end{align*}
as an adiabatic torsion-free $\m{Spin}(7)$-orbifold resolution.
\end{rem}

\subsection{Deformation to Torsion-Free $\m{Spin}(7)$ Structures on $X_{\zeta;t}$}
\label{Deformation to torsion-free Spin(7) Structures on widetildeXzeta;t}

This section will be devoted to proving the main theorem of this paper. We will use the \say{a priori bound} from section \ref{Adiabatic Spin(7)-Orbifold Resolutions} to show that the only obstruction to deform the pre-glued structure to a torsion-free structure is $\m{ob}_{\beta;t}=0$.\\

Recall that $\vartheta$ is a positive real number, satisfying 
        \begin{align*}
        \vartheta < \min\left\{
                \begin{array}{l}
                \lambda(2- \beta),\\
                -\lambda\beta+ (1-\lambda)m, \\
                \frac{m}{2} - \kappa +  \lambda,\\
                \frac{m}{2} - \kappa +(1-\lambda)m - \lambda.
                \end{array}
                \right.
    \end{align*}

\begin{manualassumption}{8}
\label{ass8Spin7}
    We assume that the resolution $\rho_{\zeta;t}\colon (X_{\zeta;t},\Phi^{pre}_{\zeta;t})\dashrightarrow(X,\Phi)$ is an isentropic resolution. 
\end{manualassumption}

Following this we will now use the derived results to proof the following main theorem.

\begin{thm}[Existence of Torsion-Free $\m{Spin}(7)$-Orbifold Resolutions]
\label{mainthm}
Let 
\begin{align*}
    (X_{\zeta;t},\Phi^{pre}_{\zeta;t})\dashrightarrow (X,\Phi)
\end{align*} denote the adiabatic-torsion-free $\m{Spin}(7)$-orbifold resolution constructed in Section \ref{Adiabatic Spin(7)-Orbifold Resolutions}, such that
\begin{align}
    \left|\left|\m{d}\Phi^{pre}_{\zeta;t}\right|\right|_{\mathfrak{C}^{0,\alpha}_{\beta-1;t}}\lesssim t^{\vartheta}
\end{align}
where $\vartheta$ satisfies \eqref{vartheta}. If 
\begin{align*}
    \rho_{\zeta;t}\colon (X_{\zeta;t},\Phi_{\zeta;t})\dashrightarrow(X,\Phi)
\end{align*}
is an isentropic resolution, there exists a torsion-free $\m{Spin}(7)$-orbifold resolution
\begin{align*}
    [\rho_{\zeta;t}\colon (X_{\zeta;t},\Phi_{\zeta;t})\dashrightarrow(X,\Phi)]\in\cat{GHRes}(X,\Phi)
\end{align*} 
satisfying 
\begin{align}
    \left|\left|\Phi_{\zeta;t}-\Phi^{pre}_{\zeta;t}\right|\right|_{\mathfrak{D}^{1,\alpha}_{\beta;t}}&\lesssim t^{\vartheta}
\end{align}
Here $\Phi_{\zeta;t}=\Theta(\Phi^{pre}_{\zeta;t}+\eta_{\zeta;t})$ where $\eta_{\zeta;t}$ is a unique solution to the fixpoint problem (\ref{fixpointproblem}).
\end{thm}

As discussed in Section \ref{The Spin(7)-Deformation Problem} we will essentially use the contraction mapping principle in a suitable function space to obtain a solution and then argue that this solution is already a smooth one.

\begin{proof}[Proof of Theorem \ref{mainthm}]
In order to improve the readability of the following prove, we will drop the subscript $\zeta$ everywhere.\\

In the following we will show that the map
\begin{align*}
    \mathcal{C}_{t} \colon\mathfrak{C}^{0,\alpha}_{\beta-1;t}(X_{t},\wedge^\bullet T^*X_{t})\rightarrow\mathfrak{C}^{0,\alpha}_{\beta-1;t}(X_{t},\wedge^\bullet T^*X_{t})
\end{align*}
\begin{align*}
       \eta\mapsto &-D^{op}_{\Phi^{pre}_{t}}\Phi^{pre}_{t}-D_{\Phi_{t}^{pre}}\pi_{-;\Phi^{pre}_t}\left\{Q_{\Phi_{t}^{pre}}(R_{\Phi^{pre}_{t}}\eta)\right\}\\
       &-D^{op}_{\Phi_{t}^{pre}}\pi_{+;\Phi^{pre}_{t}}\left\{Q_{\Phi_{t}^{pre}}(R_{\Phi^{pre}_{t}}\eta)\right\}
\end{align*}
restricted to a small enough neighbourhood of $0\in\mathfrak{C}^{0,\alpha}_{\beta-1;t}$, defines a contractive self-mapping. By using the contraction  mapping principal and the ellipticity, we will deduce the existence and smoothness of the fixpoint.\\

Firstly, by Proposition \ref{apprioriboundonlemma} we know that 
\begin{align*}
    \left|\left|\mathcal{C}_{t}(0)\right|\right|_{\mathfrak{C}^{0,\alpha}_{\beta-1;t}}=\left|\left|D^{op}_{\Phi_{t}^{pre}}\Phi^{pre}_{t}\right|\right|_{\mathfrak{C}^{0,\alpha}_{\beta-1;t}}\asymp\left|\left|\m{d}\Phi^{pre}_{t}\right|\right|_{\mathfrak{C}^{0,\alpha}_{\beta-1;t}}\lesssim e(t)=t^\vartheta.
\end{align*}

Secondly, we need to show that $\mathcal{C}_{t}$ is a contraction. Hence, we bound 
\begin{align*}
    \left|\left|\mathcal{C}_{t}(\eta_1)-\mathcal{C}_{t}(\eta_2)\right|\right|_{\mathfrak{C}^{0,\alpha}_{\beta-1;t}}\lesssim (I)+(II)
\end{align*}
for 
\begin{align*}
    (I)=&\left|\left|D_{\Phi^{pre}_{t}}\pi_{-;\Phi^{pre}_{t}}\left\{Q_{\Phi_{t}^{pre}}(R_{\Phi^{pre}_{t}}\eta_1)-Q_{\Phi_{t}^{pre}}(R_{\Phi^{pre;t}}\eta_2)\right\}\right|\right|_{\mathfrak{C}^{0,\alpha}_{\beta-1;t}}\\
    (II)=&\left|\left|D^{op}_{\Phi_{t}^{pre}}\pi_{+;\Phi^{pre}_{t}}\left\{Q_{\Phi_{t}^{pre}}(R_{\Phi^{pre}_{t}}\eta_1)-Q_{\Phi_{t}^{pre}}(R_{\Phi^{pre}_{t}}\eta_2)\right\}\right|\right|_{\mathfrak{C}^{0,\alpha}_{\beta-1;t}}.
\end{align*}
These terms, will be bounded individually  by
\begin{align*}
        (I)=&\underbrace{\left|\left|\chi_2D_{\Phi^{pre}_{t}}\pi_{-;\Phi^{pre}_{t}}\left\{Q_{\Phi_{t}^{pre}}(R_{\Phi^{pre}_{t}}\eta_1)-Q_{\Phi_{t}^{pre}}(R_{\Phi^{pre;t}}\eta_2)\right\}\right|\right|_{C^{0,\alpha}_{\m{CFS};\beta-1;t}}}_{(i)}\\
        &+\underbrace{\left|\left|(1-\chi^t_4)D_{\Phi^{pre;t}}\pi_{-;\Phi^{pre;t}}\left\{Q_{\Phi^{pre;t}}(R_{\Phi^{pre;t}}\delta_t^*\eta_1)-Q_{\Phi^{pre;t}}(R_{\Phi^{pre;t}}\delta_t^*\eta_2)\right\}\right|\right|_{\mathfrak{C}^{0,\alpha}_{\m{ACF};\beta-1;t}}}_{(ii)}
\end{align*}

whereby we bound
\begin{align*}
    (i)\lesssim& \max\left\{\left|\left|\m{d}\Phi^{pre;t}_\zeta\right|\right|_{C^{0,\alpha}_{\m{CFS};\beta-1;t}}\left|\left|1\right|\right|_{C^{0,\alpha}_{\m{CFS};-2\beta;t}},\left|\left|1\right|\right|_{C^{0,\alpha}_{\m{CFS};-\beta;t}}\right\}\\
    &\cdot\left|\left|\chi_1R_{\Phi^{pre}_{t}}(\eta_1-\eta_2)\right|\right|_{C^{1,\alpha}_{\m{CFS};\beta;t}}\left(\left|\left|\chi_1R_{\Phi^{pre}_{t}}\eta_1\right|\right|_{C^{1,\alpha}_{\m{CFS};\beta;t}}+\left|\left|\chi_1R_{\Phi^{pre;t}}\eta_2\right|\right|_{C^{1,\alpha}_{\m{CFS};\beta;t}}\right)\\
    \lesssim&q(t)\cdot\left|\left|\eta_1-\eta_2\right|\right|_{\mathfrak{C}^{0,\alpha}_{\beta-1;t}}\left(\left|\left|\eta_1\right|\right|_{\mathfrak{C}^{0,\alpha}_{\beta-1;t}}+\left|\left|\eta_2\right|\right|_{\mathfrak{C}^{0,\alpha}_{\beta-1;t}}\right)
\end{align*}
\begin{align*}
    (ii)=&\left|\left|\pi_{\mathcal{I};\beta-1}(1-\chi^t_4)D_{\Phi^{pre;t}}\pi_{-;\Phi^{pre;t}}\left\{Q_{\Phi^{pre;t}}(R_{\Phi^{pre;t}}\delta_t^*\eta_1)-Q_{\Phi^{pre;t}}(R_{\Phi^{pre;t}}\delta_t^*\eta_2)\right\}\right|\right|_{C^{0,\alpha}_{\m{ACF};\beta-1;t}}\\
    &+t^{-\kappa}\left|\left|\varpi_{\mathcal{C}o\mathcal{K};\beta-1}(1-\chi^t_4)D_{\Phi^{pre;t}}\pi_{-;\Phi^{pre;t}}\left\{Q_{\Phi^{pre;t}}(R_{\Phi^{pre;t}}\delta_t^*\eta_1)-Q_{\Phi^{pre;t}}(R_{\Phi^{pre}_{t}}\delta_t^*\eta_2)\right\}\right|\right|_{C^{0,\alpha}_{t}}\\
    \lesssim&\max\{1,t^{\kappa-\beta-m/2-\alpha}\}\max\left\{\left|\left|\m{d}\Phi^{pre;t}_\zeta\right|\right|_{C^{0,\alpha}_{\m{ACF};\beta-1;t}}\left|\left|1\right|\right|_{C^{0,\alpha}_{\m{ACF};-2\beta;t}},\left|\left|1\right|\right|_{C^{0,\alpha}_{\m{ACF};-\beta;t}}\right\}\\
    &\cdot\left|\left|(1-\chi^t_5)R_{\Phi^{pre;t}}(\delta_t^*\eta_1-\delta_t^*\eta_2)\right|\right|_{\mathfrak{C}^{1,\alpha}_{\m{ACF};\beta;t}}\\
    &\cdot\left(\left|\left|(1-\chi^t_5)R_{\Phi^{pre;t}}\delta_t^*\eta_1\right|\right|_{\mathfrak{C}^{1,\alpha}_{\m{ACF};\beta;t}}+\left|\left|(1-\chi^t_5)R_{\Phi^{pre;t}}\delta_t^*\eta_2\right|\right|_{\mathfrak{C}^{1,\alpha}_{\m{ACF};\beta;t}}\right)\\
    \lesssim&q(t)\cdot\left|\left|\eta_1-\eta_2\right|\right|_{\mathfrak{C}^{0,\alpha}_{\beta-1;t}}\left(\left|\left|\eta_1\right|\right|_{\mathfrak{C}^{0,\alpha}_{\beta-1;t}}+\left|\left|\eta_2\right|\right|_{\mathfrak{C}^{0,\alpha}_{\beta-1;t}}\right)
\end{align*}

where we used Lemma \ref{pregluedmetricclosetoDiracmetric}, which ensures that $D^{pre}_t$ and $D_{\Phi^{pre}_t}$ are close and that 
\begin{align*}
    \left|\left|\varpi_{\mathcal{C}o\mathcal{K};\beta-1}(1-\chi^t_4)D^{pre;t}_\zeta\eta\right|\right|_{C^{0,\alpha}_t}\lesssim&\left|\left|\varpi_{\mathcal{C}o\mathcal{K};\beta-1}\widehat{D}^{t}_{\zeta;H}(1-\chi^t_4)\eta\right|\right|_{C^{0,\alpha}_t}\\
    &+\left|\left|\varpi_{\mathcal{C}o\mathcal{K};\beta-1}(D^{pre;t}_\zeta-\widehat{D}^{t}_{\zeta;H})(1-\chi^t_4)\eta\right|\right|_{C^{0,\alpha}_t}\\
    &+\left|\left|\varpi_{\mathcal{C}o\mathcal{K};\beta-1}\m{cl}_{g^{pre;t}_\zeta}(\m{d}\chi^t_4)\eta\right|\right|_{C^{0,\alpha}_t}\\
    \lesssim&t^{m/2-\beta-\alpha}\left|\left|(1-\chi^t_5)\eta\right|\right|_{C^{1,\alpha}_{\m{ACF};\beta;t}}
\end{align*}
The bounds for $(II)$ are analogous and thus we deduce that  

\begin{align*}
    (I)+(II)\lesssim&q(t)\cdot \left|\left|\eta_1-\eta_2\right|\right|_{\mathfrak{C}^{0,\alpha}_{\beta-1;t}}\left(\left|\left|\eta_1\right|\right|_{\mathfrak{C}^{0,\alpha}_{\beta-1;t}}+\left|\left|\eta_2\right|\right|_{\mathfrak{C}^{0,\alpha}_{\beta-1;t}}\right)
\end{align*}
where 
\begin{align*}
    q(t)=\min\left\{\lambda\beta,\kappa-m/2-\alpha\right\}
\end{align*}
Now, for $\eta\in\mathfrak{C}^{0,\alpha}_{\beta-1;t}$ satisfying 
\begin{align*}
    \left|\left|\eta\right|\right|_{\mathfrak{C}^{0,\alpha}_{\beta-1;t}}<\varrho(t)
\end{align*}
where $\varrho(t)$ satisfies
\begin{align*}
    t^{\vartheta}<\varrho(t)< q(t)^{-1}\und{1.0cm}q(t)\rho(t)^2<e(t)
\end{align*}
$\mathcal{C}_{t}$ defines a self-mapping and by the contraction mapping principal, we obtain a fixpoint $\eta_{t}$ of $\mathcal{C}_{t}$. Consequently, the torsion-free $\m{Spin}(7)$-structure $\Phi_{t}=\Theta(\Phi^{pre}_{t}+\eta_{t})$ on $X_{t}$ satisfies 

\begin{align}
    \left|\left|\Phi_{t}-\Phi^{pre}_{t}\right|\right|_{\mathfrak{D}^{1,\alpha}_{\beta;t}}\lesssim&\left|\left|D_{\Phi^{pre}_{t}}\left(\eta_{t}+Q_{\Phi_{t}^{pre}}(\eta_{t})\right)\right|\right|_{\mathfrak{C}^{0,\alpha}_{\beta-1;t}}\\\nonumber
    \lesssim&\left|\left|*_{\Phi^{pre}_{t}}\m{d}\pi_{+;\Phi^{pre}_{t}}Q_{\Phi_{t}^{pre}}(\eta_{t})\right|\right|_{\mathfrak{C}^{0,\alpha}_{\beta-1;t}}+\left|\left|\m{d}\Phi^{pre}_{t}\right|\right|_{\mathfrak{C}^{0,\alpha}_{\beta-1;t}}\\\nonumber
    \lesssim& q(t)\rho(t)^2+e(t)\lesssim t^\vartheta
\end{align}

Finally, we will prove the smoothness of $\Phi_{t}$. We will prove this by using an elliptic bootstrap to show that the solution $\eta_{t}$ is smooth. First notice that a solution to 
\begin{align*}
    -D_{\Phi_{t}^{pre}}\eta_{t}=&D^{op}_{\Phi_{t}^{pre}}\Phi^{pre}_{t}+D_{\Phi_{t}^{pre}}\pi_{-;\Phi^{pre}_{t}}\left\{Q_{\Phi_{t}^{pre}}(\eta_{t})\right\}\\
    &+D^{op}_{\Phi_{t}^{pre}}\pi_{+;\Phi^{pre}_{t}}\left\{Q_{\Phi_{t}^{pre}}(\eta_{t})\right\}
\end{align*}
can be seen as a solution to 
\begin{align*}
        \left[D_{\Phi_{t}^{pre}}\pi_{-;\Phi^{pre}_{t}}+\left\{D_{\Phi_{t}^{pre}}\pi_{-;\Phi^{pre}_{t}}+D^{op}_{\Phi_{t}^{pre}}\pi_{+;\Phi^{pre}_{t}}\right\}Q_{\Phi_{t}^{pre}}\circ\right]\eta_{t}=-D^{op}_{\Phi_{t}^{pre}}\Phi^{pre}_{t}
\end{align*}
where we can see $\left[D_{\Phi_{t}^{pre}}\pi_{-;\Phi^{pre}_{t}}+\left\{D_{\Phi_{t}^{pre}}\pi_{-;\Phi^{pre}_{t}}+D^{op}_{\Phi_{t}^{pre}}\pi_{+;\Phi^{pre}_{t}}\right\}Q_{\Phi_{t}^{pre}}\circ\right]$ as a nonlinear, first order differential operator that depends smoothly on the 1-jet of $\Phi^{pre}_{t}$. Further notice, that 
\begin{align*}
    \left\{D_{\Phi_{t}^{pre}}\pi_{-;\Phi^{pre}_{t}}+D^{op}_{\Phi_{t}^{pre}}\pi_{+;\Phi^{pre}_{t}}\right\}Q_{\Phi_{t}^{pre}}(0)=0
\end{align*}
and hence, we can deduce that 
\begin{align*}
\left[D_{\Phi_{t}^{pre}}\pi_{-;\Phi^{pre}_{t}}+\left\{D_{\Phi_{t}^{pre}}\pi_{-;\Phi^{pre}_{t}}+D^{op}_{\Phi_{t}^{pre}}\pi_{+;\Phi^{pre}_{t}}\right\}Q_{\Phi_{t}^{pre}}\circ\right]    
\end{align*}

is elliptic\footnote{Its linearisation is an elliptic first order differential operator.} for $C^0$-small enough $\eta$, as ellipticity is an open condition. For $\eta\in C^{0,\alpha}$ the right hand side of the above equation is also in $C^{0,\alpha}$, but as the right hand side is smooth, \cite[Thm. 10.7]{agmon1964estimates} ensures that $\eta$ lies in $C^{1,\alpha}$. Elliptic bootstrapping, yields the desired smoothness. 
\end{proof}

\begin{rem}
A concrete choice of parameters that satisfies all the assumptions made in the proof is
\begin{align*}
    \lambda \approx 0.65, \quad \alpha \ll1, \quad \beta \ll0,  \und{0.5cm}  \kappa = \frac{m}{2} + \beta - \frac{\alpha}{2}.
\end{align*}
With this choice, the exponent $\vartheta$ appearing in the a priori estimate satisfies 
\begin{empheq}[box=\colorbox{blue!10}]{align*}
    \vartheta \approx 2.6, 
\end{empheq}
ensuring that the torsion is sufficiently small to apply the contraction mapping argument and complete the gluing construction.
\end{rem}

\subsubsection{Deformation to Torsion-Free $G_2$-Structures}
\label{Deformation to Torsion-Free G2 Structures}

Let $(Y,\varphi)$ be a $G_2$ orbifold, whose singular strata of codimension $m$ is given by $P\subset Y$. In the following we will resolve the $\m{Spin}(7)$-orbifold 
\begin{align*}
    (X,\Phi)=(\mathbb{T}^1\times Y ,\m{d}\theta\wedge\varphi+*_{\varphi}\varphi)
\end{align*}
in an $\mathbb{T}^1$-equivariant way. Such a $\mathbb{T}^1$-equivariant torsion-free $\m{Spin}(7)$-structure $\Phi_{\zeta;t}$ on $X_{\zeta;t}=\mathbb{T}^1\times Y_{\zeta;t}$ yields a torsion-free $G_2$ on $Y_{\zeta;t}$. This construction thus generalises and extends the work of Joyce and Karigiannis \cite{joyce2017new}.\\

In order to apply Theorem \ref{mainthm}, we need to verify that the operator $D^{pre}_{\zeta;t}$ is isentropic, i.e. $\m{ob}_{\beta;t}=0$. By Proposition \ref{tameuniformboundedness} and Corollary \ref{tameuniformboundedness} it suffices to compute the total cohomology of $Y_{\zeta;t}$ and compare its dimension with the one of $\m{xker}_\beta(D^{pre}_{\zeta;t})$.\\

\begin{lem}
Let $P\subset Y$ be a singular stratum of codimension $m=4$. The cohomology of $Y_{\zeta;t}$ is given by 
\begin{align*}
    \m{H}^k(Y_{\zeta;t})=\m{H}^{k-2}(P,\mathcal{H}^2(\zeta^*\mathfrak{M}/P))\otimes \m{H}^k(Y)
\end{align*}
Furthermore, the cohomology of $Y_{\zeta;t}$ is given by 
\begin{align*}
    \m{H}^\bullet(Y_{\zeta;t})=\m{H}^\bullet(Y)\oplus\m{H}^\bullet(Q_\zeta)/\m{H}^\bullet(P).
\end{align*}
\end{lem}

By working in a $\mathbb{T}^1$-equivariant pregluing datum, we deduce the following corollary.

\begin{cor}[Existence of Torsion-Free $G_2$-Orbifold Resolutions]
\label{ExistenceofTFG2Resolutions}
Let $(Y_{\zeta;t},\varphi_{\zeta;t})$ be a resolution of the $G_2$-orbifold $(Y,\varphi)$, such that $\mathbb{T}^1\times Y_{\zeta;t},\m{d}s\wedge\varphi^{pre}_{\zeta;t}+*_{g_{\varphi^{pre}_{\zeta;t}}}\varphi^{pre}_{\zeta;t})$ is a $\mathbb{T}^1$-equivariant, isentropic, adiabatic $\m{Spin}(7)$-resolution of the orbifold $(\mathbb{T}^1\times Y ,\m{d}s\wedge\varphi+*_{g_\varphi}\varphi)$. Then there exists a torsion-free $G_2$-structure $\varphi_{\zeta;t}$ on $Y_{\zeta;t}$ satisfying 
\begin{align}
    \left|\left|\varphi_{\zeta;t}-\varphi^{pre}_{\zeta;t}\right|\right|_{X_{\zeta;t}}\lesssim t^{\vartheta}.
\end{align}
Again, $\vartheta$ satisfies \eqref{vartheta} and depends on the choices of $\alpha$, $\beta$, $\lambda$ and $\kappa$ as well as the pregluing error.
\end{cor}

\subsection{The Holonomy of the Resolved $\m{Spin}(7)$-Structure}
\label{The Holonomy of the Resolved Spin(7)-Structure}

In the following section we will discuss how the holonomy of the resolved metric can be determined. This will be done by computing the $\hat{A}$-genus of $X_{\zeta;t}$.\\

The $\hat{A}$-genus of the resolution can be expressed in terms of the Betti-numbers of $X$ and the $L^2$-Betti numbers of $N_\zeta$. In particular,
\begin{align}
    \hat{A}(X_{\zeta;t})=1+b^7_2(X_{\zeta;t})-b_1(X_{\zeta;t})
\end{align}
holds for $\m{Spin}(7)$-manifolds. By \cite[Thm. 10.6.1]{joyce2000compact} the holonomy of a compact $\m{Spin}(7)$-manifold is the full group $\m{Spin}(7)$, if the $\hat{A}$-genus is equal to 1. Moreover, in latter case both 
\begin{align*}
    b^7_2(X_{\zeta;t})=b_1(X_{\zeta;t})=0.
\end{align*}
vanish. As we have already seen in Section \ref{The Cohomology of Resolved Orbifolds} and under the Assumption \ref{ass8Spin7}, we are able to proof the following statement.

\begin{nota}
    We will denote by 
    \begin{align*}
\beta_i(N_\zeta)\coloneqq \sum_{p+q=i}\m{dim}(E^{p,q}_2(N_\zeta))-b_i(S).
    \end{align*}    
\end{nota}

\begin{prop}
Let us assume that Assumption \ref{ass8Spin7} holds. Then the $\hat{A}(X_{\zeta;t})$ is given by 
\begin{align*}
    \hat{A}(X_{\zeta;t})=1+b^7_2(X)-b_1(X)+\beta^7_2(N_\zeta)-\beta_1(N_\zeta).
\end{align*}
where $\beta^j_i$ are the refined harmonic Betti numbers of $N_\zeta$. Let $S\subset X$ be a singular stratum of type A of codimension four, such that $(X,\Phi)$ admits a resolution $(X_{\zeta;t},\Phi_{\zeta;t})$. Furthermore, in this case 
\begin{align*}
    \beta^7_2(N_\zeta)=\beta_1(N_\zeta)=0
\end{align*}
and thus, $\m{Hol}(g_{\Phi_{\zeta;t}})=\m{Spin}(7)$ if and only if $\m{Hol}(g_\Phi)=\m{Spin}(7)$.
\end{prop}

\begin{proof}
First notice, that the $\wedge^\bullet T^*N_\zeta$ decomposes in the presence of the Ehresmann connection on $N_\zeta$ into irreducible $\m{Sp}(1)\m{Sp}(1)\m{Sp}(1)\subset\m{Spin}(7)$-representations. As we have assumed that Assumption \ref{ass8Spin7} holds, we know that
\begin{align*}
    \m{H}^1(N_\zeta)/\m{H}^1(S)=\tot{}{1}{\m{H}^\bullet(S,\mathcal{H}^\bullet(N_\zeta/S))}=0
\end{align*}
and we deduce that
\begin{align*}
    \beta_1(N_\zeta)=0
\end{align*}
as $\mathcal{H}^0(N_\zeta/S)\cong\mathcal{H}^1(N_\zeta/S)=0$. The composition into irreducible $\m{Sp}(1)\m{Sp}(1)\m{Sp}(1)$-subbundles induces a decomposition of 
\begin{align*}
    \m{H}^\bullet(N_\zeta)/\m{H}^\bullet(S)
\end{align*}
on the second page of the Leray-Serre spectral sequence as $\Phi_{\zeta}$ satisfies $(\m{d}^{1,0}+\m{d}^{0,1})\Phi_\zeta=0$. In the case $m=4$ the stabilizer of the fibred $\m{Spin}(7)$ is $H\subset\m{Sp}(1)\m{Sp}(1)\m{Sp}(1)$ and
\begin{align*}
    \mathcal{H}^2_{7}(N_\zeta)\cong\m{H}^0(S,\mathcal{H}^2_+(N_\zeta/S))\oplus \m{H}^2_+(S,\mathcal{H}^0(N_\zeta/S))\oplus \m{H}^1_{\Phi_\zeta}(S,\mathcal{H}^1(N_\zeta/S))
\end{align*}
and since $\mathcal{H}^2(N_\zeta/S)=\mathcal{H}^2_-(N_\zeta/S)$
\begin{align*}
    \beta^7_2(N_\zeta)=0
\end{align*}
we conclude the proof.\footnote{Here $\m{H}^1_{\Phi}(S,\mathcal{H}^1(N_\zeta/S))$ is a subspace of $\m{H}^1(S,\mathcal{H}^1(N_\zeta/S))$ satisfying some algebraic condition. But as the latter vanishes, $\m{H}^1_{\Phi}(S,\mathcal{H}^1(N_\zeta/S))$ vanishes identically.}
\end{proof}

\section{Examples of Compact $\m{Spin}(7)$-Manifolds}
\label{Examples of Compact Spin(7)-Manifolds}

In this section, we will illustrate the application of the existence result established in the previous section by discussing concrete examples of compact $\m{Spin}(7)$-manifolds. The construction of such manifolds typically proceeds in two main steps:\\

\begin{itemize}
    \item[1.] Identify a suitable compact $\m{Spin}(7)$-orbifold.
    \item[2.] Verify Assumption \ref{ass8Spin7} for the constructed resolution.\\
\end{itemize}

We begin by addressing the first and most crucial step: How does one find compact $\m{Spin}(7)$-orbifolds that serve as candidates for our construction? Several classes of examples are known or can be systematically produced, which we briefly summarize below.

\begin{nota}
In the following we will denote by $\cat{CY}_m$ the category of $2m$ real dimensional Calabi--Yau manifolds, $\cat{HK}_k$ the category of $4k$ real dimensional Hyperkähler manifolds, $\cat{G}_2$ for the category of $G_2$-manifolds and $\cat{Spin(7)}$ for the category of $\m{Spin}(7)$-manifolds.\footnote{Morphisms are structure preserving smooth structure.}
\end{nota}

\begin{itemize}
    \item \textbf{Flat compact $\m{Spin}(7)$-orbifolds:} These take the form $\mathbb{O}/G$, where $G \subset \m{Spin}(7) \ltimes \mathbb{O}$ is a crystallographic group. Such orbifolds are classified in low dimensions and offer a tractable source of examples.
    \item \textbf{Compact $\cat{CY}_4$ with antiholomorphic involutions:} Given a compact Calabi--Yau 4-fold $(Y, \omega, \Omega)$ equipped with an antiholomorphic involution $\sigma$, one can often construct a $\m{Spin}(7)$-structure on the quotient $Y/\langle \sigma \rangle$, possibly after resolving singularities.
    \item \textbf{Lifted group actions on lower-dimensional spaces:} This includes actions on products such as $(\cat{HK}_1 \times \cat{HK}_1)/G$, $(\mathbb{T}^2\times\cat{CY}_3)/G$, or $\mathbb{T}^1\times\cat{G}_2$, where the symmetry group $G$ can be chosen to preserve an induced $\m{Spin}(7)$-structure.
\end{itemize}

\begin{nota}
In the following we will denote by $\cat{CY}_m$] the category of $2m$ real dimensional Calabi--Yau manifolds, $\cat{HK}_k$ the category of $4k$ real dimensional Hyperkähler manifolds, $\cat{G}_2$ for the category of $G_2$-manifolds and $\cat{Spin(7)}$ for the category of $\m{Spin}(7)$-manifolds.\footnote{Morphisms are structure preserving smooth structure.}
\end{nota}

\subsection{Generalised Kummer Construction}
\label{Generalsied Kummer Construction}

The first set of example of compact $\m{Spin}(7)$-manifolds which we will discuss are the ones first considered by Joyce\cite{Joyce1996b}. The idea of how to construct these examples is by generalising the Kummer construction of K3 surfaces to the $\m{Spin}(7)$-setting. In the following we will discuss these examples in a more general context and show that the estimates on the distance of the torsion-free to the pre-glued structure can be improved. This will extend the work of Platt \cite{platt} in the $G_2$-setting.\\

Let $(\mathbb{O},\Phi_{0},g_{0})$ be the standard flat $\m{Spin}(7)$-structure on $\mathbb{O}$ and let $\Gamma\subset\m{Iso}(\mathbb{O},\Phi_{0})$ be discrete and co-compact subgroup of $\m{Spin}(7)$-isometries. In particular, we have a lattice $\Lambda$ and a subgroup $H\subset\m{Spin}(7)$ such that

\begin{equation*}
    \begin{tikzcd}
        \Lambda\arrow[r,hook]&G\arrow[r, two heads]&H
    \end{tikzcd}
\end{equation*}
is short exact.

\begin{defi}
A resolution of $\mathbb{O}/G$ is called a \textbf{generalised Kummer}-$\m{Spin}(7)$ manifold or a Joyce manifold and is a compact torsion-free $\m{Spin}(7)$-manifold. 
\end{defi}

\begin{prop}
The singular strata of the orbifolds $\mathbb{O}/G$ can always be covered by intersecting affine tori of codimensions four, six, seven and eight. In particular, if the singular strata is of type A and the singular strata is of codimension four, the resolution data exists, if the flat bundle $\mathfrak{H}$ admits a parallel section avoiding the walls.
\end{prop}

\begin{proof}
    Notice, that the image of the fix point set of an element in $g=(w,A)\in H\ltimes \Lambda=G$ in $\mathbb{O}/G$ is given by an affine torus 
    \begin{align*}
        Fix(g)=\left\{x_0+\m{ker}(A-1)\right\}/\left\{\Lambda\cap\m{ker}(A-1)\right\}.
    \end{align*}
    Consequently, the singular locus is given by the intersection of such affine tori. As the ambient metric is flat, harmonic self-dual two forms are parallel and hence non-vanishing.
\end{proof}

\begin{rem}
Furthermore, if $S$ is of type A, the normal bundle $NS$ can be covered by the trivial bundle $\mathbb{R}^m\times \mathbb{T}^s$. As a consequence, Assumption \ref{ass8Spin7} is easily verifiable. 
\end{rem}

\subsubsection{Improved Estimates for Joyce-Manifolds}
\label{Improved Estimates for Joyce-Manifolds}

The first examples of compact $\m{Spin}(7)$-manifolds were constructed by Joyce in \cite{Joyce1996b} by resolving the orbifold singularities of $\mathbb{T}^8/\mathbb{Z}^4_2$ and the ones of certain weighted projective spaces. Using Theorem \ref{mainthm} we are able to deduce improved estimates for the torsion-free $\m{Spin}(7)$-structure constructed via the generalised Kummer construction. Notice, that $\Phi_{hot}=0$ and consequently, via are able to improve on $\vartheta$.

\begin{thm}
Let $G=\Gamma\ltimes\mathbb{Z}^8$ and $(X_{\zeta;t},\Phi_{\zeta;t})\dashrightarrow (\mathbb{R}^8/G,\Phi_{st})$ be one of the Joyce manifolds constructed in \cite{Joyce1996b} and $\Phi_{\zeta;t}$ its torsion-free $\m{Spin}(7)$-structure. Let $\Phi^{pre}_{\zeta;t}$ be the Cayley form on $X_{\zeta;t}$ constructed from the pregluing process. Then 
\begin{align*}
    \left|\left|\Phi_{\zeta;t}-\Phi^{pre}_{\zeta;t}\right|\right|_{\mathfrak{D}^{1,\alpha}_{\beta;t}}\lesssim t^{\approx4}.
\end{align*}
\end{thm}

\subsection{Resolution of $\cat{Spin(7)}$ with $\mathbb{Z}_2$-Involutions}
\label{Resolutions of Spin(7) Z2 Involutions}

In this subsection we will discuss how to construct a family of compact $\m{Spin}(7)$-manifolds, starting from a $\m{Spin}(7)$-manifold with involutions.\\

\begin{lem}\cite[Prop. 10.8.6]{joyce2000compact}
    Let $(\hat{X},\hat{\Phi})$ be a $\m{Spin}(7)$-manifold with an involution 
    \begin{align*}
        \tau\colon (\hat{X},\hat{\Phi})\rightarrow (\hat{X},\hat{\Phi}).
    \end{align*}
    The fixpoint-locus of $\tau$ consists of the disjoint union of points and a (non) connected Cayley submanifold. In particular, the quotient $(\hat{X},\hat{\Phi})/\mathbb{Z}_2$ is a $\m{Spin}(7)$-orbifold with codimension four and eight, isotropy $\mathbb{Z}_2$ strata. 
\end{lem}

\subsubsection{Resolution of $\cat{CY}_4$ with Antiholomorphic Involutions}
\label{Resolving CY4 with Antiholomorphic Involutions}

In the following subsection we will construct examples of $\m{Spin}(7)$-orbifolds using real homogenous polynomials.

\begin{defi}
    A \textbf{complete intersection} $Z$ of dimension $n$ is the intersection $H_1\cap ...\cap H_k\subset \mathbb{CP}^m$ of $k$-hypersurfaces that intersect transversely so that $n=m-k$.
\end{defi}
By the adjunction formula we know that the canonical bundle of $Z$ is trivial, if 
\begin{align*}
    d_1+...+d_k=m+1
\end{align*}
where $d_i$ denotes the degree of $H_i$. By Yau's proof of the Calabi-conjecture \cite{Yau1977Calabi} a complete intersection carries a Calabi--Yau metric. In a similar manner we can use complete intersections in weighted projective spaces to construct Calabi--Yau orbifolds.\\

The complex conjugation defines an anti-holomorphic involution of $\mathbb{CP}^m$. If all hypersurfaces are cut-out by equations in real coefficients, this antiholomorphic involution restricts to one on $Z$. An equivariant version of the existence result of the Calabi--Yau metric on $Z$ shows that the Calabi--Yau metric is $\mathbb{Z}_2$-invariant. In the following, we will denote this structure by $(Z,\omega,\theta)$.\\

If the real locus $\m{Re}(Z)$ is nonempty, it is a totally geodesic special Lagrangian with respect to the Calabi--Yau structure $(Z,\omega,\theta)$.

\begin{prop}
Let $(Z,\omega,\theta)$ be a $\cat{CY}_4$ given as a complete intersection. Assume that the real locus $\m{Re}(Z)$ of $Z$ is non empty and smooth. Let $\mathfrak{Z}_2\rightarrow \m{Re}(Z)$ be a twofold, unramnified cover of $\m{Re}(Z)$. If there exists a harmonic, self-dual, $\mathfrak{Z}_2$-twisted, nonvanishing two form $\zeta$ on $\m{Re}(Z)$, then the quotient $Z/\mathbb{Z}_2$ admits a resolution 
\begin{align*}
    \rho_{\zeta;t}\colon X_{\zeta;t}\dashrightarrow Z/\mathbb{Z}_2.
\end{align*}
Moreover, the resolution carries a family of torsion-free $\m{Spin}(7)$-structure $\Phi_{\zeta;t}$ resolving the natural one on $(Z/\mathbb{Z}_2,\Phi)$ induced by the Calabi--Yau structure.
\end{prop}

\begin{rem}
It remains an open question if Assumption \ref{ass8Spin7} holds in general.
\end{rem}

    \begin{rem}
        Let $(\omega_{FS},\theta_Z)$ denote the $\m{SU}(4)$-structure induced by the Fubini-Study metric on $Z$. Assume that $(\omega,\theta_Z)$ is a Calabi--Yau structure close to $(\omega_{FS},\theta_Z)$. If $|\zeta_{FS}|>c$ for sufficiently large $c$, then $\zeta$ might be non-vanishing.
    \end{rem}

\begin{ex}
    Let $Z$ denote the complete intersection given by 
    \begin{align*}
        \{[z_0:...:z_5]\in\mathbb{CP}^5|z_0^6+z_1^6+z_2^6-z_3^6-z_5^6-z_6^6=0\}.
    \end{align*}
    The real locus $\m{Re}(Z)$ is given by 
        \begin{align*}
        \{[x_0:...:x_5]\in\mathbb{RP}^5|x_0^6+x_1^6+x_2^6-x_3^6-x_5^6-x_6^6=0\}.
    \end{align*}
    which is diffeomorphic to $(\mathbb{S}^2\times \mathbb{S}^2)/\mathbb{Z}_2$. In particular, $\m{H}^2(\m{Re}(Z),\mathbb{R})=0$ but
    \begin{align*}
        \m{H}^2(\m{Re}(Z),\widetilde{\m{Re}(Z)}\times_{\mathbb{Z}_2}\mathbb{R})=\m{H}^2(\widetilde{\m{Re}(Z)},\mathbb{R})=\mathbb{R}^2.
    \end{align*}
    In particular, 
    \begin{align*}
        b_2^+(\widetilde{\m{Re}(Z)})=\frac{1}{2}(b_2(\widetilde{\m{Re}(Z)})-\sigma(\widetilde{\m{Re}(Z))})=1.
    \end{align*}
    The harmonic representative of $\m{H}^2_+(\widetilde{\m{Re}(Z)},\mathbb{R})$ with respect to the metric on $\m{Re}(Z)$, induced by the Fubini-Study metric, the sum of two volume-forms hence, non-vanishing. 
\end{ex}

\subsubsection{Conjectural Examples in $\cat{SYZ}_4$}
\label{Conjectural Examples in SYZ4}

The SYZ-conjecture is a conjectural relation between mirror symmetry and T-duaily in String Theory. Formulated by Ströminger, Yau and Zaslow in the category $\cat{CY}_3$ the conjecture extends to special holonomy spaces.

\begin{con}\cite{SYZ}
    Let $(X,\Phi)\in \cat{Spin(7)}$. There exists a compact Riemannian four manifold $S$ and a singular fibration 
    \begin{align*}
        \pi\colon X\rightarrow S
    \end{align*}
    by toric Cayleys. Moreover, the mirror-dual $\m{Spin}(7)$-manifold is given by 
    \begin{align*}
        \overline{X}\coloneqq\m{Hom}(\pi_1(\pi^{-1}),\m{U}(1)).
    \end{align*}
\end{con}

\begin{rem}
    In the subcategory $\cat{CY}_4$ the fibration is a fibration by special Lagrangian tori. 
\end{rem}

Let $(\hat{X},\hat{\Phi})$ be a SYZ-fibred $\m{Spin}(7)$-manifold. We define an involution 
    \begin{align*}
        \tau\colon (\hat{X},\hat{\Phi})\rightarrow (\hat{X},\hat{\Phi})
    \end{align*}
by fibrewise multiplication by $-1$. The singular strata of $\m{Spin}(7)$-orbifold $(X,\Phi)/\tau$ are sixteen copies of $S$ with isotropy $\mathbb{Z}_2$.

\begin{con}
Let us assume that there exists non-vanishing, harmonic $\zeta\in \Omega^2_+(S)$. Assume that we can resolve $(X,\Phi)/\tau$ by a family 
    \begin{align*}
        \rho_{\zeta;t}\colon (X_{\zeta;t},\Phi_{\zeta;t})\dashrightarrow (X,\Phi)/\tau.
    \end{align*}
Then $(\overline{X},\overline{\Phi})/\overline{\tau}$ can be resolved by a family 
\begin{align*}
    \rho_{\zeta^*;t}\colon (\overline{X}_{\zeta^*;t},\overline{\Phi}_{\zeta^*;t})\dashrightarrow (\overline{X},\overline{\Phi})/\overline{\tau}
\end{align*}
and both $(X_{\zeta;t},\Phi_{\zeta;t})$ and $(\overline{X}_{\zeta^*;t},\overline{\Phi}_{\zeta^*;t})$ are mirror dual.
\end{con}

\subsection{Resolutions of $(\cat{HK}_1\times \cat{HK}_1)/G$}
\label{Resolutions of HK1timesHK1Gamma}

In the following subsection, we will focus on a select set of examples of compact $\m{Spin}(7)$-manifolds. These will be constructed using finite group actions on products of compact hyperkähler four-manifolds $\cat{HK}_1$.\\

Given two hyperkähler four-manifolds $(M_i,\underline{\omega}_i,g_i) \in \cat{HK}_1$, we define the natural Cayley form on their product as
\begin{align*}
    \hat{\Phi} = \m{vol}_1 - \m{tr}(\underline{\omega}_1 \wedge \underline{\omega}_2) + \m{vol}_2,
\end{align*}
where $\underline{\omega}_i$ denotes the hyperkähler triple on $M_i$, and $\m{vol}_i$ represents the associated volume forms.

Group actions by isometries on $\cat{HK}_1 \times \cat{HK}_1$ can be divided into two types:
\begin{itemize}
    \item \textbf{Product-type actions:} finite groups acting diagonally on $(M_1, g_1) \times (M_2, g_2)$.
    \item \textbf{Swapping-type actions:} involutions swapping $M_1$ and $M_2$ via an isometry $M_1 \cong M_2$.
\end{itemize}

For the first type of group actions, a finite group $G$ that preserves the Cayley form $\hat{\Phi}$ must act by orientation-preserving isometries and induce a $G$-equivariant isomorphism between the maximal positive-definite subspaces
\begin{align*}
    P_1 \subset \m{H}^2(M_1, \mathbb{R}) \und{1.0cm} P_2 \subset \m{H}^2(M_2, \mathbb{R}),
\end{align*}
with the identification $\m{H}^2(M_2, \mathbb{R})^* \cong \m{H}^2(M_2, \mathbb{R})$ from Poincaré duality. We will discuss the details of these conditions further in Subsection \ref{Actions on HK1}.\\

For the second type of group actions, the involution swapping $M_1$ and $M_2$ preserves the Cayley form $\hat{\Phi}$ if and only if it is a hyperkähler involution. In this case, the orbifold quotient is given by
\begin{align*}
    (M_1 \times M_1) / \mathbb{Z}_2 = \m{Sym}^2(M_1),
\end{align*}
the symmetric product of $M_1$. We will explain how resolutions of these symmetric products are related to the Hilbert schemes of two points. The $\m{Spin}(7)$-orbifold resolutions obtained this way will carry hyperkähler structures. Using the results from Section \ref{Existence of torsion-free Spin(7) Structures on Orbifold Resolutions}, the pre-glued hyperkähler metrics will approximate those constructed via Yau’s solution to the Calabi conjecture. These constructions and their properties will be further explored in Subsection \ref{The Hilbert Scheme Hilb^2(HK1)}.\\

\subsubsection{Actions on $\cat{HK}_1$}
\label{Actions on HK1}

A vast class of examples of compact $\m{Spin}(7)$-orbi-folds are constructed by quotients of products of hyperkähler lines, i.e. $\mathbb{H}/\Lambda\times \mathbb{H}/\Lambda'$, $\mathbb{H}/\Lambda\times K3 $ and $K3\times K3'$. The first example are flat $\m{Spin}(7)$-orbifolds (Bieberbach orbifolds) and their resolutions have been extensively studied by Joyce in \cite{joyce2000compact}. From now on we will assume that $M_i\in\cat{HK}_1$ and that at most one of them is flat.\\

Given a $M\in\cat{HK}_1$ the second cohomology $\m{H}^2(M,\mathbb{R})$ is a real vector space. The intersection form defines a definite bilinear form of signature 
\begin{align*}
    (3,3)\oder{1.0cm} (3,19)
\end{align*}
In the following we will denote by $P\subset \m{H}^2(M,\mathbb{R})$ the maximal positive subspace, corresponding to the span of the hyperkähler structure. The $\m{H}^2(M,\mathbb{R})$ contains a canonical lattice given by 
\begin{align*}
    \m{H}^2(M,\mathbb{Z})\subset\m{H}^2(M,\mathbb{R}).
\end{align*}
Let $G\subset\m{Isom}(M,g)$ denote a finite group acting via isometries. This action induces a representation 
\begin{align*}
    \varrho_G\colon G\rightarrow \m{Aut}(\m{H}^2(M,\mathbb{Z}))\subset\m{Aut}(\m{H}^2(M,\mathbb{R})).
\end{align*}
A particularly fruitful approach to constructing group actions on a K3 surface $M \in \cat{HK}_1$ involves in the first step studying the induced action on its second integral cohomology group $ \m{H}^2(M, \mathbb{Z}) $, which carries the structure of an even, unimodular lattice of signature $ (3,19) $, known as the K3 lattice. Automorphisms of $ M $ that preserve the hyperkähler structure induce lattice automorphisms preserving the intersection form, and thus correspond to elements in $ \m{O}(\m{H}^2(M, \mathbb{Z})) $.\\

The construction of such actions can be approached lattice-theoretically. Given a finite group $ G $, one seeks a faithful representation
\begin{align*}
    \varrho_G\colon  G \to \m{O}(\Lambda_{\m{K3}})
\end{align*}

that satisfies certain geometric constraints. In particular, to lift such a representation to an actual geometric action on a K3 surface, it is necessary that $ \varrho_G $ preserves a positive-definite 3-plane $ P \subset \m{H}^2(M, \mathbb{R}) $, corresponding to the span of the hyperkähler structure $ [\underline{\omega}] $. By the global Torelli theorem \cite[Thm 2.4]{huybrechts2016lectures} for K3 surfaces, such a lattice automorphism lifts to an actual isometry of the surface if and only if it preserves this period plane and satisfies appropriate monodromy conditions.\\

Following the work of Nikulin \cite{nikulin1979finite,van2007nikulin}, this can be formalized by analysing primitive embeddings of the invariant sublattice $ \Lambda^G \subset \Lambda_{\m{K3}} $, ensuring that its orthogonal complement $ (\Lambda^G)^\perp $ satisfies suitable lattice-theoretic conditions (e.g., absence of roots or obstructions to ampleness). Such data determines whether the action lifts to a holomorphic, anti-holomorphic, or non-symplectic automorphism.\\

This method is used to construct suitable actions on $ M_1 \times M_2 $ from diagonal actions $ G \subset \m{Isom}(M_1)\times\m{Isom}(M_2) $ by identifying the positive planes $ P_1 $ and $ P_2 $ via a $ G $-equivariant isometry
\begin{align*}
    \hat{\Phi}^{2,2}\colon  P_1 \xrightarrow{\cong} P_2,
\end{align*}
which ensures that the Cayley form $ \hat{\Phi} $ descends to the quotient $ (M_1\times M_2)/G $, yielding a torsion-free $ \m{Spin}(7) $-structure.

\begin{defi}
    An action of a group $G$ is holomorphic if $P_i\cong\mathbb{R}\oplus\mathbb{C}$ as a $G$-space, and the action is given by $\m{id}\oplus\chi$ for a character $\chi\colon G\rightarrow \mathbb{C}^\times$.
\end{defi}

\begin{cor}
Let $G$ act holomorphically on $(M_1,\underline{\omega}_1)$ and $(M_2,\underline{\omega}_2)$. Then $G$ preserves $\hat{\Phi}^{2,2}$ if and only if $\chi_1=\overline{\chi}_2$.
\end{cor}

\begin{rem}
    Using the Nikulin-classification \cite{nikulin1979finite,van2007nikulin,huybrechts2016lectures} of (anti-)holo-morphic and (anti-symplectic) $\mathbb{Z}_2$-actions, it is possible to construct a vast number of new examples of compact $\m{Spin}(7)$-manifolds. Using these actions by $\mathbb{Z}_2^d$ we are able to construct $\m{Spin}(7)$-orbifolds with singular strata being diffeomorphic to the $K3$ or products of two tori and Riemannian surfaces. These orbifolds can always be resolved.
\end{rem}

\begin{rem}
    If the singular stratum $S=S_1\times S_2$ is smooth, of codimension four, isotropy type $\mathbb{Z}_2$ and of type $(2,2)$ in $\cat{HK}_1\times \cat{HK}_1$, then there exists canonical, non-vanishing self-dual two form on $S$ given by 
    \begin{align*}
        \zeta=\m{vol}_{S_1}+\m{vol}_{S_2}.
    \end{align*}
\end{rem}

\subsubsection{Examples of Resolutions of $(HK1\times HK1)/G$}
\label{Examples of Resolutions of (HK1xHK1)/G}

In the following we will construct a new example of a $\m{Spin}(7)$-manifold with holonomy equal $\m{Spin}(7)$.

\begin{ex}
We will follow the idea in \cite[Chap. 7.3]{joyce2017new} and pick a K3 surface $\Sigma$ that admits a branched double cover
\begin{align*}
    \Sigma\rightarrow\mathbb{CP}^2
\end{align*}
whose branching set is given by a sixtic $C$ of genus $10$. We define an involution
\begin{align*}
    \alpha_{\Sigma}\colon \Sigma\rightarrow \Sigma
\end{align*}
by swapping the sheets of the branched covering and consequently $\m{Fix}(\alpha_{\Sigma})=C$.\\ 

Furthermore, the $\mathbb{Z}_2$ action on $\mathbb{CP}^2$ by complex conjugation can be lifted to an action 
\begin{equation*}
    \begin{tikzcd}
            \Sigma\arrow[r,"\beta_{\Sigma}"]\arrow[d]&\Sigma\arrow[d]\\
            \mathbb{CP}^2\arrow[r,"\overline{.}"]&\mathbb{CP}^2
    \end{tikzcd}
\end{equation*}
for which $\m{Fix}(\beta_{\Sigma})=\mathbb{S}^2$ and $\m{Fix}(\alpha_{\Sigma}\beta_{\Sigma})=\emptyset$. Now, set 
\begin{align*}
    \mathbb{T}^1=[0,4]/\sim
\end{align*}
an define the $\mathbb{Z}^3_2$-action
\begin{align*}
    \alpha_{\mathbb{T}^4}&\colon (x_1,x_2,x_3,x_4)\mapsto(-x_1,2-x_2,x_3,x_4)\\
    \beta_{\mathbb{T}^4}&\colon (x_1,x_2,x_3,x_4)\mapsto(2-x_1,x_2,-x_3,x_4)\\
    \gamma_{\mathbb{T}^4}&\colon (x_1,x_2,x_3,x_4)\mapsto(-x_1,-x_2,2-x_3,2-x_4)\\
\end{align*}
on the four torus $\mathbb{T}^4$, with 
\begin{align*}
    \m{Fix}(\alpha_{\mathbb{T}^4})=&\left\{{0\choose 2},{1\choose 3},\mathbb{T}^1,\mathbb{T}^1\right\}\\
    \m{Fix}(\beta_{\mathbb{T}^4})=&\left\{{1\choose 3},\mathbb{T}^1,{0\choose 2},\mathbb{T}^1\right\}\\
    \m{Fix}(\gamma_{\mathbb{T}^4})=&\left\{{0\choose 2},{0\choose 2},{1\choose 3},{1\choose 3}\right\}.
\end{align*}
We can lift the action of $\mathbb{Z}_2^3$ to $\mathbb{T}^4\times \Sigma$ via $\alpha= \alpha_{\mathbb{T}^4}\times \alpha_{\Sigma}$, $\beta=\beta_{\mathbb{T}^4}\times \beta_{\Sigma}$ and $\gamma=\gamma_{\mathbb{T}^4}\times \m{id}_{\Sigma}$ whose fixpoint sets are given by 
\begin{align*}
    \m{Fix}(\alpha)=&\left\{{0\choose 2},{1\choose 3},\mathbb{T}^1,\mathbb{T}^1\right\}\times C\\
    \m{Fix}(\beta)=&\left\{{1\choose 3},\mathbb{T}^1,{0\choose 2},\mathbb{T}^1\right\}\times \mathbb{S}^2\\
    \m{Fix}(\gamma)=&\left\{{0\choose 2},{1\choose 3},{0\choose 2},{1\choose 3}\right\}\times \Sigma.
\end{align*}
and the words 
\begin{align*}
    \alpha\beta,\alpha\gamma,\beta\gamma,\alpha\beta\gamma
\end{align*}
act freely on  $\mathbb{T}^4\times \Sigma$.\\

Clearly, this $\mathbb{Z}^3_2$ action preserves the torsion-free $\m{Spin}(7)$-structure 
\begin{align*}
    \hat{\Phi}=\m{vol}_{\mathbb{T}^4}+\m{vol}_{\Sigma}-\m{tr}\left(\underline{\omega}_{\mathbb{T}^4}\wedge\underline{\omega}_{\Sigma}\right)
\end{align*}
on $\mathbb{T}^4\times \Sigma$ and hence, the torsion-free structure $\hat{\Phi}$ descends to a torsion-free $\m{Spin}(7)$-structure $\Phi$ on the orbifold
\begin{align*}
    (\mathbb{T}^4\times \Sigma)/\mathbb{Z}^3_2.
\end{align*}

The singular strata in $(\mathbb{T}^4\times \Sigma)/\mathbb{Z}^3_2$ are nonintersecting and given by two copies of $\mathbb{T}^2\times C$, two copies of $\mathbb{T}^2\times \mathbb{S}^2$ and four copies of $\Sigma$.\\

We can pick the resolution data on $\mathbb{T}^2\times C$ to be determined by $\zeta=\m{vol}_{\mathbb{T}^2}+\m{vol}_{C}$, the resolution data on $\mathbb{T}^2\times \mathbb{S}^2$ by $\zeta=\m{vol}_{\mathbb{T}^2}+\m{vol}_{\mathbb{S}^2}$ and on the $\Sigma$ we can freely pick any combination of the hyperkähler triple as $\zeta$.\\

In order to verify Assumption \ref{ass8Spin7} we compute the Lerray-Serre spectral sequence of $N_\zeta\rightarrow S$. The spectral sequence collapse on the second page, i.e. the maps 
\begin{align*}
    \m{d}^{3,-2}\colon&\m{H}^0(S,\mathbb{R})\rightarrow \m{H}^3(S)\\
    \m{d}^{3,-2}\colon&\m{H}^1(S,\mathbb{R})\rightarrow \m{H}^4(S)     
\end{align*}
as $\Sigma$ is simply connected, and via a degree argument on the product strata $\mathbb{T}^2\times \mathbb{S}^2$ and $\mathbb{T}^2\times C$.\footnote{In the latter case we use that $N_\zeta\rightarrow S\rightarrow \mathbb{T}^2$ is a sequence of fibre bundles and hence $\m{H}^\bullet(S,\m{H}^\bullet(N_\zeta/S))\cong \m{H}^\bullet(\mathbb{T}^2,\m{H}^\bullet(S/\mathbb{T}^2,\m{H}^\bullet(N_\zeta/S)))\cong\m{H}^\bullet(\mathbb{T}^2,\m{H}^\bullet(N_\zeta/\mathbb{T}^2))\cong\m{H}^\bullet(N_\zeta)$.}
\end{ex}

\begin{lem}
The Assumption \ref{ass8Spin7} holds on $X_{\zeta;t}$ and we thus obtain a $\m{Spin}(7)$-structure $\Phi_{\zeta;t}$ satisfying 
\begin{align*}
    \left|\left|\Phi_{\zeta;t}-\Phi^{pre}_{\zeta;t}\right|\right|_{\mathfrak{D}^{1,\alpha}_{\beta;t}}\lesssim t^{\vartheta}.
\end{align*}
Moreover, as $b_1(X_{\zeta;t})=b^7_2(X_{\zeta;t})=0$ the resolved $\m{Spin}(7)$-manifold has full holonomy.
\end{lem}

\begin{ex}
Let $i\in\{1,2\}$ and in the following $g_i(z_1,z_2,z_3)\in \mathbb{R}[z_1,z_2,z_3]_4$ be two homogeneous degree four polynomial with no real solution, e.g. the $g_i(z_1,z_2,z_3)=z_1^4+z_2^4+z_3^4$. Let 
\begin{align*}
    f(z_0,z_1,z_2,z_3)=z_0^4+g_i(z_1,z_2,z_3).
\end{align*}
The hypersurface $M_i$ corresponding to the vanishing locus of $f$ in $\mathbb{CP}^3$ is a K3 surface. Define an action of $D_4=\left<\tau,g|\tau^2=g^4=1,\tau g\tau g=1\right>$ on $M$ by 
\begin{align*}
    g_i\colon [z_0:z_1:z_2:z_3]\mapsto [\m{i}z_0:z_1:z_2:z_3]
\end{align*}
and 
\begin{align*}
    \tau\colon [z_0:...:z_3]\mapsto[\overline{z}_0 \colon\overline{z}_1 \colon\overline{z}_2 \colon\overline{z}_3].
\end{align*}
Further, we have
\begin{align*}
    Fix(g_i)=\left\{[z_1:z_2:z_3]\in \mathbb{CP}^2|g_i(z_1,z_2,z_3)=0\right\}
\end{align*}
and $Fix(\tau)=\emptyset$. Moreover, $g_i^*\omega_i=\omega_i$ and $g_i^*\theta_i=\m{i}\theta_i$, and $\tau^*\omega_i=-\omega_i$ and $\tau^*\theta_i=\overline{\theta}_i$.\\

The $\m{Spin}(7)$-structure 
    \begin{align*}
        \hat{\Phi}=\m{vol}_1-\omega_1\wedge\omega_2-\m{Re}(\theta_1\wedge\overline{\theta_2})+\m{vol}_2
    \end{align*}
is $D_4$-invariant, and hence, the $(X,\Phi)\coloneqq(M_1\times M_2,\hat{\Phi})/D_4$ is a $\m{Spin}(7)$-orbifold, with singular stratum of codimension four and isotropy $\mathbb{Z}_4$, given by 
\begin{align*}
    S\coloneqq &\left\{([z_1:...:z_2],[w_1:...:w_2])\in\mathbb{CP}^2\times \mathbb{CP}^2|\right.\\
    &\left.g_1(z_1,z_2,z_3)=g_2(w_1,w_2,w_3)=0\right\}/\mathbb{Z}_2.
\end{align*}
Notice, that the bundle $\m{Fr}(X/S,\Phi)$ reduces further to the strata being a two-fold cover of a product. In particular, the frame bundle of $S$ reduces to $\m{Fr}(S)=(\m{Fr}(S_1)\times \m{Fr}(S_2))/\mathbb{Z}_2$, which is a $\m{U}(1)^2$-bundle on $S$. We twisted the parameter bundle by $\mathbb{Z}_2\subset Weyl(A_3)$ and pick 
\begin{align*}
    \zeta\coloneqq \m{vol}_{S_1}+\m{vol}_{S_2}\in \Omega^0(S,\mathfrak{P}). 
\end{align*}
\end{ex}
\begin{lem}
    There exists family $\rho_{\zeta;t}\colon (X_{\zeta;t},\Phi_{\zeta;t})\dashrightarrow (X,\Phi)$ torsion-free $\m{Spin}(7)$-orbifold resolution such that 
\begin{align*}
    \left|\left|\Phi_{\zeta;t}-\Phi^{pre}_{\zeta;t}\right|\right|_{\mathfrak{D}^{1,\alpha}_{\beta;t}}\lesssim t^{\vartheta}.
\end{align*}
\end{lem}

\subsubsection{The Hilbert Scheme $\m{Hilb}^2(\cat{HK}_1)$}
\label{The Hilbert Scheme Hilb^2(HK1)}

Using Yau's proof of the Calabi-conjecture, it was shown that the Hilbert scheme of two points $\m{Hilb}^2(\cat{HK}_1)$ of a $\cat{HK}_1$, i.e. a crepant resolution of $\m{Sym}^2(\cat{HK}_1)$ carries a hyperkähler metric.\\

Using Theorem \ref{mainthm} we will construct a family of torsion-free $\m{Spin}(7)$-struc-tures on $\m{Hilb}^2(\cat{HK}_1)$ whose holonomy reduces to a subgroup of $\m{Sp}(2)$. As such, we prove that the hyperkähler metric on $\m{Hilb}^2(\cat{HK}_1)$ is \say{close} to the pre-glued metric.\\

Let $(\Sigma,\underline{\omega}_\Sigma)\in \cat{HK}_1$ denote a hyperkähler structure. The second symmetric product $\m{Sym}^2(\cat{HK}_1)=(\cat{HK}_1\times \cat{HK}_1)/\mathbb{Z}_2$ is a $\m{Spin}(7)$-orbifold whose torsion-free $\m{Spin}(7)$-structure is induced by the product structure 
\begin{align*}
    \hat{\Phi}=\m{vol}_{1}+\m{vol}_2-\m{tr}(\underline{\omega}_1\wedge\underline{\omega}_2).
\end{align*}
The singular strata of $\m{Sym}^2(\Sigma)$ is the diagonal $\Sigma$ and its normal bundle can be identified with its tangent bundle, i.e.
\begin{equation*}
\begin{tikzcd}
        T\Sigma\arrow[r,hook]&T\m{Sym}^2(\Sigma)\arrow[r,two heads]&N\Sigma
\end{tikzcd}
\end{equation*}
splits by $N\Sigma\cong T\Sigma\ni v\mapsto (\pm v\oplus\mp v)$. There exists a two sphere of nonvanishing, harmonic, self-dual two forms corresponding to resolutions 
\begin{align*}
    N_{\omega_I}\dashrightarrow N\Sigma/\mathbb{Z}_2.
\end{align*}
The spaces $N_{\omega_I}$ are biholomorphic to $\m{Bl}_I(0\colon \Sigma\hookrightarrow N\Sigma/\mathbb{Z}_2)$ and the resolved space is biholomorphic to 
\begin{align*}
    \m{Bl}_I(\Delta\colon \Sigma\hookrightarrow \m{Sym}^2(\Sigma))\cong\m{Hilb}^2(\Sigma).
\end{align*}

The pre-gluing metric $g^{pre}_{I;t}$ interpolates between the metric on the symmetric product and the metric on the ACF space $N_{\omega_I}$.

\begin{prop}\cite[Cor. 3.9]{joyce2017new}
There exists a Morsian tubular neighbourhood of the diagonal $\Sigma$ in $\m{Sym}^2(\Sigma)$, i.e.
\begin{align*}
    \m{exp}_{g_\Phi}^*\Phi=\Phi+\Phi_{hot}\und{1.0cm}|\Phi_{hot}|=\mathcal{O}(r^2)
\end{align*}
\end{prop}

\begin{proof}
A geodesic neighbourhood of the diagonal in $K3\times K3$ is invariant under the $\mathbb{Z}_2$-action. By Taylor expanding the Cayley form and using the arguments \cite[Chap. 3.4]{joyce2017new} we can eliminate the linear order term in the expansion.
\end{proof}

\begin{prop}
There exists an  $B^3_T(0)\backslash\{0\}\subset\m{Im}(\mathbb{H})$-family of hyperkähler metrics $g_{I;t}$ on $\m{Hilb}^2(\Sigma)$ satisfying 
\begin{align*}
    \left|\left|g_{I;t}-g^{pre}_{I;t}\right|\right|_{\mathfrak{D}^{1,\alpha}_{\beta;t}}\lesssim t^{\vartheta}.
\end{align*}
\end{prop}

\subsection{Resolutions of $(\cat{CY}_3\times\mathbb{T}^2)/G$}
\label{Resolutions of CY3T2}

Another source of examples of $\m{Spin}(7)$-orbifolds are quotients of a $(Z,\omega,\theta)\in\cat{CY}_3$ and a complex torus. We define the $\m{Spin}(7)$-manifold 
\begin{align*}
    (\mathbb{T}^2\times Z, \m{d}z\wedge\m{d}\overline{z}\wedge\omega +\m{Re}(\m{d}z\wedge\theta)).
\end{align*}
We can construct a $\m{Spin}(7)$-orbifold by matching actions on $Z$ by an action on $\mathbb{T}^2$ to construct $\m{Spin}(7)$-orbifolds. These orbifolds will usually not be simply connected and hence, are not suitable to construct new examples of full holonomy $\m{Spin}(7)$-manifolds.\\

Examples with isotropy $\mathbb{Z}_2$ and codimension four singular strata can be constructed analogously to \cite[Chap. 7.4]{joyce2017new}.

\subsection{Resolutions of $(\cat{G}_2\times\mathbb{T}^1)/\mathbb{Z}_2$}
\label{Resolutions of G2T1}

In the following section we will discuss how to construct compact $\m{Spin}(7)$-manifolds using the extra twisted connected sum of $G_2$-manifolds \cite{nordstrom2023extra}. This example produces a family of $\m{Spin}(7)$-manifolds with a neck stretching phenomena. Using methods from geometric measure theory, we are able to construct ACyl $\m{Spin}(7)$-manifolds. The author is very grateful for the help of Johannes Nordström in the construction of this example. 

\begin{lem}\cite[Prop. 10.8.1]{joyce2000compact}
    Let $(Y,\varphi)$ be $G_2$-manifold with an involution 
    \begin{align*}
        \tau\colon (Y,\varphi)\mapsto (Y,-\varphi).
    \end{align*}
    If the fixpoint-locus is non-empty, then $Fix(\tau)\subset Y$ is a coassociative. 
\end{lem}

Given such a tuple $(Y,\varphi,\tau)$, we construct the $\m{Spin}(7)$-orbifold 
\begin{align*}
    (X,\Phi)\coloneqq(\mathbb{S}^1\times Y,\theta\wedge\varphi+*_{\varphi}\varphi)/\tau
\end{align*}
where $\tau$ acts on $\mathbb{S}^1\mapsto \mathbb{S}^1,\theta\mapsto \overline{\theta}$. 

\begin{lem}
    Assume that the coassociative $\m{Fix}(\tau)$ admits a non-vanishing, harmonic, self-dual, $\mathbb{Z}_2$-twisted two form 
    \begin{align*}
        \zeta\in \m{H}^2_+(Fix(\tau),\mathfrak{C}\otimes_{\mathbb{Z}_2}\mathbb{R}),
    \end{align*}
    then there exists a family of torsion-free $\m{Spin}(7)$-orbifold resolutions
    \begin{align*}
        \rho_{\zeta;t}\colon (X_{\zeta;t},\Phi_{\zeta;t})\dashrightarrow (X,\Phi),
    \end{align*}
    such that 
    \begin{align*}
        \left|\left|\Phi_{t;\zeta}-\Phi^{pre}_{\zeta;t}\right|\right|_{\mathfrak{D}^{1,\alpha}_{\beta;t}}\lesssim t^\vartheta.
    \end{align*}
\end{lem}

A large class of existing examples of $\cat{G}_2$ manifolds are constructed by gluing two $\cat{ACylG}_2$. Often this construction is described by the extra twisted connected sum \cite{nordstrom2023extra}. In this gluing construction, two $\mathbb{T}^1\times \cat{ACylCY}_3$ are joint by twisting the exterior $\mathbb{T}^1$ with the one in the cross section of the $\cat{ACylCY}_3$. A large class of $\cat{ACylCY}_3$ were constructed in \cite{CHNP,Corti_2015} using (weak) Fano threefolds.

   \begin{defi}
        Let $(Z_\pm,\omega_\pm,\theta_\pm)\in\cat{ACyl}\cat{CY}_3$ with cross-section $\mathbb{T}^1\times \Sigma$ where $\Sigma\in\cat{HK}_1$. Assume that there exists an isometry
        \begin{align*}
            \mathfrak{r}\colon (\mathbb{T}^2\times \Sigma,g_{\mathbb{T}^2\times \Sigma})\rightarrow (\mathbb{T}^2\times \Sigma,g_{\mathbb{T}^2\times \Sigma})
        \end{align*}
        such that $(-\m{id}_{\mathbb{R}}\times \mathfrak{r})$ is a translational invariant automorphism of
        $(Cyl(\mathbb{T}^2\times \Sigma),\varphi_{Cyl})$
        where
        \begin{align*}
            \varphi_{Cyl}\coloneqq\m{d}s\wedge\m{d}\theta_1\wedge\m{d}\theta_2+  \m{d}s\wedge\omega_I+\m{d}\theta_1\wedge \omega_J+\m{d}\theta_2\wedge \omega_K,
        \end{align*}
        i.e. $[\mathfrak{r},\tau_s]=\m{id}_{Cyl(\mathbb{T}^2\times \Sigma)}$. The \textbf{extra twisted connected sum} of the building blocks  $(\mathbb{T}^1\times Z_i,\varphi_i\coloneqq\m{d}\theta_1\wedge\omega_i+\m{Re}(\theta_i))$ by $\mathfrak{r}$ is the unique, gauge fixed, torsion-free $G_2$-structure close to the preglued $G_2$-structure obtained gluing along $\mathfrak{r}$. 
    \end{defi}

The hyperkähler-twist $\mathfrak{r}$ restricts to an isometry of $\Sigma$ and acts on a basis of the maximally positive subspace $P\subset \m{H}^2(\Sigma,\mathbb{R})$ by
\begin{align*}
    \mathfrak{r}=\left(\begin{array}{ccc}
         \cos(\theta)&-\sin(\theta) &0 \\
         \sin(\theta)&\cos(\theta) &0\\
         0&0&-1
    \end{array}\right).
\end{align*}

    Let $(Z_\pm,\omega_\pm,\theta_\pm)\in\cat{ACylCY}_3$ and $\tau_\pm\colon (\mathbb{T}^1\times  Z_\pm,\m{d}\theta_2\wedge\omega_\pm+\m{Re}(\theta_\pm))\rightarrow (\mathbb{T}^1\times  Z_\pm,-\m{d}\theta_2\wedge\omega_\pm-\m{Re}(\theta_\pm))$, and let $\mathfrak{r} \colon\mathbb{T}^2\times\Sigma\rightarrow \mathbb{T}^2\times \Sigma$ be a hyperkähler rotation.
    Notice, that an isometry of an ACyl spaces is asymptotic to an isometry of the form 
    \begin{align*}
        \tau=(\tau_l,\tau_{\mathbb{T}^2\times \Sigma}),
    \end{align*}
    where $\tau_l$ denotes the translation by $s\mapsto s+l$. The extra twisted connected sum 
    \begin{align*}
        (Y_{\mathfrak{r};R},\varphi_{\mathfrak{r};R})
    \end{align*}
    admits an $G_2$-antiinvolution if 
    \begin{align*}
        \tau_-\circ\mathfrak{r}=\mathfrak{r}\circ\tau_+.
    \end{align*}

In the following, we will construct two $\cat{ACylG}_2$ manifolds with $G_2$-anti-involutions, whose induced translationally invariant antiinvolutions of the $G_2$-cy-linders are intertwined by the hyperkähler rotation $\mathfrak{r}$.

\begin{ex}
Let $Q$ and $S$ be a real quadric and quartic in $\mathbb{CP}^3$, and let $C = Q \cap S$ be their transverse intersection. Denote by 
\begin{align*}
    \sigma_{\mathbb{CP}^3}\colon [z_0:z_1:z_2:z_3] \mapsto [\overline{z_0} \colon\overline{z_1} \colon\overline{z_2} \colon\overline{z_3}].
\end{align*}
Let $\pi\colon  W \rightarrow \mathbb{CP}^3$ be the $S$-branched double cover given by 
\begin{align*}
    W=\{([z_0:z_1:z_2:z_3],w)\in\mathcal{O}_{\mathbb{CP}^3}(1)|S(z_0,z_1,z_2,z_3)=w^2\}
\end{align*}
Since $S$ is real, there exists an action of $\mathbb{Z}_2^2$ on $W$, generated by the holomorphic sheet-swapping involution
\begin{align*}
    \tau_W\colon  W \rightarrow W, ([z_0:z_1:z_2:z_3],w)\mapsto ([z_0:z_1:z_2:z_3],-w)
\end{align*}
and the antiholomorphic involution 
\begin{align*}
    \sigma_W\colon W\rightarrow W, ([z_0:z_1:z_2:z_3],w)\mapsto ([\overline{z_0} \colon\overline{z_1} \colon\overline{z_2} \colon\overline{z_3}],\overline{w})
\end{align*}
lifted from $\mathbb{CP}^3$. The fixpoints of $\sigma_W$ and $\tau_W$ are 
\begin{align*}
    Fix(\tau_W) = S, \quad   Fix(\sigma_{W})=W(\mathbb{R})\xrightarrow[]{\pi} \mathbb{RP}^3,\text{ branched over }S(\mathbb{R}).
\end{align*}
This $\mathbb{Z}^2_2$-action restricts to $\pi^{-1}(S)\cong S$ and hence, we can lift the $\mathbb{Z}_2^2$-action to the semi-Fano building block
\begin{align*}
    Z_+ \coloneqq \m{Bl}_{\pi^{-1}(C)}(W) \overset{\beta}{\dashrightarrow} W.
\end{align*}
Recall, that $\m{Bl}_{\pi^{-1}(C)}(W)=W\backslash \pi^{-1}(C) \sqcup \mathbb{P}(N (W/\pi^{-1}(C)))$, and as $\pi^{-1}C$ is preserved by the holomorphic $\mathbb{Z}_2^2$-action, the linearised action on $\pi^{-1}C$ restricts to a projectivised action on $\mathbb{P}(N(W/\pi^{-1}(C)))$. Now, since 
\begin{equation*}
    \begin{tikzcd}
        T\pi^{-1}(C)\arrow[r,hook]&TW|_{\pi^{-1}(C)}\arrow[r,two heads]& N (W/\pi^{-1}(C))\\
        N (\pi^{-1}(S)/\pi^{-1}(C))\arrow[r,hook]&N (W/\pi^{-1}(C))\arrow[r,two heads]& N (W/\pi^{-1}(S))|_{\pi^{-1}(C)}
    \end{tikzcd}
\end{equation*}
we see that\footnote{The linearised action of $\tau_{Z_+}$ on $\mathbb{P}(N(W/\pi^{-1}(C)))$ acts via $[n_0:n_1]\mapsto [n_0:-n_1]$ which can be deduced from the two exact sequences.}
\begin{align*}
    Fix(\sigma_{Z_+})=&Z_+(\mathbb{R})=W(\mathbb{R})\\
    Fix(\tau_{Z_+})=&\pi^{-1}(S)\backslash\pi^{-1}(C)\sqcup (\widehat{\pi^{-1}(C)}\xrightarrow[]{2:1}\pi^{-1}(C)).
\end{align*}
The proper transform of $\pi^{-1}(Q)$ 
\begin{align*}
    \Sigma \coloneqq \beta^{-1}(\pi^{-1}(Q)) \in \cat{HK}_1
\end{align*}
is a smooth, anticanonical $K3$ surface with a $\mathbb{Z}_2^2$-action satisfying 
\begin{align}
\label{sigmatauSigma}
    \sigma_{\Sigma}^*(\omega_I,\omega_J,\omega_K)=&(-\omega_I,\omega_J,-\omega_K)\\\nonumber
    \tau_{\Sigma}^*(\omega_I,\omega_J,\omega_K)=&(\omega_I,-\omega_J,-\omega_K).
\end{align}
Consider $\Sigma'\coloneqq (\Sigma, \omega_J, -\omega_I, -\omega_K)$ with the induced $\mathbb{Z}_2^2$-action. By \eqref{sigmatauSigma} $\sigma_{\Sigma'}$ is holomorphic. We define the orbifold
\begin{align*}
    Z_0 \coloneqq \left((\Sigma, \omega_J, \omega_I, -\omega_K) \times \mathbb{CP}^1\right)/(\sigma_{\Sigma} \times \sigma_{\mathbb{CP}^1}),
\end{align*}
where $\sigma_{\mathbb{CP}^1}\colon  [w_0:w_1] \mapsto [w_1:w_0]$. We can lift $\tau\colon  Z_0 \rightarrow Z_0$ to an anti-holomorphic involution acting on $\mathbb{CP}^1$ by $[\overline{w_0} \colon\overline{w_1}]$. By results of Kovalev and Lee \cite{kovalev2003twisted} and Nordström \cite{nordstrom2023extra} respectively, we obtain a building block
\begin{align*}
    (p_-\colon Z_- \rightarrow \mathbb{CP}^1) \overset{\beta_-}{\dashrightarrow} (p_0\colon Z_0 \rightarrow \mathbb{CP}^1)
\end{align*}
by either resolving or smoothing $Z_0$. This can be done in a $\tau$-equivariant way and hence, $(Z_-,\omega_-,\theta_-)$ admits an antiholomorphic involution $\tau_-$ such that the twisting map
\begin{equation*}
\begin{tikzcd}
	{\mathbb{R} \times \mathbb{T}^2 \times \Sigma} & {\mathbb{R} \times \mathbb{T}^2 \times \Sigma} \\
	{(s, \theta_1, \theta_2, m)} & {(-s, -\theta_2, \theta_1, m)}
	\arrow[from=1-1, to=1-2, "\mathfrak{r}"]
	\arrow[maps to, from=2-1, to=2-2]
\end{tikzcd}
\end{equation*}
intertwines $\tau_\pm$. The resulting equivariant twisted connected sum
\begin{align*}
    (Y_R, \varphi_R)
\end{align*}
is a family of compact $G_2$-manifolds, with a $G_2$-antiinvolution
\begin{align*}
    \tau_R\colon (Y_R, \varphi_R) \rightarrow (Y_R, -\varphi_R)
\end{align*}
fixing the family of coassociatives. For large $R \gg 0$, this family should admit a nonvanishing self-dual harmonic two-form, coming from matching, nonvanishing self-dual harmonic two forms on the two ends. These are induced by restricting the Kähler form onto $Fix(\tau_+)\subset Z_+$ and the pairing of the harmonic one form on $Fix(\tau_-)\subset Z_-$ with $\theta_2$.
\end{ex}

\subsection{Resolutions of Weighted Projective Spaces}
\label{Resolutions of Weighted Projective Spaces}

In \cite{Joyce1999}, Joyce constructed $\m{Spin}(7)$-orbifold resolutions of weighted projective spaces by gluing asymptotically conical $\m{Spin}(7)$-manifolds into the isolated conical singularities of $\mathbb{CP}^{[a_0,\dots,a_4]}$. \\

Using the adapted function spaces developed in this paper, we are able to refine Joyce’s analysis and improve the estimates for the distance between the torsion-free $\m{Spin}(7)$-structure and the preglued structure.

\begin{thm}
Let $X_{\zeta;t}\dashrightarrow \mathbb{CP}^{[a_0,...,a_4]}$ be one of the Joyce manifolds constructed in \cite{Joyce1999} and $\Phi_{\zeta;t}$ its torsion-free $\m{Spin}(7)$-structure. Let $\Phi^{pre}_{\zeta;t}$ be the Cayley form on $X_{\zeta;t}$ constructed from the pregluing process. Then 
\begin{align*}
    \left|\left|\Phi_{\zeta;t}-\Phi^{pre}_{\zeta;t}\right|\right|_{\mathfrak{D}^{1,\alpha}_{\beta;t}}\lesssim t^{\approx 4}.
\end{align*}
\end{thm}

\appendix

\section{Orbifolds and Fibration Conventions}
\label{Orbifolds and Fibration Conventions}

In the following section we will briefly discuss Riemannian orbifolds by following the work of Moerdijk in \cite{moerdijk2002orbifolds}. In this approach, the category of orbifolds is associated with the category of proper foliation groupoids up to Morita equivalence. By representing an orbifold by a Lie groupoid, the notions of tangent bundle and normal bundle can be understood in a more natural way than by the usual chart based approach.

\subsection{Orbifolds}
\label{Orbifolds}

In this section, we recall the groupoid-theoretic definition of orbifolds and their morphisms, following the modern approach via proper étale Lie groupoids and Morita equivalence.

\begin{defi}
Let $X^\bullet$ be a Lie groupoid, i.e. a groupoid object \footnote{A category in which every morphism is an isomorphism.} in $\cat{C}^\infty\cat{Mfd}$. We will denote by
\begin{align*}
    X^0\coloneqq \m{Obj}(X^\bullet)\und{1.0cm}X^1\coloneqq \m{Mor}(X^\bullet).
\end{align*}

\begin{itemize}
    \item $X^\bullet$ is called a \textbf{proper Lie groupoid} if $(s,t)\colon X^1\rightarrow X^0\times X^0$ is a proper map. In particular, every isotropy group is a compact Lie group.
    \item $X^\bullet$ is called a \textbf{foliation groupoid} if every isotropy group is discrete.
    \item $X^\bullet$ is called \textbf{étale} if $s$ and $t$ are local diffeomorphisms.
\end{itemize}
Let $X^\bullet$ be an étale Lie groupoid. Then every arrow $\phi\colon x\to y$ in $X^\bullet$ induces a germ of a diffeomorphism $\tilde{\phi}\colon (U_x,x)\rightarrow (V_y,y)$ as $\tilde{\phi}=t\circ \hat{\phi}$, where $\hat{\phi}\colon U_X\rightarrow X^1$ is an $s\colon X^1\rightarrow X^0$ section, defined on a sufficiently small neighbourhood $U_x$ of $x$. 
\begin{itemize}
    \item $X^\bullet$ is called \textbf{faithfull}, if for each point $x\in X^0$ the homomorphism $G_x\rightarrow \m{Diff}_x(X^0)$ is injective.
\end{itemize}
\end{defi}

\begin{defi}
Let $X^\bullet$ be a Lie groupoid. A \textbf{left (right) $X^\bullet$-space} is a smooth manifold $E$ equipped with an action of $X^\bullet$. Such an action is given by the smooth morphisms  
\begin{align*}
    \pi\colon E\rightarrow X^0\hspace{1.0cm}\mu_L\colon X^1\times_{s,\pi}E\rightarrow E\hspace{1.0cm}(\mu_R\colon E\times_{\pi;t}X^0\rightarrow E)
\end{align*}
satisfying $\pi(\phi.e)=s(\phi)$, $1_x.e=e$ and $\eta.(\phi.e)=(\eta\phi).e$. For each $X^\bullet$ space $E$, we define the translation groupoid $( X\ltimes E)^\bullet$ whose objects are points in $E$ and whose morphisms $\phi\colon e\rightarrow e'$ are morphisms $\phi\colon \pi(e)\rightarrow\pi(e')$ in $X^1$ with $\phi.e=e'$. There exists a homomorphism $\pi_E\colon (X\ltimes E)^\bullet\rightarrow X^\bullet$ of orbifolds. 
\end{defi}

\begin{rem}
By replacing the left by right actions in the above definition we can define right $X^\bullet$-spaces.
\end{rem}

\begin{rem}
Notice that on the level of groupoids, the fibre over $x\in X^0$ is $\pi^{-1}(x)$; at the level of orbit spaces $|X\ltimes E|\rightarrow |X|$ the fibres are $\pi^{-1}(x)/\m{Isot}_x$.
\end{rem}

\begin{defi}
A \textbf{right/left-$X^\bullet$-$Y^\bullet$-bibundle} $Z$ is given by smooth maps $\alpha_X\colon Z\rightarrow X^0$ and $\alpha_Y\colon Z\rightarrow Y^0$ called the moment maps such that both $\alpha_X$ is a right/left $X^\bullet$-space and $\alpha_Y$ is a right/left $Y^\bullet$-space, i.e.
\begin{equation*}
\begin{tikzcd}
	{X^1} && Z && {Y^1} \\
	\\
	{X^0} &&&& {Y^0}
	\arrow["t"{description}, shift left=2, from=1-1, to=3-1]
	\arrow["s"{description}, shift right=2, from=1-1, to=3-1]
	\arrow["{\alpha_X}"{description}, from=1-3, to=3-1]
	\arrow["{\alpha_Y}"{description}, from=1-3, to=3-5]
	\arrow["t"{description}, shift left=2, from=1-5, to=3-5]
	\arrow["s"{description}, shift right=2, from=1-5, to=3-5]
\end{tikzcd}
\end{equation*}
\end{defi}

\begin{defi}
A \textbf{generalised morphism} 
\begin{align*}
    [Z]\colon X^\bullet\rightsquigarrow Y^\bullet
\end{align*}
is given by an isomorphism class of right $X^\bullet$-$Y^\bullet$-bibundles.
\end{defi}

\begin{defi}
A Lie groupoid $X^\bullet$ is \textbf{Morita equivalent} to Lie groupoid $Y^\bullet$ if there exists a generalised morphism 
\begin{align*}
    [Z]\colon X^\bullet\rightsquigarrow Y^\bullet
\end{align*}
which is also a left $X^\bullet$-$Y^\bullet$-bibundle.
\end{defi}

\begin{defi}
A \textbf{smooth orbifold} $X$ is the Morita equivalence class of a smooth, proper, foliation Lie groupoid $X^\bullet$. Given such an equivalence class, we will write  
\begin{align*}
    [X^0/X^1]\coloneqq[X^\bullet].
\end{align*}
A \textbf{morphism of orbifolds} is given by an equivalence class of generalised morphism 
\begin{align*}
    [[Z]]\colon [X^\bullet]\rightsquigarrow[Y^\bullet].
\end{align*}
\end{defi}

\begin{rem}
In every equivalence class of an orbifold $[X^\bullet]\subset\cat{LieGrpd}$ there exists a proper, étale groupoid representing the orbifold. 
\end{rem}

\begin{rem}
Notice that the orbit space/coarse moduli space $X^0/X^1$ just carries the data of a topological space. The stacky quotient $[X^0/X^1]$ \say{remembers} the isotropy.
\end{rem}

\begin{defi}
The \textbf{tangent bundle} of an orbifold $[X^\bullet]$ is defined by Morita equivalence class of the $X^\bullet$-vector bundle 
\begin{align*}
    \pi_X\colon TX^0 \rightarrow X^0,\quad \mu_{TX}\colon X^1 \times_{s, \pi_{TX^0}} TX^0 \rightarrow TX^0,
\end{align*}
where $\mu(g, V_x) = T g(V_x)$ for $g\colon  x \to y$ in $X^1$. 
\end{defi}

If $S^\bullet\subset X^\bullet$ is subgroupoid such that all isotropy groups are conjugate to a finite group $\Gamma$, then the normal bundle $NS^\bullet \rightarrow S^\bullet$ is naturally an $S^\bullet$-vector bundle, whereas the normal bundle $N|S| \to |S|$ is a (quasi-)conical bundle on the coarse moduli space.

\begin{defi}
The \textbf{frame bundle} of an $n$-dimensional orbifold $[X^\bullet]$ is defined to be the Morita equivalence class of the $X^\bullet$-principal $\mathrm{GL}(n)$-bundle given by 
\begin{align*}
    f_X\colon  \m{Fr}(X^0) \rightarrow X^0,\quad 
    \mu_{\m{Fr}(X)}\colon  X^1 \times_{s, \pi_{X^0}} \m{Fr}(X^0) \rightarrow \m{Fr}(X^0),
\end{align*}
where $\mu_{\m{Fr}(X)}(\psi, f_x) = T\psi \circ f_x$ for $\psi\colon  x \to y$ in $X^1$.
\end{defi}

The following corollary can be seen as the traditional definition of orbifolds via orbifold charts, i.e. $X$ being a topological space covered by charts of the form $(\m{Isot}_x\ltimes U_x)^\bullet$ subject to gluing conditions.

\begin{cor}\cite[Sec. 3.4]{moerdijk2002orbifolds}
Let $X$ be a smooth orbifold presented by a proper étale Lie groupoid $X^\bullet$. For every open subset $U\subset X^0$ we write $(X|_U)^\bullet$ for the full subgroupid whose objects are $U$. Since $X^1$ is assumed to be proper and étale, for $x\in X^0$ there exists an arbitrarily small neighbourhood $U_x$ such that $(X|_{U_x})^\bullet$ is isomorphic to the action groupoid $(\m{Isot}_x\ltimes U_x)^\bullet$. 
\end{cor}

\begin{defi}
    The \textbf{dimension of a smooth orbifold} is defined as $ \m{dim}(X)\coloneqq \m{dim}(U_x)$ which is independent of the point $x$.
\end{defi}

Orbifolds are examples of stratified spaces. The stratification of $X$ consists of the suborbifolds $S\subset X$ for which $\m{Isot}_s\cong\m{Isot}_{s'}$ for all $s,s'\in S$ and
\begin{align*}
    X^1/X^0=\bigcup_{[ \Gamma_\alpha]}X_\alpha\in\cat{Top}
\end{align*}
where the union runs over all conjugacy classes $[ \Gamma_\alpha]$ of finite subgroups of $\m{GL}(n)$. Notice, that every $X_\alpha$ is a smooth (open) manifold.\\ 

We define a partial order on the set $\{X_\alpha\}$ by 
\begin{align*}
    X_\alpha\leq  X_\beta
\end{align*}
if $ \Gamma_\beta\subset  \Gamma_\alpha$. Notice that since  $\{X_\alpha\}$ defines a stratification, $X_\beta\leq X_\alpha$ implies $X_\beta\subset \overline{X_\alpha}$. Hence, we can define a prestratification of $X$ by $X_{\tilde{\alpha}}=\{\bigcup_{\beta\leq\alpha}X_\beta\}_\alpha$. In particular, we say that two connected singular strata $X_\alpha$ and $X_{\alpha'}$ intersect in a singular stratum $X_\beta$, if $X_\beta\leq X_\alpha$ and $X_\beta\leq X_{\alpha'}$ or equivalently 
\begin{align*}
    X_{\tilde{\alpha}}\cap X_{\tilde{\alpha}'}=X_{\beta}.
\end{align*}
Notice, that $\overline{X_\alpha}=X_\alpha\sqcup \bigsqcup_{\beta\leq\alpha}X_\beta$. 

\begin{rem}
    The top stratum of an $n$-dimensional orbifold $X$, i.e. the open dense subset consisting of points with trivial isotropy will be denoted by $X^{reg}$. If the isotropy group $\Gamma_{reg} = 1$, we call the orbifold effective. For the remainder of this paper we assume all orbifolds to be effective in order to avoid additional complications.
\end{rem}

The following definitions quantify how severe or complex a given singularity can be. The depth measures the position of a stratum within the nested hierarchy of singularities, while the singularity type describes the local model and isotropy representation near the stratum. Together, these notions capture both the local complexity and the global stratified structure of the singular set.

\begin{defi}
    We define the \textbf{depth} of $X_{\alpha}$, as the length of the longest chain of inclusions of singular strata, i.e. 
    \begin{align*}
        \m{depth}(X_\alpha)=\max\{i\in \mathbb{N}|X_{\alpha}= X_{\beta_0}\leq...\leq X_{\beta_i}= X^{reg}|X_{\beta_k}\neq X_{\beta_{k-1}}\}.
    \end{align*}
\end{defi}

\begin{defi}
Let $(X,g)$ be a Riemannian orbifold and let $X^{sing}$ denote the union of its singular strata. Let $S$ be a stratum whose isotropy group is conjugate to a finite group $\Gamma\subset\m{O}(n)$. We say
\begin{itemize}
    \item $X$ is of \textbf{singularity type A} at $S$, if $X$ at $S\subset X$ is locally modelled on
    \begin{align*}
    \mathbb{R}^{n-m}\times\mathbb{R}^m/\Gamma
    \end{align*}
    such that $\Gamma$ acts freely on $\mathbb{R}^m\backslash\{0\}$.
    \item $X$ is of \textbf{singularity type B} at $S$, if $X$ at $S\subset X$ is locally modelled on
    \begin{align*}
    \mathbb{R}^{n-\sum_im_i}\times\prod_i\mathbb{R}^{m_i}/ \Gamma_i
    \end{align*}
    such that $ \Gamma_i$ acts freely on $\mathbb{R}^{m_i}\backslash\{0\}$.
    \item $X$ is of \textbf{singularity type C} at $S$, if it is neither type A nor type B
\end{itemize}
\end{defi}

\begin{rem}
Let $S_1$ and $S_2$ be two singular strata of isotropy $ \Gamma_1$ and $ \Gamma_2$. Then the intersection $S_3=\overline{S_1}\cap\overline{S_2}$ is a singular strata of isotropy $ \Gamma_1 \Gamma_2\subset \Gamma_3$ and $X$ is singular of type B or C.
\end{rem}

\begin{nota}
    Throughout this paper we will denote by 
    \begin{align*}
        \m{N}_G(H)\und{1.0cm} \m{C}_G(H)
    \end{align*}
    the normaliser and the commutator subgroup of $H$ in $G$.
\end{nota}

\begin{prop}
Let $S\subset X$ be a singular strata of isotropy type $\Gamma$. Then there exists a $\m{N}_{\m{GL}(n)}(\Gamma)\hookrightarrow \m{GL}(n)$ reduction of the restriction of the frame bundle $Fr_{X^\bullet}$ to $S$.
\end{prop}

\begin{proof}
Let $X^\bullet$ be a Lie groupoid representing the orbifold $[X^0/X^1]$. Let $S^\bullet$ denote the sub Lie groupoid corresponding to $S\subset [X^0/X^1]$ in the quotient space. Notice that the $X^\bullet$-frame bundle restricts to a $\m{GL}(n)$-principal bundle of $S^\bullet$, i.e. 
\begin{align*}
    f_X|_S\colon Fr_{X^0}|_{S^0}\rightarrow S^0\hspace{1.0cm}\mu_{Fr_X|_S}\colon S^1\times_{s,\phi_{X^0}|_{S^0}}Fr_{X^0}|_{S^0}\rightarrow Fr_{X^0}|_{S^0}.
\end{align*}
There exists a short exact sequence of $S^\bullet$-vector bundles 
\begin{equation*}
    \begin{tikzcd}
        TS^\bullet\arrow[r,"T\iota",hook]&TX^\bullet|_S\arrow[r,"p",two heads]&NS^\bullet
    \end{tikzcd}
\end{equation*}
In particular, there exists a natural reduction to frames 
\begin{align*}
    f_x\colon H\oplus V\rightarrow TX^0
\end{align*}
such that $p\circ f_x|_V\in \m{Fr}(NS^\bullet)$. A moment's thought reveals that this reduction can be represented by smooth principal-$\m{N}_{\m{GL}(n)}({\Gamma})$-bundle 
\begin{align*}
    \m{Fr}(X/S)\rightarrow S
\end{align*}
\end{proof}

\subsection{$G$-Structures}
\label{G Structures}

We begin by recalling the general framework of $G\subset\m{GL}(n)$-structures orbifolds. These are constructed via reductions of the frame bundle of the orbifold. 
 
\begin{defi}
Let $X$ be an orbifold and $\m{Fr}(X)$ its $\m{GL}(n)$-frame bundle. The \textbf{bundle of $G$-structures} on $X$ is the homogeneous bundle
    \begin{align*}
        \gamma\colon \m{Str}_G(X)\coloneqq\m{Fr}(X)/G\rightarrow X.
    \end{align*}
    A $G$-structure on an orbifold is given by a section $\Phi\in \Gamma(X,\m{Str}_G(X))$. Such a section induces a $G$-reduction of the frame bundle by 
    \begin{align*}
        \m{Fr}(X,\Phi)\coloneqq \Phi^*\m{Fr}(X)\hookrightarrow \m{Fr}(X)
    \end{align*}
    pulling back the $G$-bundle $\phi_{/G}\colon \m{Fr}(X)\rightarrow \m{Str}_G(X)$. 
\end{defi}

\begin{lem}
    Given a nested subgroup $G\subset H\subset \m{GL}(W)$, there exists a natural projection map 
    \begin{align*}
        \m{Str}_G(X)\twoheadrightarrow \m{Str}_H (X),
    \end{align*}
    whose fibres are isomorphic to the homogeneous space $H/G$.
\end{lem}

The existence of a $G$-structure on an orbifold can be obstructed by the nature of its isotropy groups. The following proposition clarifies the necessary condition.

\begin{prop}
    Let $X$ be an orbifold modelled on finite quotients of a finite dimensional vector space $W$ and let $\phi\colon \m{Fr}(X)\rightarrow X$ denote its frame bundle. If $\m{Fr}(X)$ admits a reduction to a $G$-structure, then all isotropy groups are conjugate to subgroups of $G$.
\end{prop}

\begin{proof}
    The isotropy group of a point $x\in X$ can be identified with a conjugacy class of a finite subgroup $\m{Isot}(x)\subset \m{GL}(W)$. The frame bundle of $X$ at $x$ naturally reduces to frames  
    \begin{align*}
        f\colon W\rightarrow T_xX
    \end{align*}
    that normalizes the linearised action of the isotropy group, i.e. its representation in $\m{GL}(T_xX)$. These frames form a $\m{N}_{\m{GL}(T_xX)}(\m{Isot}(x))$ reduction of the frame bundle along the stratum containing $x$. If $X$ contains strata whose isotropy is not conjugate to a subgroup of $G$, the reduction $\m{N}_{\m{GL}(T_xX)}(\m{Isot}(x))$ can not be a $G$ reduction. Hence, $\m{Fr}(X)$ can not be reduced to a sub $G$-bundle.
\end{proof}

\begin{ex}
    Let $(W,\mathfrak{o})$ be an oriented vector space. An orientation on an orbifold, locally modelled by finite quotients of $W$, is oriented if it admits a section of the orientation bundle 
    \begin{align*}
        \mathfrak{o}(X)\coloneqq \m{Fr}(X)/\m{GL}_+(W)\cong (\wedge^{top}T^*X)/\mathbb{R}_{>0}.
    \end{align*}
\end{ex}

\begin{ex}
Let $(W,g_0)$ be an Euclidean vector space. The most important $G$-structure on an orbifold $X$ locally modelled on $W$ in Riemannian geometry is an $\m{O}(W,g_0)$-structure, i.e. a Riemannian orbifold structure. A Riemannian metric on $X$ corresponds to a section of the homogeneous bundle 
\begin{align*}
    \m{Met}(X)\coloneqq \m{Fr}(X)/\m{O}(W,g_0)\rightarrow X
\end{align*}
with fibres being 
\begin{align*}
    \m{Met}(W)\coloneqq \m{GL}(W)/\m{O}(W,g_0).
\end{align*}
Given a section $g\in \Gamma(X,\m{Met}(X))$ we obtain a reduction of the framebundle 
\begin{align*}
    \m{Fr}(X,g)\coloneqq g^*\m{Fr}(X)\hookrightarrow \m{Fr}(X)
\end{align*}
by pulling back the $\m{O}(W,g_0)$-bundle $\phi_{/\m{O}(W,g_0)}\colon \m{Fr}(X)\rightarrow \m{Met}(X)$.
\end{ex}

We will now define the Clifford bundle of $X$. This is a finite dimensional orbifold bundle, such that given a section $g\colon X\rightarrow \m{Met}(X)$, the pullback of the Clifford bundle corresponds to the Clifford bundle of $X$ with respect to $g$. This notion will be needed in the discussion of families of Dirac bundles later in this paper.

\begin{defi}
    We define the \textbf{Clifford bundle} of $X$ to be the associated bundle 
    \begin{align*}
        \m{Cl}(X)\coloneqq \m{Fr}(X)\times_{\m{O}(W,g_0)}\m{Cl}(W^*,g_0)\rightarrow \m{Met}(X).
    \end{align*}
    Given a metric on $X$, i.e. a section $g\colon X\rightarrow \m{Met}(X)$, we have
    \begin{align*}
        g^*\m{Cl}(X)=\m{Cl}(T^*X,g).
    \end{align*}
\end{defi}

\subsection{Fibrations}
\label{Fibrations}
Throughout this paper we will work with geometric structures on fibrations. The following section is devoted to the discussion of such structures from the general perspective of $G$-structures as introduced in the previous sections.\\

Let $M$ be an orbifold locally modelled on a vector space $W=H\oplus V$ and $\pi\colon M\rightarrow B$ a fibration. We define the group 
\begin{align*}
    \m{GL}(W\to H)\coloneqq \{A\in\m{GL}(H\oplus V)|A(V)=V,\m{pr}_H\circ A(H)=H\}.
\end{align*}
There exists a natural $\m{GL}(W\to H)$-reduction of the frame bundle 
\begin{align*}
    \m{Fr}(\pi\colon M\rightarrow B)\hookrightarrow \m{Fr}(X)
\end{align*}
of frames
\begin{align*}
    \m{Fr}(\pi\colon M\rightarrow B)\coloneqq\{f_1\colon H\oplus V \rightarrow T_{m}M|f_1|_{V}\colon V\cong V_m\pi,\, T\pi\circ f_1|_H\in \m{Fr}(B)_{\pi(m)}\}. 
\end{align*}

We will now be interested in reductions of such fibered structures and their torsions. In particular, this will identify the torsion of such structures with geometric quantities. In general, we will set 
\begin{align*}
    G(W\to H)\coloneqq G\cap \m{GL}(W\to H).
\end{align*}
Notice that this group depends on the embedding of $G\subset \m{GL}(W)$. We will now investigate fibred Riemannian structures. 

\begin{defi}
    A \textbf{fibred Riemannian structure} is given by a section of 
    \begin{align*}
        \m{Fr}(\pi\colon M\rightarrow B)/O(W\to H)\rightarrow M.
    \end{align*}
    Here $\m{O}(W\to H)=O(W)\cap \m{GL}(W\to H)\cong \m{O}(H)\times \m{O}(V)$. 
\end{defi}

The unique torsion-free Levi-Civita connection $\varphi_g$ on $\m{Fr}(M,g)$ pulls back to $\m{Fr}(\pi\colon M\to B,g)$ and decomposes into  
    \begin{align*}
        \varphi_g=\varphi_\oplus+T_{\oplus}.
    \end{align*}
We will refer to $T_{\oplus}$ as the torsion of the fibred Riemannian structure. This tensor is valued in
    \begin{align*}
        T_\oplus\in \Omega^0(\m{Fr}(\pi\colon M\rightarrow B),\wedge^2 H^*\otimes V\oplus \m{Sym}^2 V^*\otimes H)^{\m{O}(W\to H)}.
    \end{align*}

We will now identify the tensor $T_{\oplus}$ with geometric quantities of the Riemannian fibration. These results will be used throughout this paper to understand the behaviour of Dirac operators on collapsing fibrations and to construct resolutions of tubular neighbourhoods of singular strata of $\m{Spin}(7)$-orbifolds. The broader discussion is based on the work of Gromoll and Walshap in \cite{gromoll2009metric} and Berline, Getzler and Vergne in \cite{berline1992heat}.\\

The connection $\varphi_{\oplus}$ induces a covariant derivative that can be identified with
\begin{align*}
    \nabla^{\oplus}=\pi^*\nabla^{g_S}\oplus\nabla^{g_{V}},
\end{align*}
i.e. the direct sum connection of the vertical connection and the lift of the Levi-Civita connection on the base space through the Ehresmann connection $H=VM^\perp$. Let in the following 
\begin{align*}
    X,Y,Z\in\mathfrak{X}^{1,0}(M)\und{1.0cm}U,V,W\in\mathfrak{X}^{0,1}(M)
\end{align*}
and denote by $\nabla$ the Levi-Civita connection of $g$.

\begin{defi}
We define the two tensor $A\in\Omega^{1,0}(M,\m{Hom}(H,VM))$ by
\begin{align*}
    A(X)Y=\nabla^{0,1}_XY=\frac{1}{2}[X,Y]^{0,1}
\end{align*}
and
    \begin{align*}
    S\in\Omega^{1,0}(M,\m{End}(VM)) \quad \text{by} \quad S(X)U=-(\nabla_UX)^{0,1}.
\end{align*}

\end{defi}

The second fundamental form of the fibre can thus be identified with the tensor
\begin{align*}
    \m{II}(U,V)=\left<S U,V\right>^\flat.
\end{align*}
By taking the vertical trace we obtain the fibrewise mean curvature, i.e. the one form 
\begin{align*}
    k=\m{tr}_{g;V}\m{II}.
\end{align*}

\begin{rem}
For a Riemannian fibration, the fibrewise second fundamental form $\m{II}$ or equivalently $S$ vanishes if and only if the fibres are totally geodesic. Furthermore, the tensor $A$ vanishes if and only if the Ehresmann connection $H$ is flat.
\end{rem}

\begin{prop}\cite[Prop. 10.6]{berline1992heat}
Given $X,Y,Z\in\mathfrak{X}(M)$ the torsion tensor $T_{\oplus}\in\Omega^1(M,\wedge^2T^* M)$ can be expressed by 
\begin{align*}
    T_{\oplus}(X)(Y,Z)=&g(\m{II}(X,Z),Y)-g(\m{II}(X,Y),Z)\\
    &+\tfrac{1}{2}\left(g(F_{H}(X,Z),Y)-g(F_{H}(X,Y),Z)\right.\\
    &+\left.g(F_{H}(Y,Z),X)\right).
\end{align*}
\end{prop}

\begin{rem}\cite[Prop. 10.1]{berline1992heat}
The de Rham differential on $M$ decomposes with respect to $H$ and the Riemannian structure $g=p^*g_S\oplus g_{V}$ into 
\begin{align*}
    \m{d}=&\m{d}^{1,0}+\m{d}^{0,1}+\m{d}^{2,-1}\\
    =&\m{d}_{\nabla^{g_{V}}}-\hat{\iota}_{\m{II}}+\m{d}^{0,1}+\hat{\iota}_{F_H}.
\end{align*}
In particular, let $\m{vol}_{V}$ denote the vertical volume form with respect to $g_{V}$. Then 
\begin{align*}
    \m{d}k(X,Y)=-2\m{div}(A(X)Y)\und{1.0cm}\m{d}\m{vol}_{V}=k\wedge\m{vol}_{V}+\iota_{F_{H}}\m{vol}_{V}.
\end{align*}
\end{rem}

\newpage
\bibliographystyle{alpha}
\bibliography{literature}

\end{document}